\documentclass[a4paper,11pt,reqno]{amsart}

\usepackage{amsmath,amssymb,amsthm,mathtools,mathrsfs}
\usepackage{microtype}
\usepackage{enumitem}
\usepackage{booktabs,array,tabularx}
\usepackage{tikz-cd}
\usepackage[all]{xy}
\usepackage{xcolor}
\usepackage{aliascnt}
\usepackage[hypertexnames=false]{hyperref}
\usepackage[nameinlink,noabbrev]{cleveref}

\setlist{itemsep=2pt,topsep=4pt}
\allowdisplaybreaks
\definecolor{TR4link}{RGB}{24,70,125}
\definecolor{TR4cite}{RGB}{32,100,70}
\hypersetup{
  colorlinks=true,
  linkcolor=TR4link,
  citecolor=TR4cite,
  urlcolor=TR4link,
  pdftitle={From Pre-triangulation to Triangulation: Obstructions and Exact Lifting},
  pdfauthor={Ping He and Bo Le}
}

\newtheorem{theorem}{Theorem}[section]
\newaliascnt{proposition}{theorem}
\newtheorem{proposition}[proposition]{Proposition}
\aliascntresetthe{proposition}
\newaliascnt{lemma}{theorem}
\newtheorem{lemma}[lemma]{Lemma}
\aliascntresetthe{lemma}
\newaliascnt{corollary}{theorem}
\newtheorem{corollary}[corollary]{Corollary}
\aliascntresetthe{corollary}
\theoremstyle{definition}
\newaliascnt{definition}{theorem}
\newtheorem{definition}[definition]{Definition}
\aliascntresetthe{definition}
\newaliascnt{notation}{theorem}
\newtheorem{notation}[notation]{Notation}
\aliascntresetthe{notation}
\newtheorem*{notation*}{Notation}
\newaliascnt{fact}{theorem}
\newtheorem{fact}[fact]{Fact}
\aliascntresetthe{fact}
\newaliascnt{convention}{theorem}

\aliascntresetthe{convention}
\newaliascnt{remark}{theorem}
\newtheorem{remark}[remark]{Remark}
\aliascntresetthe{remark}
\newtheorem*{remark*}{Remark}

\crefname{section}{Section}{Sections}
\Crefname{section}{Section}{Sections}
\crefname{subsection}{Section}{Sections}
\Crefname{subsection}{Section}{Sections}
\crefname{theorem}{Theorem}{Theorems}
\Crefname{theorem}{Theorem}{Theorems}
\crefname{proposition}{Proposition}{Propositions}
\Crefname{proposition}{Proposition}{Propositions}
\crefname{lemma}{Lemma}{Lemmas}
\Crefname{lemma}{Lemma}{Lemmas}
\crefname{corollary}{Corollary}{Corollaries}
\Crefname{corollary}{Corollary}{Corollaries}
\crefname{definition}{Definition}{Definitions}
\Crefname{definition}{Definition}{Definitions}
\crefname{notation}{Notation}{Notations}
\Crefname{notation}{Notation}{Notations}
\crefname{fact}{Fact}{Facts}
\Crefname{fact}{Fact}{Facts}
\crefname{convention}{Convention}{Conventions}
\Crefname{convention}{Convention}{Conventions}
\crefname{remark}{Remark}{Remarks}
\Crefname{remark}{Remark}{Remarks}
\theoremstyle{plain}
\newtheorem*{maintheorema}{Theorem A}
\newtheorem*{maintheoremb}{Theorem B}
\newtheorem*{maintheoremc}{Theorem C}
\newtheorem*{maintheoremd}{Theorem D}
\newtheorem*{maintheoreme}{Theorem E}

\numberwithin{equation}{section}

\newcommand{\T}{\mathcal T}
\newcommand{\C}{\mathcal C}
\newcommand{\D}{\mathcal D}

\newcommand{\coker}{\operatorname{coker}}
\newcommand{\im}{\operatorname{im}}
\newcommand{\Hom}{\operatorname{Hom}}
\newcommand{\End}{\operatorname{End}}
\newcommand{\Ext}{\operatorname{Ext}}
\newcommand{\Aut}{\operatorname{Aut}}
\newcommand{\xto}[1]{\xrightarrow{#1}}
\newcommand{\kk}{\Bbbk}
\DeclareRobustCommand{\correspondingmark}{%
	\textcolor{blue}{\(\dagger\)}%
}

\title[From Pre-triangulation to Triangulation]
{From Pre-triangulation to Triangulation: Obstructions and Exact Lifting}
\author[Ping He]{Ping He}
\address{Beijing Institute of Mathematical Sciences and Applications,
	Beijing 101408, People's Republic of China}
\email{pinghe@bimsa.cn}

\author[Bo Le]{Bo Le \textsuperscript{\correspondingmark}}
\address{Department of Mathematical Sciences, Tsinghua University,
Beijing 100084, People's Republic of China}
\email{bole\_math@163.com}

\thanks{\correspondingmark  Corresponding author.}
\date{}


\begin{document}

\begin{abstract}
We give a new equivalent formulation of Verdier's octahedral axiom.  An
initial square in a pre-triangulated category determines an obstruction in a
quotient of ordinary Hom groups.  Its vanishing is equivalent to the
existence of a good completion, and $\mathrm{TR4}$ is equivalent to this
vanishing for the squares associated with composable morphisms.  As a
consequence, if two
pre-triangulations have the same underlying additive category and one is
triangulated, then the other satisfies $\mathrm{TR4}$ precisely when their
relative inverse Heller comparison admits an exact lift along every short
exact sequence in the Freyd category.

As applications, we prove that the exotic pre-triangulated category of
D\'iaz Cabrera--Muro is in fact a triangulated category over every
algebraically closed field.
We construct scalar families for type-$A_{3a-1}$ preprojective algebras over arbitrary fields
and for a range of Dynkin preprojective algebras in characteristic two; in each family
the zero parameter is the unique triangulated member.  The local equations also yield separable
descent, including canonical triangulations for non-modular
equivariantizations and for the Markman--Mehrotra K3 deformation categories.
\end{abstract}

\maketitle
\section{Introduction}
\label{sec:introduction}

Verdier's octahedral axiom controls the compatibility of cones with
composition.  The first three axioms provide cones, rotations, and
morphisms of triangles, but do not require these choices to be compatible
for two composable morphisms.  The independence problem was recorded in
\cite[Conjecture~9.9]{Beligiannis2000} and
\cite[Remark~1.3.15]{Neeman2001}.  The independence of $\mathrm{TR4}$ from
$\mathrm{TR1}$--$\mathrm{TR3}$ was proved recently by
Chen--Liu--Lu--Zhang \cite{ChenLiuLuZhang2026}.  An independent
computer-assisted example in characteristic two appears in
\cite{AnonymousD4}.  These results show that the missing compatibility is
genuine.  The purpose of this paper is to identify it explicitly and turn
it into a criterion that can be used both to prove and to disprove the
octahedral axiom.

Throughout, a \emph{pre-triangulated category} means a category satisfying
$\mathrm{TR1}$--$\mathrm{TR3}$ in the sense of Puppe.  It is
\emph{triangulated} if it also satisfies $\mathrm{TR4}$; see
\cref{def:pretriangulation}.  Let $\T$ be a
pre-triangulated category with suspension $\Sigma$ and distinguished
triangles $\Delta$, and put
\[
 h_X=\T(-,X),\qquad
 \mathcal E=\mathsf{mod}\,\T.
\]
Here $\mathsf{mod}\,\T$ is the Freyd category of finitely presented
additive functors $\T^{\mathrm{op}}\to\mathsf{Ab}$; see
\cref{not:freyd}.  It is Frobenius by \cref{prop:freyd-category}.
A distinguished triangle beginning with a morphism $c$ determines a
characteristic extension class (see
\cref{def:characteristic-class-pretriangulation})
\[
 \chi_\Delta(c)\in
 \Ext^1_{\mathcal E}(L(c),K(c)),
 \qquad K(c)=\coker h_c,\quad L(c)=\ker h_{\Sigma(c)}.
\]
A candidate triangle is Yoneda exact if its associated sequence of
contravariant representable functors is exact; see
\cref{def:yoneda-exact}.  By \cref{prop:characteristic-extension}, such a
triangle beginning with $c$ is distinguished precisely when it represents
$\chi_\Delta(c)$.

Now let $q$ be a commutative square between the first maps of two chosen
distinguished triangles.  We call these data an \emph{initial square}; see
\cref{def:initial-square}.  Following Neeman
\cite[Definition~1.3.14]{Neeman2001}, a completion of $q$ is called
\emph{good} if the mapping cone of the resulting morphism of triangles is
distinguished.  Completing the first map of this cone candidate triangle
to a distinguished triangle reduces goodness to two identities for a
single morphism; see \cref{lem:two-arrow-recognition}.  These identities
define the local octahedral obstruction
\[
 o_\Delta(q)\in\operatorname{Obs}_\Delta(q),
\]
where $\operatorname{Obs}_\Delta(q)$ is a quotient of ordinary Hom groups;
see \cref{def:octahedral-defect}.  Its vanishing is independent of the
auxiliary distinguished triangle used in its construction, by
\cref{prop:local-choice-independence}.
For a composable pair $X\xrightarrow{f}Y\xrightarrow{g}Z$, let
$q(g\circ f)$ denote the initial square formed by distinguished
completions of $f$ and $g\circ f$, with vertical maps $1_X$ and $g$;
see \eqref{eq:composition-square}.

\begin{maintheorema}[{\normalfont see~\Cref{thm:boundary-lifting,thm:composition-obstruction}}]
Let $q$ be an initial square.  Then $q$ has a good completion if and only if
$o_\Delta(q)=0$.  Moreover, the following are equivalent.
\begin{enumerate}[label=\textup{(\roman*)}]
\item $\Delta$ satisfies $\mathrm{TR4}$.
\item One has $o_\Delta(q)=0$ for every initial square $q$.
\item One has $o_\Delta(q(g\circ f))=0$ for every composable pair
      $X\xrightarrow{f}Y\xrightarrow{g}Z$.
\end{enumerate}
\end{maintheorema}

Thus $\mathrm{TR4}$ is characterized by the vanishing of explicit local
obstruction classes in quotients of ordinary Hom groups.

The local criterion treats one pre-triangulation at a time.  A complementary
relative criterion compares two such structures.  For \(i=0,1\), equip the
same additive category \(\T\) with a suspension \(\Sigma_i\) and a
pre-triangulation \(\Delta_i\), and set
\[
 \mathsf F_i=(\Sigma_i^{-1})_*:\mathcal E\longrightarrow\mathcal E.
\]
These functors also act on the stable category $\underline{\mathcal E}$,
obtained by factoring out morphisms through projective-injective objects;
see \cref{not:stable-freyd}.  Their inverse Heller comparisons, constructed
in \cref{prop:intrinsic-inverse-comparison}, determine a stable natural
isomorphism
\[
 \tau:\mathsf F_0\xrightarrow{\sim}\mathsf F_1.
\]
This is the relative comparison of \cref{def:relative-comparison}.
For a short exact sequence $e$ in $\mathcal E$, an exact lift of $\tau$
along $e$ is a morphism of short exact sequences
\[
 \mathsf F_0(e)\longrightarrow\mathsf F_1(e)
\]
whose three components represent those of $\tau$ in the stable category;
see \cref{def:exact-lift}.  We obtain the following relative criterion.

\begin{maintheoremb}[{\normalfont see~\Cref{thm:relative-heller-criterion}}]
Assume that $\Delta_0$ satisfies $\mathrm{TR4}$.  Then $\Delta_1$
satisfies $\mathrm{TR4}$ if and only if the relative comparison $\tau$
admits an exact lift along every short exact sequence in $\mathcal E$.
\end{maintheoremb}

The criterion can be used both to establish $\mathrm{TR4}$ and to detect
its failure.  Constructing exact lifts along all short exact sequences
proves $\mathrm{TR4}$.  Failure of exact lifting along a single short exact
sequence rules it out.

Most triangulated categories arising in algebra, geometry, and topology
come with an algebraic or topological model.  A triangulated category is
called \emph{exotic} if it is neither algebraic nor topological.
Muro--Schwede--Strickland constructed the first examples
\cite{MuroSchwedeStrickland2007}.  Rizzardo--Van den Bergh later produced
the first $\kk$-linear examples, over fields of characteristic zero
\cite{RizzardoVandenBergh2020}.  More recently, D\'iaz Cabrera--Muro
constructed, over every algebraically closed field, a pre-triangulation on
the category of finitely generated projective modules over the generalized
type-$\mathbb L_2$ preprojective algebra which is neither algebraic nor
topological \cite{DiazCabreraMuro2026}.  They left open whether this
structure satisfies the octahedral axiom.

The relative criterion resolves this question.

\begin{maintheoremc}[{\normalfont see~\Cref{thm:L2-octahedral}}]
Over every algebraically closed field, the D\'iaz Cabrera--Muro
pre-triangulation $\Delta_\sigma$ on
$\operatorname{proj}P(\mathbb L_2)$ satisfies $\mathrm{TR4}$.
Consequently,
$(\operatorname{proj}P(\mathbb L_2),\Delta_\sigma)$ is an exotic
triangulated category.
\end{maintheoremc}

To detect failure of $\mathrm{TR4}$, we use the central Heller twists
$\Delta_\lambda^\delta$ of \cref{def:central-twist}, obtained from a
triangulation $\Delta_0$ by a square-zero central endomorphism $\delta$.
For these twists, \cref{prop:central-twist-obstruction} expresses exact
liftability in terms of a class in a quotient of a stable endomorphism
space.  If this class is nonzero, only the original triangulation,
corresponding to $\lambda=0$, satisfies $\mathrm{TR4}$.
Under the Calabi--Yau hypotheses of \cref{sec:serre-traces}, Serre traces
give a criterion for this nonvanishing.  In particular,
\cref{prop:periodic-orbit-detector} gives a criterion using a socle orbit
(\cref{def:periodic-socle-orbit}) and a short exact sequence ending at an
object of that orbit.  The first two terms have no summands in the orbit,
and the sequence has compatible period identifications
(\cref{def:trace-functional}).
The constructions in \cref{sec:deformations,sec:ade-lines} provide these
data for the preprojective algebras in the following theorem.

\begin{maintheoremd}[{\normalfont see~\Cref{thm:scalar-family,thm:ade-explicit-lines}}]
For every field $\kk$ and every $a\ge2$, the projectives over the
type-$A_{3a-1}$ preprojective algebra carry pairwise distinct
pre-triangulations
$\{\Delta_{\kk,\lambda}\}_{\lambda\in\kk}$ satisfying
\[
 \Delta_{\kk,\lambda}\text{ satisfies }\mathrm{TR4}
 \quad\Longleftrightarrow\quad\lambda=0.
\]
If $\operatorname{char}\kk=2$, the same conclusion holds for the
preprojective algebras of types
$A_n$ $(n\ge5)$, $D_n$ $(n\ge5)$, $E_6$, $E_7$, and $E_8$.
It also holds for type $D_4$ if $\operatorname{char}\kk=2$ and $|\kk|>2$.
\end{maintheoremd}

The local obstruction also behaves well under retracts.  Its defining
equations are functorial under exact functors, and a separability retraction
on Hom groups carries a solution back to the source.
Here an exact functor includes a comparison with suspension, as in
\cref{def:exact-functor}.  Separability means that the induced maps on Hom
groups admit a natural bimodule retraction, as in \cref{def:separable}.

\begin{maintheoreme}[{\normalfont see~\Cref{thm:separable-descent}}]
Let \(\D\) be a pre-triangulated category, let \(\C\) be a triangulated
category, and let \(F:\D\to\C\) be exact and separable.  Then the given
pre-triangulation on \(\D\) satisfies \(\mathrm{TR4}\).
\end{maintheoreme}

This applies to the canonical pre-triangulations on non-modular
equivariantizations \cite{Sun2019} and on the K3 deformation categories of
Markman--Mehrotra \cite{MarkmanMehrotra2015}.  The geometric hypotheses
needed for the latter were verified in
\cite[Theorem~1.11]{MarkmanMehrotraVerbitsky2019}.  It also gives the
corresponding separable-monad transfer for ordinary triangulated
categories.

This paper is organized as follows. In Section~\ref{sec:preliminaries},
we collect the required facts on
pre-triangulations, Freyd categories, inverse Heller comparisons, and
Heller parameters for a fixed suspension.  Section~\ref{sec:obstruction}
develops characteristic extensions and the local obstruction.
Section~\ref{sec:relative-heller} proves the relative exact-lifting
criterion and a trace criterion for central twists.  Section~\ref{sec:L2-test}
treats the generalized type-$\mathbb L_2$ category.
Section~\ref{sec:transfer} proves separable descent and its equivariant and
geometric applications.  Sections~\ref{sec:deformations}
and~\ref{sec:ade-lines} construct the type-$A_{3a-1}$ families over
arbitrary fields and the Dynkin families in characteristic two, respectively.

\section{Preliminaries}
\label{sec:preliminaries}

All categories are additive and essentially small unless stated otherwise.
A suspended category is written $(\T,\Sigma)$, with $\Sigma$ an additive
autoequivalence.  Composition is written $f\circ g$ when $g$ is followed by
$f$.  If $a:i\to j$ and $b:j\to\ell$ are arrows in a quiver, then
$ab:i\to\ell$ denotes the path which first traverses $a$ and then $b$.
As a composite of morphisms, this path is $b\circ a$.
Whenever a ground field is present, it is denoted by \(\kk\).
All modules over rings and algebras are right modules.

\subsection{Pre-triangulations and mapping cones}

\begin{definition}[{\normalfont cf.~\cite[Definition~2.2]{ChenLiuLuZhang2026},
\cite[Definitions~1.1.1--1.1.2]{Neeman2001}, and \cite{Puppe1967}}]
\label{def:pretriangulation}
A \emph{candidate triangle} in $(\T,\Sigma)$ is a sequence
\begin{equation*}
\xymatrix@C=2.2pc{
 X\ar[r]^-{u}&Y\ar[r]^-{v}&Z\ar[r]^-{w}&\Sigma(X)
}
\end{equation*}
such that \(v\circ u=0\), \(w\circ v=0\), and
\(\Sigma(u)\circ w=0\).
A morphism from this candidate triangle to
\begin{equation*}
	\xymatrix@C=2.2pc{
		X'\ar[r]^-{u'}&Y'\ar[r]^-{v'}&Z'\ar[r]^-{w'}&\Sigma(X')
	}
\end{equation*}
is a triple $(f,g,h)$ satisfying
\[
 g\circ u=u'\circ f,\qquad h\circ v=v'\circ g,\qquad
 w'\circ h=(\Sigma(f))\circ w.
\]
The \emph{mapping-cone candidate triangle} of such a morphism is
\begin{equation}
\label{eq:mapping-cone}
\xymatrix@C=2.0pc{
Y\oplus X'\ar[r]^-{\alpha}&Z\oplus Y'\ar[r]^-{\beta(h)}&
\Sigma(X)\oplus Z'\ar[r]^-{\gamma}&
\Sigma(Y)\oplus\Sigma(X')
}
\end{equation}
where
\begin{equation}
\label{eq:cone-matrices}
 \alpha=
 \begin{bmatrix}-v&0\\ g&u'\end{bmatrix},\qquad
 \beta(h)=
 \begin{bmatrix}-w&0\\ h&v'\end{bmatrix},\qquad
 \gamma=
 \begin{bmatrix}-\Sigma(u)&0\\ \Sigma(f)&w'\end{bmatrix}.
\end{equation}

A \emph{pre-triangulation} on $(\T,\Sigma)$ is a class $\Delta$ of
candidate triangles, called \emph{distinguished triangles}, satisfying
\textup{(TR1)}--\textup{(TR3)} below.  We record \textup{(TR4)} in the
same list for later reference.
\begin{description}[leftmargin=4.2em,style=nextline]
\item[\textup{(TR1)}]
The class $\Delta$ is closed under isomorphisms,
\begin{equation*}
	\xymatrix@C=2.2pc{
		X\ar[r]^-{1_X}&X\ar[r]&0\ar[r]&\Sigma(X)
	}
\end{equation*}
is distinguished for every $X\in\T$, and every morphism
$u:X\to Y$ is the first map of a distinguished triangle.
\item[\textup{(TR2)}]
A candidate triangle
$X\xrightarrow{u}Y\xrightarrow{v}Z\xrightarrow{w}\Sigma(X)$
is distinguished if and only if its left rotation
\begin{equation*}
	\xymatrix@C=2.2pc{
		Y\ar[r]^-{v}&Z\ar[r]^-{w}&
		\Sigma(X)\ar[r]^-{-\Sigma(u)}&\Sigma(Y)
	}
\end{equation*}
is distinguished.
\item[\textup{(TR3)}]
Every commutative square on the first maps of two distinguished triangles
extends to a morphism of candidate triangles.
\item[\textup{(TR4)}]
Suppose that
\[
\begin{gathered}
 X\xrightarrow{u}Y\xrightarrow{i}X'\xrightarrow{i'}\Sigma(X),
 \qquad Y\xrightarrow{v}Z\xrightarrow{j}Z'\xrightarrow{j'}\Sigma(Y),\\
 X\xrightarrow{v\circ u}Z\xrightarrow{k}Y'\xrightarrow{k'}\Sigma(X)
\end{gathered}
\]
are distinguished.  There exist morphisms
$u':X'\to Y'$ and $v':Y'\to Z'$ such that
\[
 u'\circ i=k\circ v,\qquad
 v'\circ k=j,\qquad
 k'\circ u'=i',\qquad
 j'\circ v'=(\Sigma(u))\circ k',
\]
and the candidate triangle
\begin{equation*}
	\xymatrix@C=2.6pc{
		X'\ar[r]^-{u'}&Y'\ar[r]^-{v'}&
		Z'\ar[r]^-{(\Sigma(i))\circ j'}&\Sigma(X')
	}
\end{equation*}
is distinguished.
\end{description}
A category equipped with a pre-triangulation is
\emph{pre-triangulated}.  A pre-triangulation satisfying
\textup{(TR4)} is a \emph{triangulation}.  A \emph{triangulated category}
is an additive category equipped with a suspension and a triangulation.
\end{definition}

We follow Verdier's numbering and include closure under isomorphisms in
\textup{(TR1)}.  This closure axiom is denoted \textup{(TR0)} in
\cite[Definition~1.1.2]{Neeman2001}.  We refer to \textup{(TR4)} above as
Verdier's octahedral axiom.

Following Neeman, we write \textup{(TR4$'$)} for the \emph{mapping-cone
axiom}: in \textup{(TR3)}, the third component can be chosen so that the
mapping-cone candidate triangle \eqref{eq:mapping-cone} is distinguished;
cf.~\cite[Definition~1.3.13]{Neeman2001}.  Under
\textup{(TR1)}--\textup{(TR3)}, axioms \textup{(TR4)} and
\textup{(TR4$'$)} are equivalent.  The implication
\textup{(TR4$'$)}$\Rightarrow$\textup{(TR4)} is
\cite[Proposition~1.4.6]{Neeman2001}, and the converse is
\cite[Theorem~1.8]{Neeman1991}.

Unless otherwise stated, throughout \cref{sec:preliminaries,sec:obstruction}
we fix a pre-triangulated category $(\T,\Sigma,\Delta)$.

\begin{definition}[{\normalfont cf.~\cite[Definition~2.1.1]{Neeman2001}}]
\label{def:exact-functor}
Let \(\T\) and \(\T'\) be pre-triangulated categories, with suspensions
\(\Sigma,\Sigma'\) and pre-triangulations \(\Delta,\Delta'\), respectively.
An \emph{exact functor} \(F:\T\to\T'\) is an additive functor
together with a natural isomorphism
$\epsilon:F\circ\Sigma\xrightarrow{\sim}\Sigma'\circ F$
such that the image of every distinguished triangle
\begin{equation*}
\xymatrix@C=2.2pc{
 X\ar[r]^-{u}&Y\ar[r]^-{v}&Z\ar[r]^-{w}&\Sigma(X).
}
\end{equation*}
is the distinguished triangle
\begin{equation*}
\xymatrix@C=3.2pc{
 F(X)\ar[r]^-{F(u)}&F(Y)\ar[r]^-{F(v)}&F(Z)
 \ar[r]^-{\epsilon_X\circ F(w)}&\Sigma'(F(X)).
}
\end{equation*}
\end{definition}

When both categories are triangulated, an exact functor is also called a
\emph{triangulated functor}.  A \emph{morphism of exact functors}
$\eta:(F,\epsilon_F)\Rightarrow(G,\epsilon_G)$ is a natural transformation
$\eta:F\Rightarrow G$ satisfying
\[
 (\epsilon_G)_X\circ\eta_{\Sigma(X)}
 =\Sigma'(\eta_X)\circ(\epsilon_F)_X
 \qquad(X\in\T).
\]
For a triangulated functor $(F,\epsilon_F)$, we write
$\operatorname{Aut}_{\mathrm{tri}}(F)$ for its automorphism group in this
sense.  Suspension comparisons are left implicit when no ambiguity can arise.

We shall use the following standard exactness fact repeatedly.

\begin{lemma}[{\normalfont see~\cite[Lemma~1.1.10 and
Remark~1.1.11]{Neeman2001}}]
\label{lem:representable-exactness}
Let $\T$ be a pre-triangulated category, and let
$X\xrightarrow{u}Y\xrightarrow{v}Z\xrightarrow{w}\Sigma(X)$ be
distinguished.  For every $A\in\T$, the associated sequences
\[
 \begin{aligned}
 \cdots&\longrightarrow\T(A,\Sigma^n(X))
 \xrightarrow{(\Sigma^n(u))_*}\T(A,\Sigma^n(Y))
 \xrightarrow{(\Sigma^n(v))_*}\T(A,\Sigma^n(Z))\\
 &\xrightarrow{(\Sigma^n(w))_*}\T(A,\Sigma^{n+1}(X))
 \longrightarrow\cdots
 \end{aligned}
\]
and
\[
 \begin{aligned}
 \cdots&\longrightarrow\T(\Sigma^{n+1}(X),A)
 \xrightarrow{(\Sigma^n(w))^*}\T(\Sigma^n(Z),A)
 \xrightarrow{(\Sigma^n(v))^*}\T(\Sigma^n(Y),A)\\
 &\xrightarrow{(\Sigma^n(u))^*}\T(\Sigma^n(X),A)
 \longrightarrow\cdots
 \end{aligned}
\]
are exact.
\end{lemma}

\begin{definition}[{\normalfont cf.~\cite[Definition~2.1]{GKO2013}}]
\label{def:yoneda-exact}
Let $\T$ be a pre-triangulated category, and let
$X\xrightarrow{u}Y\xrightarrow{v}Z\xrightarrow{w}\Sigma(X)$ be a candidate
triangle.  It is \emph{Yoneda exact} if the associated sequence of
contravariant representable functors
\[
 \cdots\longrightarrow h_{\Sigma^{-1}(Z)}
 \xrightarrow{h_{\Sigma^{-1}(w)}}h_X
 \xrightarrow{h_u}h_Y\xrightarrow{h_v}h_Z
 \xrightarrow{h_w}h_{\Sigma(X)}\longrightarrow\cdots
\]
is exact.
\end{definition}

Every distinguished triangle is Yoneda exact by
\cref{lem:representable-exactness}.  We also use the elementary observation
that mapping cones preserve this weaker exactness.

\begin{lemma}
\label{lem:cone-exact}
Let $\T$ be a pre-triangulated category.  The mapping cone
of a morphism between Yoneda-exact candidate triangles is Yoneda exact.
\end{lemma}

\begin{proof}
After applying $\T(A,-)$, the two candidate triangles become exact complexes
and \eqref{eq:mapping-cone} is, up to the signs in
\eqref{eq:cone-matrices}, the ordinary mapping cone of the induced chain map.
The mapping cone of a chain map between exact complexes is exact.
\end{proof}

\subsection{The Freyd category}
\label{sec:freyd}

\begin{definition}[{\normalfont cf.~\cite{Freyd1966} and
\cite[Example~4.8]{Posur2021}}]
\label{not:freyd}
Let $\T$ be a pre-triangulated category.  Put
\[
 h_X=\T(-,X).
\]
The \emph{Freyd category} $\mathsf{mod}\,\T$ is the full subcategory of
additive functors $\T^{\mathrm{op}}\to\mathsf{Ab}$ consisting of the
finitely presented ones, i.e.\ those admitting a presentation
\begin{equation*}
\xymatrix@C=2.3pc{
 h_{X_1}\ar[r]&h_{X_0}\ar[r]&M\ar[r]&0.
}
\end{equation*}
\end{definition}

The following standard fact is the only structural property of the Freyd
category that we need.

\begin{proposition}[{\normalfont see~\cite[Proposition~2.5(b)]{GKO2013};
cf.~\cite[Proposition~5.1.10 and Corollary~5.1.23]{Neeman2001}}]
\label{prop:freyd-category}
Let $\T$ be a pre-triangulated category.  Then
$\mathsf{mod}\,\T$ is an abelian
Frobenius category.  Its projective-injective objects are precisely the
direct summands of representable functors.
\end{proposition}

Although Neeman states the result for triangulated categories, his proof
uses only $\mathrm{TR1}$--$\mathrm{TR3}$.  The same is true of
\cite[Proposition~2.5(b)]{GKO2013}.  Since representables are also injective,
applying $\Hom_{\mathsf{mod}\,\T}(-,h_A)$ to the associated exact complex shows
that every Yoneda-exact candidate triangle also gives an exact sequence
after applying $\T(-,A)$, for every $A\in\T$.  Thus the dual representable
sequence is automatic and need not be included in the definition.

\begin{definition}
\label{not:stable-freyd}
Let $\T$ be a pre-triangulated category.  Write
\[
 \underline{\mathsf{mod}}\,\T
 =\mathsf{mod}\,\T/
 \{\text{morphisms factoring through projective-injectives}\}.
\]
Its syzygy and cosyzygy functors are
$\Omega$ and $\Omega^{-1}$.  We write $\underline{\Hom}$,
$\underline{\End}$ and $\underline{\Aut}$ for morphisms in this stable
category.  The suspension of $\T$ induces the exact autoequivalence
\[
 \Sigma_*(M)=M\circ\Sigma^{-1}
\]
of $\mathsf{mod}\,\T$ and of its stable category.  Unless a different
abelian category is displayed, all $\Ext$-groups below are taken in
$\mathsf{mod}\,\T$.
\end{definition}

The stable Freyd category is triangulated by the standard Frobenius
construction \cite[Chapter~I, Section~2]{Happel1988}.

For later use, recall the natural isomorphisms in a Frobenius category
$\mathcal E$,
\begin{equation}
\label{eq:stable-dimension-shift}
 \Ext^n_{\mathcal E}(A,B)
 \xrightarrow{\sim}
 \underline{\Hom}_{\mathcal E}(A,\Omega^{-n}B),
 \qquad n\geq1.
\end{equation}
They are obtained by splicing the chosen projective-injective cosyzygy
sequences in the displayed order; see
\cite[Chapter~I, Section~2]{Happel1988}.  We use representable
projective-injective resolutions in $\mathsf{mod}\,\T$ and this convention
throughout.

\begin{notation}
\label{not:characteristic-functors}
For a morphism $c:C_0\to C_1$, put
\[
 B(c)=\ker h_c,\qquad
 K(c)=\coker h_c,\qquad
 L(c)=\ker h_{\Sigma(c)}.
\]
Completing $c$ to a distinguished triangle and using
\cref{lem:representable-exactness} shows that these functors are finitely
presented.
\end{notation}

\subsection{Inverse Heller comparisons and parameters}

Put
\[
 \mathcal E=\mathsf{mod}\,\T,
 \qquad
 \underline{\mathcal E}=\underline{\mathsf{mod}}\,\T.
\]
For convenience, every suspension is equipped with a fixed quasi-inverse,
and we use the chosen unit and counit to identify its two composites with
the identity.  Thus we write
\[
 \Sigma^{-1}\circ\Sigma=1_{\T},
 \qquad
 \Sigma\circ\Sigma^{-1}=1_{\T}.
\]
The same convention applies to the induced functors on $\mathcal E$ and
$\underline{\mathcal E}$.

If
\[
0\longrightarrow N\longrightarrow P_2\longrightarrow P_1
 \longrightarrow P_0\longrightarrow M\longrightarrow0
\]
is exact and the three middle terms are projective-injective, we write its
stable connecting morphism in the inverse direction
$\Omega^3(M)\to N$.

Let $\T$ be a pre-triangulated category and put
\[
 \mathsf F_\Delta=(\Sigma^{-1})_*:\mathcal E\longrightarrow\mathcal E.
\]
Every object of $\mathcal E$ is isomorphic to $B(c)$ for some morphism $c$
of $\T$: complete a morphism in a representable presentation to a
distinguished triangle and use representable exactness.  For
$M\in\mathcal E$, choose $c:C_0\to C_1$, an isomorphism
$\mathsf F_\Delta(M)\cong B(c)$, and a distinguished completion
\begin{equation*}
\xymatrix@C=2.2pc{
C_0\ar[r]^-{c}&C_1\ar[r]^-{d}&Z\ar[r]^-{e}&\Sigma(C_0)
}
\end{equation*}
of $c$.  After transporting along the chosen isomorphism, representable
exactness gives
\begin{equation}
\label{eq:intrinsic-heller-row}
\xymatrix@C=1.8pc{
0\ar[r]&\mathsf F_\Delta(M)\ar[r]&h_{C_0}\ar[r]^-{h_c}&h_{C_1}
\ar[r]^-{h_d}&h_Z\ar[r]&M\ar[r]&0.
}
\end{equation}
Let $\theta_{\Delta,M}:\Omega^3(M)\to\mathsf F_\Delta(M)$ be its stable
connecting morphism.  For convenience, we record and prove the following
fact.

\begin{proposition}
\label{prop:intrinsic-inverse-comparison}
The following assertions hold.
\begin{enumerate}[label=\textup{(\roman*)}]
 \item The morphism $\theta_{\Delta,M}$ is independent of all choices and
       natural in $M$.  Hence the maps $\theta_{\Delta,M}$ form a natural
       isomorphism
       \begin{equation}
       \label{eq:intrinsic-inverse-comparison}
        \theta_\Delta:\Omega^3\xrightarrow{\sim}\mathsf F_\Delta
       \end{equation}
       on $\underline{\mathcal E}$.
 \item Let
       \begin{equation}
\label{eq:representable-heller-row}
\xymatrix@C=1.8pc{
0\ar[r]&\mathsf F_\Delta(M)\ar[r]^-{\iota}&h_{X_2}\ar[r]^-{h_a}&
h_{X_1}\ar[r]^-{h_b}&h_{X_0}\ar[r]^-{p}&M\ar[r]&0.
}
\end{equation}
       be exact.  Let $d:X_0\to\Sigma(X_2)$ be the unique morphism with
       \[
        h_d=(\Sigma_*\iota)\circ p.
       \]
       Then the candidate triangle
       \begin{equation*}
\xymatrix@C=2.2pc{
X_2\ar[r]^-{a}&X_1\ar[r]^-{b}&X_0\ar[r]^-{d}&\Sigma(X_2)
}
\end{equation*}
       is Yoneda exact.  It is distinguished if and only if the stable
       connecting morphism of \eqref{eq:representable-heller-row} is
       $(\theta_\Delta)_M$.
\end{enumerate}
\end{proposition}

\begin{proof}
We prove \textup{(i)} by showing naturality for arbitrary choices.  Let
$\alpha:M\to N$, identify $\mathsf F_\Delta(M)\cong B(c)$ and
$\mathsf F_\Delta(N)\cong B(c')$, and write
$\bar\alpha:B(c)\to B(c')$ for the transported map
$\mathsf F_\Delta(\alpha)$.  Choose distinguished triangles beginning
with $c:C_0\to C_1$ and $c':C'_0\to C'_1$.  Since $h_{C'_0}$ is
injective, the composite
\[
 B(c)\xrightarrow{\bar\alpha}B(c')\hookrightarrow h_{C'_0}
\]
extends across $B(c)\hookrightarrow h_{C_0}$.  By the Yoneda lemma the
extension is induced by a morphism $a_0:C_0\to C'_0$.  The composite
$h_{c'}\circ h_{a_0}$ vanishes on $B(c)$, hence factors through
$\operatorname{im}h_c\subseteq h_{C_1}$.  Injectivity of $h_{C'_1}$ and
the Yoneda lemma give $a_1:C_1\to C'_1$ with
\[
 a_1\circ c=c'\circ a_0.
\]
Axiom $\mathrm{TR3}$ extends $(a_0,a_1)$ to a morphism of the chosen
distinguished triangles.  Applying the Yoneda embedding gives a morphism
between the rows \eqref{eq:intrinsic-heller-row}.  Its maps on the end terms
are $\mathsf F_\Delta(\alpha)$ and $\alpha$.  Naturality of stable
connecting morphisms yields
\[
 \mathsf F_\Delta(\alpha)\circ(\theta_\Delta)_M
 = (\theta_\Delta)_N\circ\Omega^3(\alpha).
\]
Taking $\alpha=1_M$ and interchanging the two choices proves independence.
Each row \eqref{eq:intrinsic-heller-row} has three projective-injective
middle terms, so its stable connecting morphism is an isomorphism.

For \textup{(ii)}, exactness identifies $\iota$ with the inclusion of
$B(a)=\ker h_a$.  Since $p$ is epic and $\Sigma_*\iota$ is monic, the
definition of $d$ gives
\[
 \ker h_d=\ker p=\operatorname{im}h_b,
 \qquad
 \operatorname{im}h_d=\operatorname{im}(\Sigma_*\iota)
 =\ker h_{\Sigma(a)}.
\]
Together with \eqref{eq:representable-heller-row} and its
$\Sigma_*$-translates, these identities prove Yoneda exactness.

Choose a distinguished completion
\[
 X_2\xrightarrow{a}X_1\xrightarrow{s}T
 \xrightarrow{t}\Sigma(X_2).
\]
The given row and the Heller row of this completion have the same first
two projective terms.  Via \eqref{eq:stable-dimension-shift}, their tails
\[
 0\longrightarrow K(a)\longrightarrow h_{X_0}
   \longrightarrow M\longrightarrow0,
 \qquad
 0\longrightarrow K(a)\longrightarrow h_T
   \longrightarrow M\longrightarrow0
\]
correspond respectively to the stable connecting morphism of
\eqref{eq:representable-heller-row} and to $(\theta_\Delta)_M$.  The two
stable maps therefore agree if and only if these extensions have the same
Yoneda class.  In that case there is a morphism of short exact sequences
which is the identity on $K(a)$ and $M$.  Its middle component is an
isomorphism by the short five lemma and, by the Yoneda lemma, is induced by
an isomorphism $T\to X_0$ identifying the two candidate triangles.  The
converse follows from the same comparison.
\end{proof}

\begin{definition}
\label{def:heller-resolution}
A \emph{Heller resolution of $M\in\mathcal E$ for $\Delta$} is an exact
sequence
\begin{equation*}
\xymatrix@C=1.8pc{
0\ar[r]&\mathsf F_\Delta(M)\ar[r]&P_2\ar[r]&P_1\ar[r]&P_0\ar[r]&M\ar[r]&0
}
\end{equation*}
whose three middle terms are projective-injective and whose stable
connecting morphism is $(\theta_\Delta)_M$.  It is
\emph{representable} if each $P_i$ is representable.
\end{definition}

The preceding proposition gives the following recognition criterion: a
representable exact row is a Heller resolution if and only if its associated
Yoneda-exact candidate triangle is distinguished.

We use throughout the inverse comparison
\[
 \theta_\Delta:\Omega^3\xrightarrow{\sim}(\Sigma^{-1})_*.
\]
The suspension comparison for $\Omega^3$ is $-1_{\Omega^2}$, and
$(\Sigma^{-1})_*$ has the exact structure induced from $\mathcal E$.
With these structures, $\theta_\Delta$ is the inverse form of the
isomorphism of triangulated functors $\Sigma_*\xrightarrow{\sim}\Omega^{-3}$ in
\cite[Lemma~3.2 and Proposition~3.4]{GKO2013}.

\begin{lemma}[{\normalfont see~\cite[Theorem~16.4]{Heller1968} and
\cite[Lemma~3.3 and Proposition~3.4]{GKO2013}}]
\label{lem:heller-parameters}
Let \(\T\) be an idempotent-complete triangulated category with suspension
\(\Sigma\) and triangulation \(\Delta_0\).
Put
\[
 \mathsf F_0=(\Sigma^{-1})_*:\mathcal E\longrightarrow\mathcal E,
\]
and let $\theta_0:\Omega^3\xrightarrow{\sim}\mathsf F_0$ be the inverse
Heller comparison of $\Delta_0$.  The functors $\Omega^3$ and $\mathsf F_0$
are then fixed triangulated autoequivalences of $\underline{\mathcal E}$.
For the fixed suspension $\Sigma$, the assignment
$\Delta\mapsto\theta_\Delta$ induces a bijection
\[
\left\{
 \begin{array}{c}
 \text{pre-triangulations on $(\T,\Sigma)$}
 \end{array}
\right\}
\longleftrightarrow
\left\{
 \begin{array}{c}
 \text{isomorphisms of triangulated functors}\\[-2pt]
 \theta:\Omega^3\xrightarrow{\sim}\mathsf F_0
 \end{array}
\right\}.
\]
\end{lemma}

\section{Characteristic extensions and local octahedral obstructions}
\label{sec:obstruction}

\subsection{Characteristic extensions}

We associate an extension class to a morphism $c$ by completing it to a
distinguished triangle.  This class characterizes the distinguished
triangles among all Yoneda-exact candidate triangles beginning with $c$.

\begin{definition}
\label{def:characteristic-class-candidate}
Let
\begin{equation*}
\xymatrix@C=2.2pc{
 E:\quad C_0\ar[r]^-{c}&C_1\ar[r]^-{d}&Z\ar[r]^-{e}&\Sigma(C_0).
}
\end{equation*}
be a Yoneda-exact candidate triangle in a pre-triangulated category
$(\T,\Sigma,\Delta)$.  Write
$q_c:h_{C_1}\twoheadrightarrow K(c)$ and
$i_c:L(c)\hookrightarrow h_{\Sigma(C_0)}$ for the canonical maps.
Yoneda exactness gives unique morphisms $\bar d$ and $\bar e$ such
that
\[
 h_d=\bar d\circ q_c,
 \qquad
 h_e=i_c\circ\bar e,
\]
and the resulting sequence
\begin{equation*}
\xymatrix@C=2.2pc{
0\ar[r]&K(c)\ar[r]^-{\bar d}&h_Z\ar[r]^-{\bar e}&L(c)\ar[r]&0.
}
\end{equation*}
is exact.  Its Yoneda class is the \emph{characteristic class of $E$}:
\[
 \chi(E):=
 \bigl[0\to K(c)\xrightarrow{\bar d}h_Z
       \xrightarrow{\bar e}L(c)\to0\bigr]
 \in\Ext^1_{\mathsf{mod}\,\T}(L(c),K(c)).
\]
\end{definition}

Axiom $\mathrm{TR1}$ supplies a distinguished completion
\begin{equation*}
\xymatrix@C=2.2pc{
E_c:\quad C_0\ar[r]^-{c}&C_1\ar[r]^-{s}&T\ar[r]^-{t}&\Sigma(C_0).
}
\end{equation*}

\begin{proposition}
\label{prop:characteristic-class-independence}
Let $\T$ be a pre-triangulated category.  If $E_c$ and
$E'_c$ are distinguished completions of the same morphism $c$, then
$\chi(E_c)=\chi(E'_c)$.
\end{proposition}

\begin{proof}
Axiom $\mathrm{TR3}$ extends the identity square on $c$ to a morphism of
triangles $(1_{C_0},1_{C_1},\psi):E_c\to E'_c$.  Its third component is an
isomorphism by \cite[Proposition~1.1.20]{Neeman2001}.  Applying the Yoneda
embedding gives an isomorphism between the characteristic short exact
sequences which is the identity on $K(c)$ and $L(c)$, so their Yoneda
classes agree.
\end{proof}

By \cref{prop:characteristic-class-independence}, the class $\chi(E_c)$ is
independent of the distinguished completion $E_c$ of $c$.  We may therefore
make the following definition.

\begin{definition}
\label{def:characteristic-class-pretriangulation}
Let $\T$ be a pre-triangulated category.  For a morphism
$c:C_0\to C_1$, define
\[
 \chi_\Delta(c):=\chi(E_c)
 \in\Ext^1_{\mathsf{mod}\,\T}(L(c),K(c)),
\]
where $E_c$ is any distinguished completion of $c$.
\end{definition}

\begin{proposition}
\label{prop:characteristic-extension}
Let $\T$ be a pre-triangulated category, and let
\[
 E:\quad C_0\xrightarrow{c}C_1\xrightarrow{d}Z
 \xrightarrow{e}\Sigma(C_0)
\]
be a Yoneda-exact candidate triangle.
\begin{enumerate}[label=\textup{(\roman*)}]
\item The candidate triangle $E$ is distinguished if and only if
$\chi(E)=\chi_\Delta(c)$.
\item Pushout along endomorphisms of $K(c)$ induces an isomorphism
\begin{equation}
\label{eq:stable-unit}
 \Phi_c:
 \underline{\End}_{\mathsf{mod}\,\T}(K(c))
 \xrightarrow{\sim}
 \Ext^1_{\mathsf{mod}\,\T}(L(c),K(c)),
 \qquad
 [a]\longmapsto a_*\chi_\Delta(c),
\end{equation}
which sends $[1_{K(c)}]$ to $\chi_\Delta(c)$.
\end{enumerate}
\end{proposition}

\begin{proof}
For~\textup{(i)}, choose a distinguished completion
$E_c=(C_0\xrightarrow{c}C_1\xrightarrow{s}T
\xrightarrow{t}\Sigma(C_0))$.
If $\chi(E)=\chi_\Delta(c)$, equality of the Yoneda classes gives a
morphism of short exact sequences which is the identity on the end terms:
\[
\begin{tikzcd}[column sep=large]
0 \arrow[r] & K(c) \arrow[r,"\bar s"] \arrow[d,equal]
  & h_T \arrow[r,"\bar t"] \arrow[d,"\theta"]
  & L(c) \arrow[r] \arrow[d,equal] &0\\
0 \arrow[r] & K(c) \arrow[r,"\bar d"']
  & h_Z \arrow[r,"\bar e"']
  & L(c) \arrow[r] &0.
\end{tikzcd}
\]
The short five lemma shows that $\theta$ is an isomorphism.
By the Yoneda lemma, $\theta=h_\phi$ for a unique morphism
$\phi:T\to Z$.  Since $q_c$ is epic and $i_c$ is monic, commutativity gives
$\phi\circ s=d$ and $e\circ\phi=t$.  The Yoneda embedding reflects
isomorphisms, so $\phi$ is an isomorphism.  Hence $E$ is isomorphic to
$E_c$ and is distinguished.  Conversely, if $E$ is distinguished, then
$\chi(E)=\chi(E_c)$ by
\cref{prop:characteristic-class-independence}.  This proves (i).

For~\textup{(ii)}, let $a\in\End(K(c))$ and form the pushout square
\[
\begin{tikzcd}[column sep=large,row sep=large]
 K(c) \arrow[r,"\bar s"] \arrow[d,"a"']
   & h_T \arrow[d,"j_a"] \\
 K(c) \arrow[r,"i_a"'] & P_a.
\end{tikzcd}
\]
The resulting exact sequence
$0\to K(c)\xrightarrow{i_a}P_a\to L(c)\to0$ represents
$a_*\chi_\Delta(c)$.  Apply
$\Hom_{\mathsf{mod}\,\T}(-,K(c))$ to the characteristic sequence of
$E_c$.  Since $h_T$ is projective, we obtain
\[
 \Hom(h_T,K(c))\longrightarrow\End(K(c))
 \xrightarrow{\partial}
 \Ext^1(L(c),K(c))\longrightarrow0,
\]
where the pushout identifies $\partial(a)=a_*\chi_\Delta(c)$.  Thus
$\partial$ is surjective.  Its kernel consists of the endomorphisms of
$K(c)$ which extend across $K(c)\hookrightarrow h_T$.
Such an extension gives a factorization through the projective-injective
object $h_T$.
Conversely, if an endomorphism factors as $K(c)\to P\to K(c)$ with $P$
projective-injective, injectivity of $P$ extends the first map across
$K(c)\hookrightarrow h_T$.
Thus the kernel consists precisely of the endomorphisms which factor
through a projective-injective object.  Passing to the stable quotient
gives \eqref{eq:stable-unit}.
\end{proof}

We record the following useful reformulation.

\begin{lemma}
\label{lem:two-arrow-recognition}
Let $\T$ be a pre-triangulated category.  Let
\begin{equation*}
\xymatrix@C=2.2pc{
C_0\ar[r]^-{c}&C_1\ar[r]^-{s}&T\ar[r]^-{t}&\Sigma(C_0).
}
\end{equation*}
be distinguished, and let
\begin{equation*}
\xymatrix@C=2.2pc{
 E:\quad C_0\ar[r]^-{c}&C_1\ar[r]^-{d}&Z\ar[r]^-{e}&\Sigma(C_0).
}
\end{equation*}
be Yoneda exact.  Then $E$ is distinguished if and only if there is a
morphism $\phi:T\to Z$ satisfying
\[
 \phi\circ s=d,
 \qquad
 e\circ\phi=t.
\]
Every such $\phi$ is an isomorphism.
\end{lemma}

\begin{proof}
If $E$ is distinguished, $\mathrm{TR3}$ extends the identity square on $c$
to a morphism of triangles $(1_{C_0},1_{C_1},\phi)$.  Its third component is
an isomorphism by \cite[Proposition~1.1.20]{Neeman2001}, and it gives the two
displayed equations.  Conversely, these equations make $h_\phi$ a morphism
between the characteristic short exact sequences of the two candidate
triangles,
with the identity on both end terms.  Hence $h_\phi$ is an isomorphism by
the short five lemma.  The Yoneda lemma then shows that $\phi$ is an
isomorphism, so $E$ is isomorphic to the distinguished triangle.
\end{proof}

\subsection{The local obstruction and composition}

We apply \cref{lem:two-arrow-recognition} to the mapping cone of an
initial square.

\begin{definition}[{\normalfont cf.~\cite[Definition~1.3.14]{Neeman2001}}]
\label{def:initial-square}
Let $\T$ be a pre-triangulated category.  Fix distinguished
triangles
\begin{equation}
\label{eq:local-rows}
 X\xrightarrow{u}Y\xrightarrow{v}Z\xrightarrow{w}\Sigma(X),
 \qquad
 X'\xrightarrow{u'}Y'\xrightarrow{v'}Z'\xrightarrow{w'}\Sigma(X')
\end{equation}
and a commutative square on their first maps,
\begin{equation}
\label{eq:initial-square}
\begin{tikzcd}[column sep=large,row sep=large]
X \arrow[r,"u"] \arrow[d,"f"'] & Y \arrow[d,"g"] \\
X'\arrow[r,"u'"'] &Y'.
\end{tikzcd}
\end{equation}
We call the data in \eqref{eq:local-rows}--\eqref{eq:initial-square} an
\emph{initial square} and denote it by $q$.  A \emph{completion} of $q$
is a morphism $h:Z\to Z'$ for which $(f,g,h)$ is a morphism of triangles.
It is \emph{good} if the mapping cone of $(f,g,h)$ is distinguished.
\end{definition}

\begin{notation}
\label{not:cone-data}
Let $\T$ be a pre-triangulated category, and let $q$ be an
initial square.  Axiom $\mathrm{TR3}$ gives at least one completion.  For any
completion $h:Z\to Z'$, denote the associated cone candidate triangle by $C(h)$ and
write it as
\begin{equation}
\label{eq:local-cone}
\xymatrix@C=2.2pc{
A\ar[r]^-{\alpha}&B\ar[r]^-{\beta(h)}&C\ar[r]^-{\gamma}&\Sigma(A).
}
\end{equation}
where
\[
 A=Y\oplus X',\qquad B=Z\oplus Y',\qquad
 C=\Sigma(X)\oplus Z'.
\]
In the notation of \eqref{eq:cone-matrices}, the three maps are
\[
 \alpha=\begin{bmatrix}-v&0\\ g&u'\end{bmatrix},\qquad
 \beta(h)=\begin{bmatrix}-w&0\\ h&v'\end{bmatrix},\qquad
 \gamma=\begin{bmatrix}-\Sigma(u)&0\\ \Sigma(f)&w'\end{bmatrix}.
\]
Thus $A,B,C,\alpha$, and $\gamma$ depend only on $q$, whereas $\beta(h)$
depends on the generally nonunique completion $h$.  Put
\[
 \beta_0=
 \begin{bmatrix}-w&0\\0&v'\end{bmatrix},
 \qquad
 J_q=\left\{
 \begin{bmatrix}0&0\\a&0\end{bmatrix}
 \,\middle|\,a\in\T(Z,Z')
 \right\}\subseteq\T(B,C).
\]
The subgroup $J_q$ consists of morphisms whose only possibly nonzero entry
is in position $(2,1)$, and it is canonically isomorphic to $\T(Z,Z')$.  We have $\beta(h)-\beta_0\in J_q$.  For $b\in\T(B,C)$,
write $\bar b$ for its class in $\T(B,C)/J_q$.
\end{notation}

Choose a distinguished triangle beginning with the fixed map $\alpha$,
\begin{equation}
\label{eq:alpha-triangle}
\xymatrix@C=2.2pc{
A\ar[r]^-{\alpha}&B\ar[r]^-{s}&E\ar[r]^-{t}&\Sigma(A).
}
\end{equation}
The candidate triangle $C(h)$ and \eqref{eq:alpha-triangle} have the same first
arrow.  A morphism $\phi:E\to C$ compares their remaining arrows through
$\phi\circ s=\beta(h)$ and $\gamma\circ\phi=t$.  Passing to the quotient
by \(J_q\) makes the first equation independent of the chosen completion
\(h\).

\begin{definition}
\label{def:octahedral-defect}
Let $\T$ be a pre-triangulated category, let $q$ be an
initial square, and use the choices above.  Define
\begin{equation}
\label{eq:boundary-map}
 \delta_q:\T(E,C)\longrightarrow
 \T(B,C)/J_q\oplus\T(E,\Sigma(A)),
 \qquad
 \phi\longmapsto(\overline{\phi\circ s},\gamma\circ\phi),
\end{equation}
and set
\begin{equation}
\label{eq:octahedral-defect}
 \operatorname{Obs}_\Delta(q)=\coker\delta_q
 =\frac{\T(B,C)/J_q\oplus\T(E,\Sigma(A))}{\operatorname{im}\delta_q}.
\end{equation}
Let
\[
 o_\Delta(q)=[(\bar\beta_0,t)]\in\operatorname{Obs}_\Delta(q)
\]
be the image under the canonical quotient map.  We call it the
\emph{local octahedral obstruction} of $q$.  The subscript records the
dependence on the pre-triangulation.  When $(\T,\Sigma,\Delta)$ is fixed, we
write simply $o(q)$.
\end{definition}

\begin{theorem}
\label{thm:boundary-lifting}
Let \(\T\) be a pre-triangulated category with suspension \(\Sigma\) and
pre-triangulation \(\Delta\), and let \(q\) be an initial square.  Then \(q\)
has a good completion if and only if \(o(q)=0\).
More precisely, any preimage $\phi:E\to C$ of
$(\bar\beta_0,t)$ under $\delta_q$ determines a unique completion $h:Z\to Z'$,
characterized by
\begin{equation}
\label{eq:recover-completion}
 \phi\circ s=\beta(h),
\end{equation}
and this completion is good.  Conversely, every good completion arises from at
least one lift.
\end{theorem}

\begin{proof}
Let $\phi:E\to C$ be a lift, so
$\overline{\phi\circ s}=\bar\beta_0$ and $\gamma\circ\phi=t$.  The first equality says
that $\phi\circ s-\beta_0$ belongs to $J_q$.  Its $(2,1)$-entry determines
a unique morphism $h:Z\to Z'$ and gives \eqref{eq:recover-completion}.

It remains to show that $h$ is a completion.  Since $s\circ\alpha=0$,
\[
 \beta(h)\circ\alpha=(\phi\circ s)\circ\alpha=0.
\]
The $(2,1)$-entry of this equality is
$-h\circ v+v'\circ g=0$.  Similarly, $\gamma\circ\phi=t$ and $t\circ s=0$ give
\[
 \gamma\circ\beta(h)=\gamma\circ\phi\circ s=t\circ s=0;
\]
its $(2,1)$-entry is $-(\Sigma(f))\circ w+w'\circ h=0$.  Hence
\[
 h\circ v=v'\circ g,\qquad w'\circ h=(\Sigma(f))\circ w,
\]
so $(f,g,h)$ is a morphism of triangles.

Its mapping cone \eqref{eq:local-cone} is Yoneda exact by
\cref{lem:cone-exact}.  The equations
$\phi\circ s=\beta(h)$ and $\gamma\circ\phi=t$ give the following morphism
of candidate triangles:
\[
\begin{tikzcd}[column sep=large]
A \arrow[r,"\alpha"] \arrow[d,equal]
  & B \arrow[r,"s"] \arrow[d,equal]
  & E \arrow[r,"t"] \arrow[d,"\phi"]
  & \Sigma(A) \arrow[d,equal]\\
A \arrow[r,"\alpha"']
  & B \arrow[r,"\beta(h)"']
  & C \arrow[r,"\gamma"']
  & \Sigma(A).
\end{tikzcd}
\]
The top row is the chosen distinguished triangle and the bottom row is the
Yoneda-exact cone candidate triangle.  By \cref{lem:two-arrow-recognition}, the
bottom row is distinguished.

Conversely, let $h$ be good.  The same recognition result supplies
$\phi:E\to C$ with $\phi\circ s=\beta(h)$ and $\gamma\circ\phi=t$, so
$\phi$ is a lift.  Thus every lift gives a good completion, and every good
completion is obtained from a lift.
\end{proof}

We next extend $\delta_q$ to a three-term complex.  This separates the roles
of \textup{(TR3)} and the mapping-cone axiom \textup{(TR4$'$)}: the former
makes the distinguished element a cocycle, while the latter makes its class
vanish.

\begin{definition}
\label{def:local-obstruction-complex}
Let \(\T\) be a pre-triangulated category with suspension \(\Sigma\) and
pre-triangulation \(\Delta\), and let \(q\) be an initial square.  Put
\[
 \gamma J_q=\{\gamma\circ j\mid j\in J_q\}\subseteq
 \T(B,\Sigma(A)).
\]
The \emph{local obstruction complex} $\mathcal O_q^\bullet$ is concentrated
in degrees $0,1,2$ and is given by
\begin{equation*}
\xymatrix@C=2.2pc{
\mathcal O_q^0\ar[r]^{d^0}&\mathcal O_q^1\ar[r]^{d^1}&\mathcal O_q^2.
}
\end{equation*}
where
\[
 \begin{aligned}
 \mathcal O_q^0&=\T(E,C),\\
 \mathcal O_q^1&=\T(B,C)/J_q\oplus\T(E,\Sigma(A)),\\
 \mathcal O_q^2&=\T(B,\Sigma(A))/(\gamma J_q),
 \end{aligned}
 \qquad
 \begin{aligned}
 d^0(\phi)&=(\overline{\phi\circ s},\gamma\circ\phi),\\
 d^1(\bar b,r)&=[\gamma\circ b-r\circ s].
 \end{aligned}
\]
All other terms are zero.
\end{definition}

The formula for $d^1$ is independent of the representative of $\bar b$:
replacing $b$ by $b+j$ changes $\gamma\circ b-r\circ s$ by an element of
$\gamma J_q$.  Moreover,
\[
 (d^1\circ d^0)(\phi)
 =[\gamma\circ\phi\circ s-(\gamma\circ\phi)\circ s]=0,
\]
so $\mathcal O_q^\bullet$ is indeed a cochain complex.

By $\mathrm{TR3}$, the initial square has a completion $h$.  Its cone
candidate triangle is a complex, so $\gamma\circ\beta(h)=0$.  Since
$\beta(h)-\beta_0\in J_q$ and $t\circ s=0$, the element
\[
 x_q=(\bar\beta _0,t)\in\mathcal O_q^1
\]
is a cocycle.

\begin{definition}
\label{def:cohomological-local-obstruction}
Let \(\T\) be a pre-triangulated category with suspension \(\Sigma\) and
pre-triangulation \(\Delta\), and let \(q\) be an initial square.  The class
\[
 \omega(q)=[x_q]\in H^1(\mathcal O_q^\bullet)
\]
is the \emph{cohomological local obstruction} of $q$.
\end{definition}

\begin{proposition}
\label{prop:local-obstruction-complex}
Let \(\T\) be a pre-triangulated category with suspension \(\Sigma\) and
pre-triangulation \(\Delta\), and let \(q\) be an initial square.  Then
\(\omega(q)=0\) if and only if \(q\) has a good
completion.  Moreover, the inclusion
$\ker d^1\subseteq\mathcal O_q^1$ induces a canonical injection
\[
 H^1(\mathcal O_q^\bullet)\lhook\joinrel\longrightarrow
 \coker\delta_q
\]
which carries $\omega(q)$ to $o(q)$.
\end{proposition}

\begin{proof}
The class of $x_q$ is zero precisely when
$x_q=d^0(\phi)=\delta_q(\phi)$ for some $\phi:E\to C$.  By
\cref{thm:boundary-lifting}, this is equivalent to the existence of a good
completion.  Finally,
\[
 H^1(\mathcal O_q^\bullet)
   =\ker d^1/\operatorname{im}d^0
   \subseteq \mathcal O_q^1/\operatorname{im}\delta_q,
\]
which proves the asserted injection and identifies the two distinguished
classes.
\end{proof}

The construction depends on the distinguished triangle chosen to complete
$\alpha$.  The following proposition shows that the resulting isomorphism
classes are independent of this choice.

\begin{proposition}
\label{prop:local-choice-independence}
Let \(\T\) be a pre-triangulated category with suspension \(\Sigma\) and
pre-triangulation \(\Delta\), and let \(q\) be an initial square.  Any two
distinguished triangles chosen to complete
$\alpha$ yield isomorphisms of the pairs
\[
 (\operatorname{Obs}_\Delta(q),o_\Delta(q))
 \quad\text{and}\quad
 (H^1(\mathcal O_q^\bullet),\omega(q))
\]
which carry each distinguished class to its counterpart.  In particular,
the isomorphism classes of these pairs are independent of the choice.
\end{proposition}

\begin{proof}
Let $A\to B\xto{s'}E'\xto{t'}\Sigma(A)$ be another distinguished triangle.
The identity square on $\alpha$ extends to an isomorphism
$\theta:E\to E'$ satisfying $\theta\circ s=s'$ and $t'\circ\theta=t$.
Precomposition with $\theta^{-1}$ identifies the homomorphisms $\delta_q$
obtained from the two choices and
induces an isomorphism of the cokernels carrying $o_\Delta(q)$ to its
counterpart.

In degrees (0,1,2), respectively, the maps
\[
 \phi\longmapsto\phi\circ\theta^{-1},\qquad
 (\bar b,r)\longmapsto(\bar b,r\circ\theta^{-1}),\qquad
 [z]\longmapsto[z]
\]
form an isomorphism of obstruction complexes.  They carry
$(\bar\beta_0,t)$ to $(\bar\beta_0,t')$, and hence induce the required
isomorphism on $H^1$, carrying $\omega(q)$ to its counterpart.
\end{proof}

We now specialize the local obstruction to composable morphisms.  Let $\T$
be a pre-triangulated category.  For composable maps
\begin{equation*}
\xymatrix@C=2.4pc{
X\ar[r]^-{f}&Y\ar[r]^-{g}&Z.
}
\end{equation*}
fix distinguished triangles completing $f$ and $g\circ f$:
\begin{equation}
\label{eq:composition-boundary-triangles}
\xymatrix@C=1.55pc{X\ar[r]^-{f}&Y\ar[r]^-{a}&F\ar[r]^-{b}&\Sigma(X)}
\qquad
\xymatrix@C=1.55pc{X\ar[r]^-{g\circ f}&Z\ar[r]^-{c}&H\ar[r]^-{d}&\Sigma(X)}.
\end{equation}
Together with the commutative square
\begin{equation}
\label{eq:composition-square}
\begin{tikzcd}[column sep=large,row sep=large]
X \arrow[r,"f"] \arrow[d,equal] & Y \arrow[d,"g"]\\
X \arrow[r,"g\circ f"'] & Z.
\end{tikzcd}
\end{equation}
these triangles form the composition square $q(g\circ f)$ associated with
the displayed pair $(f,g)$.
Different choices are isomorphic by maps which are the identity on the first
two terms.  Transport along these isomorphisms preserves good completions.
Consequently the vanishing of $o(q(g\circ f))$ is independent of the choices.

We shall use the following classical composition form of $\mathrm{TR4}$.
For every composable pair $X\xrightarrow{f}Y\xrightarrow{g}Z$ and every choice
of \eqref{eq:composition-boundary-triangles}, there is $h:F\to H$ with
\[
 h\circ a=c\circ g,\qquad d\circ h=b,
\]
such that
\begin{equation}
\label{eq:homotopy-cartesian-composition}
\xymatrix@C=3.0pc{
Y\ar[r]^-{\left[\begin{smallmatrix}-a\\g\end{smallmatrix}\right]}&
F\oplus Z\ar[r]^-{[h\ \ c]}&H\ar[r]^-{(\Sigma(f))\circ d}&\Sigma(Y).
}
\end{equation}
is distinguished.  Under $\mathrm{TR1}$--$\mathrm{TR3}$ this is equivalent
to Verdier's octahedral axiom; see \cite[Theorem~1.8]{Neeman1991} and
\cite[Proposition~1.4.6]{Neeman2001}.

\begin{theorem}
\label{thm:composition-obstruction}
Let $\T$ be a pre-triangulated category.  The following are
equivalent:
\begin{enumerate}[label=\textup{(\roman*)}]
\item \(\Delta\) satisfies \(\mathrm{TR4}\);
\item every square \eqref{eq:composition-square} has a good completion;
\item $o(q(g\circ f))=0$ for every composable pair $(f,g)$.
\end{enumerate}
\end{theorem}

\begin{proof}
Use the notation of \eqref{eq:composition-boundary-triangles}.  For a
completion $h:F\to H$, denote its mapping-cone candidate triangle by $M(h)$ and the
candidate triangle in \eqref{eq:homotopy-cartesian-composition} by $T(h)$.  Explicitly,
$M(h)$ is
\begin{equation*}
\xymatrix@C=3.5pc{
Y\oplus X\ar[r]^-{\left[\begin{smallmatrix}-a&0\\g&g\circ f\end{smallmatrix}\right]}&
F\oplus Z\ar[r]^-{\left[\begin{smallmatrix}-b&0\\h&c\end{smallmatrix}\right]}&
\Sigma(X)\oplus H\ar[r]^-{\left[\begin{smallmatrix}-\Sigma(f)&0\\1&d\end{smallmatrix}\right]}&
\Sigma(Y)\oplus\Sigma(X).
}
\end{equation*}
Apply the isomorphisms
\begin{equation}
\label{eq:composition-cone-change}
 R_A=\begin{bmatrix}1_Y&f\\0&1_X\end{bmatrix},\qquad
 R_B=1_{F\oplus Z},\qquad
 R_C=\begin{bmatrix}0&1_H\\1_{\Sigma(X)}&d\end{bmatrix},\qquad
 R_{\Sigma(A)}=\Sigma(R_A).
\end{equation}
Write the three cone maps as $D_0,D_1,D_2$.  Since
$a\circ f=0$, $d\circ c=0$, $h\circ a=c\circ g$, and
$d\circ h=b$, direct multiplication gives
\[
\begin{aligned}
 R_B\circ D_0\circ R_A^{-1}
   &=\begin{bmatrix}-a&0\\ g&0\end{bmatrix},\\
 R_C\circ D_1\circ R_B^{-1}
   &=\begin{bmatrix}h&c\\0&0\end{bmatrix},\\
 R_{\Sigma(A)}\circ D_2\circ R_C^{-1}
   &=\begin{bmatrix}(\Sigma(f))\circ d&0\\0&1_{\Sigma(X)}\end{bmatrix}.
\end{aligned}
\]
Thus the transformed cone is the direct sum of
$T(h)$ on the first summands and
\begin{equation*}
\xymatrix@C=2.2pc{
X\ar[r]&0\ar[r]&\Sigma(X)\ar[r]^{1_{\Sigma(X)}}&\Sigma(X).
}
\end{equation*}
on the complementary summands.  In other words,
\[
 M(h)\cong T(h)\oplus
 \bigl(X\to0\to\Sigma(X)\xrightarrow{1_{\Sigma(X)}}\Sigma(X)\bigr)
\]
as candidate triangles.  The complementary summand is isomorphic to a
rotation of the split triangle on $1_X$, and is therefore distinguished.
Consequently,
\cite[Propositions~1.2.1 and~1.2.3]{Neeman2001} imply that $M(h)$ is
distinguished if and only if $T(h)$ is distinguished.

Assume \textup{(i)}.  The classical composition form of $\mathrm{TR4}$
provides a completion $h$ for which $T(h)$ is distinguished.  Hence $M(h)$
is distinguished and $q(g\circ f)$ has a good completion.  This proves
\textup{(i)}$\Rightarrow$\textup{(ii)}.  Conversely, under \textup{(ii)} a
good completion makes $M(h)$, and therefore $T(h)$, distinguished.  The
same classical composition criterion gives $\mathrm{TR4}$, proving
\textup{(ii)}$\Rightarrow$\textup{(i)}.  Finally,
\textup{(ii)}$\Leftrightarrow$\textup{(iii)} is
\cref{thm:boundary-lifting}.
\end{proof}

For a composition square, an additional endomorphism in the $(2,1)$-entry
of the third mapping-cone morphism can be removed by changing the completion.
More precisely:

\begin{corollary}
\label{cor:cone-shear-removal}
Use the notation of \eqref{eq:composition-boundary-triangles}.  For a
morphism \(h:F\to H\) and a morphism
\(\lambda:\Sigma(X)\to\Sigma(X)\), put
\[
 \alpha_q=
 \begin{bmatrix}-a&0\\g&g\circ f\end{bmatrix},\qquad
 \beta(h)=
 \begin{bmatrix}-b&0\\h&c\end{bmatrix},\qquad
 \gamma_\lambda=
 \begin{bmatrix}-\Sigma(f)&0\\1_{\Sigma(X)}+\lambda&d\end{bmatrix}.
\]
Then \(q(g\circ f)\) has a good completion if and only if there exist a
completion \(h\) and a morphism \(\lambda\) for which the candidate triangle
\begin{equation}
\label{eq:lower-left-sheared-cone}
\xymatrix@C=2.0pc{
Y\oplus X\ar[r]^-{\alpha_q}&F\oplus Z\ar[r]^-{\beta(h)}&
\Sigma(X)\oplus H\ar[r]^-{\gamma_\lambda}&
\Sigma(Y)\oplus\Sigma(X)
}
\end{equation}
is distinguished.
\end{corollary}

\begin{proof}
A good completion gives \eqref{eq:lower-left-sheared-cone} with
\(\lambda=0\).  Conversely, suppose that the displayed candidate triangle is
distinguished.  The $(2,1)$-entry of
\(\Sigma(\alpha_q)\circ\gamma_\lambda=0\) gives
\[
 \Sigma(g\circ f)\circ\lambda=0.
\]
Exactness of \(\T(\Sigma(X),-)\) on the second triangle in
\eqref{eq:composition-boundary-triangles} therefore gives
\(c_0:\Sigma(X)\to H\) such that \(d\circ c_0=\lambda\).  Put
\[
 \phi=
 \begin{bmatrix}1_{\Sigma(X)}&0\\c_0&1_H\end{bmatrix},
 \qquad
 \gamma=
 \begin{bmatrix}-\Sigma(f)&0\\1_{\Sigma(X)}&d\end{bmatrix}.
\]
Direct multiplication gives
\[
 \gamma\circ\phi=\gamma_\lambda,
 \qquad
 \phi\circ\beta(h)=\beta(h-c_0\circ b).
\]
Take \eqref{eq:lower-left-sheared-cone} as the chosen distinguished
triangle beginning with \(\alpha_q\).  In the notation of
\cref{sec:obstruction}, the displayed identities say
\[
 \delta_q(\phi)=(\bar\beta_0,\gamma_\lambda).
\]
Thus \cref{thm:boundary-lifting} identifies \(h-c_0\circ b\) as a good
completion of \(q(g\circ f)\).
\end{proof}

\section{Relative stable comparisons and exact lifting}
\label{sec:relative-heller}

For \(i=0,1\), equip the same additive category \(\T\) with a suspension
\(\Sigma_i\) and a pre-triangulation \(\Delta_i\), and put
\[
 \mathcal E=\mathsf{mod}\,\T,
 \qquad
 \underline{\mathcal E}=\underline{\mathsf{mod}}\,\T.
\]
The inverse Heller comparisons recalled in
\cref{sec:preliminaries} determine a stable comparison between the two
structures.  We characterize \(\mathrm{TR4}\) for the second structure
by exact liftability of this relative comparison.

\subsection{The relative comparison and exact lifting}
\begin{definition}
\label{def:relative-comparison}
For \(i=0,1\), equip the additive category \(\T\) with a suspension
\(\Sigma_i\) and a pre-triangulation \(\Delta_i\).  Put
\[
 \mathsf F_i=(\Sigma_i^{-1})_*\qquad(i=0,1),
\]
and let \(\theta_i:\Omega^3\xrightarrow{\sim}\mathsf F_i\) be the
comparisons of \cref{prop:intrinsic-inverse-comparison}.  The
\emph{relative comparison} from \(\Delta_0\) to \(\Delta_1\) is the stable
natural isomorphism
\begin{equation}
\label{eq:relative-heller-comparison}
 \tau=\theta_1\circ\theta_0^{-1}:
 \mathsf F_0\xrightarrow{\sim}\mathsf F_1.
\end{equation}
\end{definition}

\begin{definition}
\label{def:exact-lift}
Let
\begin{equation*}
\xymatrix@C=2.2pc{
e:\quad0\ar[r]&A\ar[r]^-{i}&B\ar[r]^-{p}&C\ar[r]&0.
}
\end{equation*}
be short exact in \(\mathcal E\).  An \emph{exact lift of \(\tau\) along
\(e\)} is a morphism of short exact sequences
\begin{equation}
\label{eq:exact-lift-diagram}
\begin{tikzcd}[column sep=large]
0\arrow[r]&\mathsf F_0(A)\arrow[r,"\mathsf F_0(i)"]\arrow[d,"t_A"']&
\mathsf F_0(B)\arrow[r,"\mathsf F_0(p)"]\arrow[d,"t_B"']&
\mathsf F_0(C)\arrow[r]\arrow[d,"t_C"']&0\\
0\arrow[r]&\mathsf F_1(A)\arrow[r,"\mathsf F_1(i)"']&
\mathsf F_1(B)\arrow[r,"\mathsf F_1(p)"']&\mathsf F_1(C)\arrow[r]&0
\end{tikzcd}
\end{equation}
such that \(\underline{t_X}=\tau_X\) for \(X\in\{A,B,C\}\).  We say that
\(\tau\) \emph{admits exact lifts} if it admits one along every short exact
sequence in \(\mathcal E\).
\end{definition}

Exact liftability is invariant under isomorphism of short exact sequences.
Indeed, if \(\rho:e'\xrightarrow{\sim}e\), conjugation by
\(\mathsf F_i(\rho)\) transports exact lifts along \(e\) to exact lifts
along \(e'\).

\begin{definition}
\label{def:square-extension}
Let \(q\) be the initial square
\eqref{eq:local-rows}--\eqref{eq:initial-square}, and retain the notation
\(u,v,u',v',f,g\) used there.  Set
\[
 \alpha_q=
 \begin{bmatrix}-v&0\\ g&u'\end{bmatrix}:
 Y\oplus X'\longrightarrow Z\oplus Y'.
\]
For every \(T\in\T\),
\[
 B(\alpha_q)(T)=
 \bigl\{(y,x')\in\T(T,Y)\oplus\T(T,X')\mid
 v\circ y=0,\ g\circ y+u'\circ x'=0\bigr\}.
\]
Define natural transformations
\[
 \iota_q:B(u')\longrightarrow B(\alpha_q),
 \qquad
 \pi_q:B(\alpha_q)\longrightarrow B(v)
\]
by
\[
 \iota_q(T)(x')=(0,x'),
 \qquad
 \pi_q(T)(y,x')=y.
\]
We write
\begin{equation}
\label{eq:square-extension}
 e_q:\quad
 0\longrightarrow B(u')\xrightarrow{\iota_q}B(\alpha_q)
 \xrightarrow{\pi_q}B(v)\longrightarrow0
\end{equation}
for the resulting sequence, called the \emph{square extension} of \(q\).
\end{definition}

\begin{remark}
\label{rem:square-extension}
The sequence \eqref{eq:square-extension} is exact.  Indeed, let
\(y\in B(v)(T)\).  Exactness of the first distinguished triangle in
\eqref{eq:local-rows} gives \(y=u\circ x\) for some \(x:T\to X\), and then
\((y,-f\circ x)\in B(\alpha_q)(T)\).  Thus \(\pi_q(T)\) is surjective, while
the pointwise description of \(B(\alpha_q)(T)\) gives
\(\ker\pi_q(T)=\operatorname{im}\iota_q(T)\).

\end{remark}

\begin{theorem}
\label{thm:relative-heller-criterion}
For \(i=0,1\), equip the additive category \(\T\) with a suspension
\(\Sigma_i\) and a pre-triangulation \(\Delta_i\).  Assume that
\(\Delta_0\) satisfies
\(\mathrm{TR4}\).  The following are equivalent.
\begin{enumerate}[label=\textup{(\roman*)}]
 \item \(\Delta_1\) satisfies \(\mathrm{TR4}\).
 \item The relative comparison
       \(\tau:\mathsf F_0\xrightarrow{\sim}\mathsf F_1\) admits an exact
       lift along every short exact sequence in \(\mathcal E\).
\end{enumerate}
\end{theorem}

\subsection{Proof of the relative criterion}

\begin{definition}
\label{def:ses-category}
We define \(\mathsf{SES}(\mathcal E)\) to be the category whose objects
are the short exact sequences
\[
 0\longrightarrow A\longrightarrow B\longrightarrow C\longrightarrow0
\]
in \(\mathcal E\).  A morphism is a commutative diagram between two such
sequences.  We equip \(\mathsf{SES}(\mathcal E)\) with the componentwise
exact structure: a composable pair of morphisms is a conflation precisely
when its sequences of left, middle, and right terms are short exact in
\(\mathcal E\).
\end{definition}

A split short exact sequence of projective-injective objects is
projective-injective in \(\mathsf{SES}(\mathcal E)\).  Indeed, it is
isomorphic to a direct sum of sequences of the forms
\[
 0\longrightarrow I\xrightarrow{1_I}I\longrightarrow0\longrightarrow0,
 \qquad
 0\longrightarrow0\longrightarrow J\xrightarrow{1_J}J\longrightarrow0,
\]
where \(I,J\) are projective-injective in \(\mathcal E\).
The lifting and extension properties for these sequences follow from
projectivity and injectivity of \(I\) and \(J\), respectively.
Thus a commutative
three-row diagram with exact rows and columns, whose internal columns have
this form, gives a partial projective resolution in
\(\mathsf{SES}(\mathcal E)\) when read by columns.

\begin{lemma}
\label{lem:standard-three-row}
Let \(\T\) be a pre-triangulated category.  Denote its suspension by
\(\Sigma\) and its pre-triangulation by \(\Delta\).  Let \(q\) be the initial square
\eqref{eq:local-rows}--\eqref{eq:initial-square}, and let
\(h:Z\to Z'\) be a completion.  Put
\[
 L=B(u'),\qquad M=B(\alpha_q),\qquad R=B(v).
\]
The cone matrices induce the following commutative diagram with exact rows
and columns:
\begin{equation}
\label{eq:standard-heller-square}
\begin{tikzcd}[column sep=1.8em,row sep=1.8em]
0 \arrow[r] &
 L \arrow[r,"j_L"] \arrow[d,"\iota_q"'] &
 h_{X'} \arrow[r,"h_{u'}"] \arrow[d,hook] &
 h_{Y'} \arrow[r,"h_{v'}"] \arrow[d,hook] &
 h_{Z'} \arrow[r,"\varepsilon_L"] \arrow[d,hook] &
 \Sigma_*(L) \arrow[r] \arrow[d,"\Sigma_*(\iota_q)"] &0\\
0 \arrow[r] &
 M \arrow[r,"j_M"] \arrow[d,"\pi_q"'] &
 h_{Y\oplus X'} \arrow[r,"h_{\alpha_q}"] \arrow[d,two heads] &
 h_{Z\oplus Y'} \arrow[r,"h_{\beta(h)}"] \arrow[d,two heads] &
 h_{\Sigma(X)\oplus Z'} \arrow[r,"\varepsilon_M"] \arrow[d,two heads] &
 \Sigma_*(M) \arrow[r] \arrow[d,"\Sigma_*(\pi_q)"] &0\\
0 \arrow[r] &
 R \arrow[r,"j_R"] & h_Y \arrow[r,"h_{-v}"] &
 h_Z \arrow[r,"h_{-w}"] & h_{\Sigma(X)} \arrow[r,"\varepsilon_R"] &
 \Sigma_*(R) \arrow[r] &0.
\end{tikzcd}
\end{equation}
The top and bottom rows are Heller resolutions of \(\Sigma_*(L)\) and \(\Sigma_*(R)\),
respectively.
The middle row is a representable exact row ending in \(\Sigma_*(M)\).  Denote its
stable connecting morphism by
\[
 \kappa_h:\Omega^3(\Sigma_*(M))\longrightarrow M.
\]
Then
\begin{equation}
\label{eq:middle-row-recognition}
 h\text{ is a good completion}
 \quad\Longleftrightarrow\quad
 \kappa_h=(\theta_\Delta)_{\Sigma_*(M)}.
\end{equation}
\end{lemma}

\begin{proof}
The three internal columns are the canonical split sequences associated with
the displayed biproducts, and direct multiplication of the cone matrices
shows that all internal squares commute.  The two outer columns are
\(e_q\) and \(\Sigma_*(e_q)\), which are exact by
\cref{rem:square-extension}.  The top and bottom rows are the representable
resolutions associated with the two distinguished triangles in
\eqref{eq:local-rows}.  The middle row is exact because the cone candidate triangle is
Yoneda exact by \cref{lem:cone-exact}.  By construction, the candidate triangle recovered from the middle row by
\cref{prop:intrinsic-inverse-comparison}\textup{(ii)} is the cone candidate triangle
\(C(h)\).  The assertions about stable connecting morphisms, including
\eqref{eq:middle-row-recognition}, follow.
\end{proof}

Let \(q=q(g\circ f)\) be the composition square of
\eqref{eq:composition-square}, and fix a completion \(h\).
In the notation of \cref{lem:standard-three-row}, we have
\[
 \alpha_q=\begin{bmatrix}-a&0\\ g&g\circ f\end{bmatrix},
 \qquad L=B(g\circ f),\quad M=B(\alpha_q),\quad R=B(a).
\]
Thus
\begin{equation}
\label{eq:relative-composition-extension}
 e_q:\quad0\longrightarrow L\xrightarrow{\iota_q}M
 \xrightarrow{\pi_q}R\longrightarrow0.
\end{equation}
The two fixed distinguished triangles are
\eqref{eq:composition-boundary-triangles}.
Given \(\xi:R\to L\), set
\[
 \varphi_M=1_M+\iota_q\circ\xi\circ\pi_q.
\]
Since \(\pi_q\circ\iota_q=0\), the triple
\(\varphi=(1_L,\varphi_M,1_R)\) is an automorphism of \(e_q\), and
\[
 \varphi_M^{-1}=1_M-\iota_q\circ\xi\circ\pi_q.
\]
Let \(E_h^\varphi\) denote the lower row in the commutative diagram
\begin{equation}
\label{eq:middle-row-change}
\begin{tikzcd}[column sep=1.6em,row sep=3em]
0\arrow[r]&M\arrow[r,"j_M"]\arrow[d,"\varphi_M"']&
h_{Y\oplus X}\arrow[r]\arrow[d,equal]&h_{F\oplus Z}\arrow[r]\arrow[d,equal]&
h_{\Sigma(X)\oplus H}\arrow[r]\arrow[d,equal]&\Sigma_*(M)\arrow[r]\arrow[d,equal]&0\\
0\arrow[r]&M\arrow[r,"j_M\circ\varphi_M^{-1}"']&
h_{Y\oplus X}\arrow[r]&h_{F\oplus Z}\arrow[r]&h_{\Sigma(X)\oplus H}\arrow[r]&\Sigma_*(M)\arrow[r]&0.
\end{tikzcd}
\end{equation}
By \cref{prop:intrinsic-inverse-comparison}\textup{(ii)}, the exact row
\(E_h^\varphi\) determines a Yoneda-exact candidate triangle
\begin{equation}
\label{eq:middle-row-candidate}
\begin{tikzcd}[column sep=2.2em]
Y\oplus X\arrow[r,"\alpha_q"]&
F\oplus Z\arrow[r,"\beta(h)"]&
\Sigma(X)\oplus H\arrow[r,"\gamma_\varphi"]&
\Sigma(Y)\oplus\Sigma(X).
\end{tikzcd}
\end{equation}
We denote this candidate triangle by \(C_h^\varphi\).

\begin{lemma}
\label{lem:middle-row-correction}
The stable connecting morphism of \(E_h^\varphi\) is
\[
 \underline{\varphi_M}\circ\kappa_h:
 \Omega^3(\Sigma_*(M))\longrightarrow M.
\]
Moreover, there is a morphism
\(\lambda:\Sigma(X)\to\Sigma(X)\) for which
\begin{equation}
\label{eq:corrected-third-cone-map}
 \gamma_\varphi=
 \begin{bmatrix}
  -\Sigma(f)&0\\
  1_{\Sigma(X)}+\lambda&d
 \end{bmatrix}.
\end{equation}
\end{lemma}

\begin{proof}
The vertical maps in \eqref{eq:middle-row-change} give a morphism from the
upper exact row to \(E_h^\varphi\).  Naturality of stable connecting
morphisms gives the first assertion.  Since the three projective terms and
the epimorphism to \(\Sigma_*(M)\) are unchanged, the first two morphisms in
\eqref{eq:middle-row-candidate} are \(\alpha_q\) and \(\beta(h)\).

Let
\[
 p=\begin{bmatrix}1_{\Sigma(X)}&0\end{bmatrix}:
   \Sigma(X)\oplus H\longrightarrow\Sigma(X),
 \qquad
 i=\begin{bmatrix}0\\1_{\Sigma(X)}\end{bmatrix}:
   \Sigma(X)\longrightarrow\Sigma(Y)\oplus\Sigma(X).
\]
Using the two outer columns of \eqref{eq:standard-heller-square} in the
construction of \cref{prop:intrinsic-inverse-comparison}\textup{(ii)} gives
\[
 h_{\gamma_\varphi-\gamma}
 =-h_i\circ\Sigma_*(j_L)\circ\Sigma_*(\xi)
   \circ\varepsilon_R\circ h_p,
\]
where
\[
 \gamma=
 \begin{bmatrix}
  -\Sigma(f)&0\\ 1_{\Sigma(X)}&d
 \end{bmatrix}
\]
is the third morphism of the original cone candidate triangle.  By the Yoneda lemma
there is a unique \(\lambda:\Sigma(X)\to\Sigma(X)\) such that
\[
 h_\lambda
 =-\Sigma_*(j_L)\circ\Sigma_*(\xi)\circ\varepsilon_R.
\]
Hence \(\gamma_\varphi-\gamma=i\circ\lambda\circ p\), which is precisely
\eqref{eq:corrected-third-cone-map}.
\end{proof}

\begin{lemma}
\label{lem:realize-short-exact-sequence}
Let \(\T\) be a pre-triangulated category.  Denote its suspension by
\(\Sigma\) and its pre-triangulation by \(\Delta\).  Every short exact sequence in
\(\mathcal E\) is isomorphic to \(\Sigma_*(e_q)\) for some initial square
\(q\).
\end{lemma}

\begin{proof}
Let
\[
 e:\quad0\longrightarrow A\longrightarrow B\longrightarrow C
 \longrightarrow0
\]
be short exact, and put \(\mathsf F_\Delta=(\Sigma^{-1})_*\).  As observed
before \cref{prop:intrinsic-inverse-comparison}, every object
of \(\mathcal E\) is isomorphic to \(B(c)\) for a morphism \(c\) of \(\T\).
Choose distinguished triangles
\begin{equation}
\label{eq:realization-two-triangles}
 X\xrightarrow{u}Y\xrightarrow{v}Z\xrightarrow{w}\Sigma(X),
 \qquad
 X'\xrightarrow{u'}Y'\xrightarrow{v'}Z'\xrightarrow{w'}\Sigma(X')
\end{equation}
such that \(\mathsf F_\Delta(A)\cong B(u')\) and
\(\mathsf F_\Delta(C)\cong B(v)\).  For \(C\), a rotation of a distinguished
completion places the realizing morphism in the second position \(v\).
Transporting \(\mathsf F_\Delta(e)\) along
these isomorphisms on the end terms, write it as
\begin{equation}
\label{eq:transported-short-exact-sequence}
 0\longrightarrow B(u')\longrightarrow N\longrightarrow B(v)
 \longrightarrow0.
\end{equation}

The first triangle in \eqref{eq:realization-two-triangles} yields
\[
 0\longrightarrow B(u)\xrightarrow{j_u}h_X
 \xrightarrow{q_u}B(v)\longrightarrow0.
\]
Pushout along a morphism \(B(u)\to B(u')\) induces the connecting
isomorphism
\begin{equation}
\label{eq:square-connecting-isomorphism}
 \partial:
 \underline{\Hom}_{\mathcal E}(B(u),B(u'))
 \xrightarrow{\sim}\Ext^1_{\mathcal E}(B(v),B(u')),
\end{equation}
because \(h_X\) is projective-injective.  Choose a representative
\(\bar f:B(u)\to B(u')\) satisfying
\begin{equation}
\label{eq:realization-extension-class}
 -\partial(\underline{\bar f})=[\mathsf F_\Delta(e)].
\end{equation}

Since \(h_{X'}\) is injective, the composite
\(B(u)\xrightarrow{\bar f}B(u')\hookrightarrow h_{X'}\) extends across
\(j_u\) to a map \(h_f:h_X\to h_{X'}\), induced by a morphism
\(f:X\to X'\).  The map \(h_{u'}\circ h_f\) vanishes on \(B(u)\), hence
factors through \(q_u\) as a morphism \(\rho:B(v)\to h_{Y'}\).  Since
\(h_{Y'}\) is injective, \(\rho\) extends across
\(B(v)\hookrightarrow h_Y\) to a map \(h_g:h_Y\to h_{Y'}\), induced by
\(g:Y\to Y'\).  Thus
\[
 h_{u'}\circ h_f=h_g\circ h_u,
\]
and faithfulness of the Yoneda embedding gives
\(u'\circ f=g\circ u\).  Hence \((f,g)\) is an initial square \(q\).

Define \(\psi:h_X\to B(\alpha_q)\) by
\[
 \psi_T(x)=(u\circ x,-f\circ x).
\]
Then \(\psi\circ j_u=\iota_q\circ(-\bar f)\) and
\(\pi_q\circ\psi=q_u\), so there is a commutative diagram
\[
\begin{tikzcd}[column sep=2.4em,row sep=1.35em]
0\arrow[r]&B(u)\arrow[r,"j_u"]\arrow[d,"-\bar f"']&
 h_X\arrow[r,"q_u"]\arrow[d,"\psi"]&
 B(v)\arrow[r]\arrow[d,equal]&0\\
0\arrow[r]&B(u')\arrow[r,"\iota_q"']&
 B(\alpha_q)\arrow[r,"\pi_q"']&B(v)\arrow[r]&0.
\end{tikzcd}
\]
The induced map from the pushout of the top row along \(-\bar f\) to the
bottom row is the identity on the two end terms, and hence is an isomorphism.
Therefore
\[
 [e_q]=-\partial(\underline{\bar f})=[\mathsf F_\Delta(e)].
\]
Thus \(\mathsf F_\Delta(e)\cong e_q\).  Applying \(\Sigma_*\) gives
\(e\cong\Sigma_*(e_q)\).
\end{proof}

For \(i=0,1\) and an initial square \(q_i\) for \(\Delta_i\), write
\begin{equation}
\label{eq:shifted-composition-extension}
 e_{q_i}^{\,i}:=\Sigma_{i*}(e_{q_i}).
\end{equation}
For a composition square \(q\) for \(\Delta_1\), we retain
\(L,M,R\) from \eqref{eq:relative-composition-extension}.
Thus \(e_q^1\) has terms \(\Sigma_{1*}(L),\Sigma_{1*}(M),\Sigma_{1*}(R)\).

\begin{proposition}
\label{prop:exact-lift-kills-obstruction}
Assume that \(\Delta_0\) satisfies \(\mathrm{TR4}\), and let
\(q=q(g\circ f)\) be a composition square for \(\Delta_1\).  If the relative
comparison \(\tau\) admits an exact lift along \(e_q^1\), then \(q\) has a
good completion.
\end{proposition}

\begin{proof}
Choose a completion \(h\) of \(q\), and let
\(\kappa_h:\Omega^3(\Sigma_{1*}(M))\to M\) be the stable connecting morphism of the
middle row of \eqref{eq:standard-heller-square}.

Apply \cref{lem:realize-short-exact-sequence} to \(e_q^1\) with respect to
\(\Delta_0\).  We obtain an initial square \(q_0\) and an isomorphism
\[
 \rho:e_{q_0}^0\xrightarrow{\sim}e_q^1.
\]
Since \(\Delta_0\) satisfies \(\mathrm{TR4}\), choose a good completion of
\(q_0\) and transport its standard three-row diagram along \(\rho\).  Reading
this diagram and the \(\Delta_1\)-diagram for \(h\) by columns gives two
partial projective resolutions of \(e_q^1\) in
\(\mathsf{SES}(\mathcal E)\).  Using projectivity of the internal
columns, we construct, successively from right to left, a morphism over
\(1_{e_q^1}\):
\begin{equation}
\label{eq:composition-comparison-diagram}
\begin{tikzcd}[column sep=1.15em,row sep=1.3em]
0\arrow[r]&e_q\arrow[r]\arrow[d,"u"']&
\mathcal Q_2\arrow[r]\arrow[d,"u_2"]&\mathcal Q_1\arrow[r]\arrow[d,"u_1"]&
\mathcal Q_0\arrow[r]\arrow[d,"u_0"]&e_q^1\arrow[r]\arrow[d,equal]&0\\
0\arrow[r]&\mathsf F_0(e_q^1)\arrow[r]&
\mathcal P_2\arrow[r]&\mathcal P_1\arrow[r]&\mathcal P_0\arrow[r]&e_q^1\arrow[r]&0.
\end{tikzcd}
\end{equation}
Here \(\mathcal Q_j\) and \(\mathcal P_j\) are the split internal columns of the two standard
diagrams.

Let
\[
 t:\mathsf F_0(e_q^1)\longrightarrow
   \mathsf F_1(e_q^1)=e_q
\]
be the given exact lift, and put
\[
 r=(r_L,r_M,r_R):=t\circ u:e_q\longrightarrow e_q.
\]
Naturality of stable connecting morphisms in
\eqref{eq:composition-comparison-diagram}, followed by
\(\underline t=\tau=\theta_1\circ\theta_0^{-1}\), gives
\begin{equation}
\label{eq:exact-lift-closes-comparison}
 \underline{r_N}\circ(\theta_1)_{\Sigma_{1*}(N)}=(\theta_1)_{\Sigma_{1*}(N)}
 \quad(N=L,R),\qquad
 \underline{r_M}\circ\kappa_h=(\theta_1)_{\Sigma_{1*}(M)}.
\end{equation}
Hence \(\underline{r_L}=1_L\) and \(\underline{r_R}=1_R\).

Choose factorizations of \(r_L-1_L\) and \(r_R-1_R\) through
projective-injective objects \(I_L\) and \(I_R\), respectively.  Injectivity
of \(I_L\) and projectivity of \(I_R\) extend them to morphisms
of short exact sequences whose composite is
\[
 e_q\longrightarrow
 \bigl(0\longrightarrow I_L\longrightarrow I_L\oplus I_R
 \longrightarrow I_R\longrightarrow0\bigr)
 \longrightarrow e_q
\]
and has components on the end terms \(r_L-1_L\) and \(r_R-1_R\).  Subtracting
this composite from \(r\), we obtain an endomorphism
\[
 \varphi=(1_L,\varphi_M,1_R):e_q\longrightarrow e_q
 \qquad\text{with}\qquad
 \underline{\varphi_M}=\underline{r_M}.
\]
Since \(\varphi\) is the identity on the end terms, exactness of \(e_q\)
gives a unique \(\xi:R\to L\) such that
\[
 \varphi_M=1_M+\iota_q\circ\xi\circ\pi_q.
\]

Let \(E_h^\varphi\) be the lower exact row in
\eqref{eq:middle-row-change}.  By \cref{lem:middle-row-correction}, its stable
connecting morphism is
\[
 \underline{\varphi_M}\circ\kappa_h
 =\underline{r_M}\circ\kappa_h
 =(\theta_1)_{\Sigma_{1*}(M)}.
\]
Hence the candidate triangle \(C_h^\varphi\) in
\eqref{eq:middle-row-candidate} is distinguished by
\cref{prop:intrinsic-inverse-comparison}\textup{(ii)}.  Its third morphism is
\[
 \begin{bmatrix}
  -\Sigma_1(f)&0\\
  1_{\Sigma_1(X)}+\lambda&d
 \end{bmatrix}
\]
for some \(\lambda:\Sigma_1(X)\to\Sigma_1(X)\).  Applying
\cref{cor:cone-shear-removal} gives a good completion of \(q\).  The
proof is complete.
\end{proof}

\begin{proof}[Proof of \cref{thm:relative-heller-criterion}]
Assume first that \(\Delta_1\) satisfies \(\mathrm{TR4}\), and let
\[
e:\quad0\longrightarrow A\longrightarrow B\longrightarrow C
\longrightarrow0
\]
be short exact in \(\mathcal E\). For \(i=0,1\),
\cref{lem:realize-short-exact-sequence} gives an initial square \(q_i\)
for \(\Delta_i\) and an isomorphism
\[
\rho_i:e_{q_i}^{\,i}\xrightarrow{\sim}e.
\]
Both pre-triangulations satisfy \(\mathrm{TR4}\), so choose a good
completion of each \(q_i\). In the diagram \eqref{eq:standard-heller-square} associated with
\(q_i\) and the chosen good completion, transport the right-hand
column along \(\rho_i\) and the left-hand column along
\(\mathsf F_i(\rho_i)\).

By \cref{lem:standard-three-row}, the chosen good completions make
all three rows Heller resolutions. Naturality of \(\theta_i\) ensures
that, after transport, they are Heller resolutions of \(A,B,C\),
respectively. The internal columns are split short exact sequences of
projective-injective objects. Thus, reading either diagram by columns
gives a partial projective resolution of \(e\) in
\(\mathsf{SES}(\mathcal E)\).

Using projectivity of the internal columns in
\(\mathsf{SES}(\mathcal E)\), we construct, successively from right
to left, a morphism from the first partial resolution to the second
whose component at \(e\) is \(1_e\). The induced morphism between
the leftmost kernels is
\[
t:\mathsf F_0(e)\longrightarrow\mathsf F_1(e).
\]
Since the comparison is the identity on the right-hand column,
naturality of stable connecting morphisms gives
\[
\underline{t_X}\circ(\theta_0)_X=(\theta_1)_X
\qquad(X\in\{A,B,C\}).
\]
Hence \(\underline{t_X}=\tau_X\) for every \(X\in\{A,B,C\}\), so \(t\)
is an exact lift of \(\tau\) along \(e\).

Conversely, assume that \(\tau\) admits exact lifts, and let \(q\) be a
composition square for \(\Delta_1\).  Apply the hypothesis to \(e_q^1\).
By \cref{prop:exact-lift-kills-obstruction}, \(q\) has a good completion.
The composition criterion \cref{thm:composition-obstruction} then implies
that \(\Delta_1\) satisfies \(\mathrm{TR4}\).
\end{proof}

\subsection{Defects of exact lifts and central twists}

We first consider endomorphisms of a short exact sequence whose first two
stable components are identities.  For central Heller twists, the possible
last components determine a local obstruction in a quotient of a stable
endomorphism space.  We then test its nonvanishing using Serre traces.

Consider a short exact sequence in \(\mathcal E\),
\begin{equation*}
\xymatrix@C=2.2pc{
e:\quad0\ar[r]&A\ar[r]^-{i}&B\ar[r]^-{p}&M\ar[r]&0.
}
\end{equation*}
It determines the distinguished
triangle
\begin{equation}
\label{eq:stable-triangle-of-short-exact-sequence}
\xymatrix@C=2.1pc{
A\ar[r]^-{\underline i}&B\ar[r]^-{\underline p}&M\ar[r]^-{\partial_e}&\Omega^{-1}(A).
}
\end{equation}
in the stable category.  The last map \(\partial_e\) is the stable connecting
morphism of \(e\).

\begin{definition}
\label{def:endpoint-defect}
Let \(t=(t_A,t_B,t_M):e\to e\) be an endomorphism of the short exact
sequence \(e\).  If \(\underline{t_A}=1_A\) and
\(\underline{t_B}=1_B\), we call
\[
 \underline{t_M}-1_M\in\underline{\End}(M)
\]
the \emph{defect of \(t\) at \(M\)}.  The set of all such defects is
\[
 D_e=\bigl\{\underline{t_M}-1_M\bigm|
 t:e\to e,\ \underline{t_A}=1_A,\ \underline{t_B}=1_B\bigr\}.
\]
\end{definition}

\begin{lemma}
\label{lem:endpoint-defect}
The set \(D_e\) is the additive subgroup
\begin{equation}
\label{eq:endpoint-defect-subgroup}
 D_e=
 \bigl\{
 \underline p\circ a\circ\partial_e
 \mid
 a\in\underline{\Hom}(\Omega^{-1}(A),B)
 \bigr\}
 \subseteq\underline{\End}(M).
\end{equation}
\end{lemma}

\begin{proof}
We first prove the inclusion
\[
 D_e\subseteq
 \bigl\{
 \underline p\circ a\circ\partial_e
 \mid
 a\in\underline{\Hom}(\Omega^{-1}(A),B)
 \bigr\}.
\]
Let \(t=(t_A,t_B,t_M):e\to e\) satisfy
\(\underline{t_A}=1_A\) and \(\underline{t_B}=1_B\), and put
\(r_X=t_X-1_X\).  Then
\((r_A,r_B,r_M)\) is again an endomorphism of the short exact sequence
\(e\).  Choose a factorization
\[
 r_A=v\circ u,
 \qquad
 A\xrightarrow{u}P\xrightarrow{v}A,
\]
through a projective-injective object \(P\).  Since \(P\) is injective,
\(u\) extends across the monomorphism \(i:A\hookrightarrow B\) to a morphism
\(\widetilde u:B\to P\) satisfying \(\widetilde u\circ i=u\).
Define
\[
 k=i\circ v\circ\widetilde u:B\longrightarrow B.
\]
Since
\[
 k\circ i=i\circ r_A=r_B\circ i,
\]
the morphism \(r_B-k\) vanishes on \(\operatorname{im}i\).  Because
\(p\) is the cokernel of \(i\), there is a unique morphism
\(w:M\to B\) such that
\begin{equation}
\label{eq:endpoint-defect-factor}
 r_B-k=w\circ p.
\end{equation}
The right square of the endomorphism \(t\) gives
\(r_M\circ p=p\circ r_B\).  Since \(p\circ i=0\), we have
\(p\circ k=0\).  Using \eqref{eq:endpoint-defect-factor}, we obtain
\[
 r_M\circ p=p\circ r_B
 =p\circ(r_B-k)=p\circ w\circ p.
\]
As \(p\) is an epimorphism, it follows that
\[
 r_M=p\circ w.
\]
Passing to the stable category, we compute
\[
 \underline w\circ\underline p
 =\underline{r_B}-\underline k=0,
\]
because both \(r_B\) and \(k\) factor through projective-injective
objects.  Exactness of \(\underline{\Hom}(-,B)\) applied to the triangle
\eqref{eq:stable-triangle-of-short-exact-sequence} therefore yields
\(a\in\underline{\Hom}(\Omega^{-1}(A),B)\) with
\[
 \underline w=a\circ\partial_e.
\]
Consequently
\[
 \underline{t_M}-1_M
 =\underline{r_M}
 =\underline p\circ\underline w
 =\underline p\circ a\circ\partial_e.
\]
This proves the inclusion ``\(\subseteq\)''.

Conversely, let
\(a\in\underline{\Hom}(\Omega^{-1}(A),B)\), and choose a morphism
\(w:M\to B\) representing \(a\circ\partial_e\).  Set
\[
 r_A=0,
 \qquad r_B=w\circ p,
 \qquad r_M=p\circ w.
\]
Then
\[
 r_B\circ i=0=i\circ r_A,
 \qquad
 p\circ r_B=r_M\circ p,
\]
so \((r_A,r_B,r_M)\) is an endomorphism of the short exact sequence
\(e\).  Moreover,
\[
 \underline{r_B}
 =\underline w\circ\underline p
 =a\circ\partial_e\circ\underline p=0,
\]
so \((1_A,1_B+r_B,1_M+r_M)\) has identity components on the first two
terms in the stable category and defect at \(M\) equal to
\(\underline p\circ a\circ\partial_e\).  This proves the inclusion
``\(\supseteq\)'', and the right-hand side of
\eqref{eq:endpoint-defect-subgroup} is an additive subgroup.
\end{proof}

Let \(\kk\) be a field, and let \(\T\) be an idempotent-complete
\(\kk\)-linear triangulated category with suspension \(\Sigma\) and
triangulation \(\Delta_0\).  Put
\(\mathcal E=\mathsf{mod}\,\T\) and
\(\mathsf F_0=(\Sigma^{-1})_*\), and let
\(\theta_0:\Omega^3\xrightarrow{\sim}\mathsf F_0\) be the inverse Heller
comparison associated with \(\Delta_0\).  In this \(\kk\)-linear setting,
\cref{lem:endpoint-defect} also shows that each \(D_e\) is a
\(\kk\)-subspace of the relevant stable endomorphism space.

\begin{definition}[{\normalfont cf.~\cite[Definition~2.1]{KrauseYe2011}}]
\label{def:degree-zero-center}
Let \(\mathcal C\) be a triangulated category with suspension \(S\).  Its
degree-zero center is
\[
 Z^0(\mathcal C)=
 \bigl\{\delta:\operatorname{Id}_{\mathcal C}\Rightarrow
 \operatorname{Id}_{\mathcal C}
 \mid
 \delta_{S(X)}=S(\delta_X)\text{ for every }X\in\mathcal C
 \bigr\}.
\]
\end{definition}

\begin{definition}
\label{def:central-twist}
Retain the preceding setup, and let
\(\delta\in Z^0(\underline{\mathcal E})\) be square-zero; explicitly,
\begin{equation}
\label{eq:stable-square-zero}
 \delta_X\circ\delta_X=0
 \quad\text{in }\underline{\End}_{\mathcal E}(X)
 \quad\text{for every }X\in\underline{\mathcal E}.
\end{equation}
For \(\lambda\in\kk\), define
\begin{equation}
\label{eq:central-heller-comparison}
 (\eta_{\lambda,\delta})_M
 =1_{\Omega^3(M)}+\lambda\delta_{\Omega^3(M)},
 \qquad
 \theta_\lambda=\theta_0\circ\eta_{\lambda,\delta}.
\end{equation}
The pre-triangulation corresponding to \(\theta_\lambda\) by
\cref{lem:heller-parameters} is denoted by
\(\Delta_\lambda^\delta\) and is called the \emph{central Heller twist}
of \(\Delta_0\) by \(\lambda\delta\).
\end{definition}

\begin{remark}
Naturality of \(\delta\) makes
\(\eta_{\lambda,\delta}\) a natural endomorphism of \(\Omega^3\), and
\eqref{eq:stable-square-zero} gives
\[
 \eta_{\lambda,\delta}^{-1}
 =1_{\Omega^3}-\lambda\delta_{\Omega^3}.
\]
Since \(\delta\in Z^0(\underline{\mathcal E})\), the natural automorphism
\(\eta_{\lambda,\delta}\) is compatible with the stable suspension
\(\Omega^{-1}\).  Thus
\(\eta_{\lambda,\delta}\in\operatorname{Aut}_{\mathrm{tri}}(\Omega^3)\),
and \(\theta_\lambda\) is an isomorphism of triangulated functors.  The bijection
in \cref{lem:heller-parameters} therefore yields the stated
pre-triangulation.
\end{remark}

The suspension is not changed by the central twist.  Hence the two induced
exact functors are equal,
\[
 \mathsf F_{\Delta_\lambda^\delta}
 = (\Sigma^{-1})_*=\mathsf F_0,
\]
and the relative comparison of \eqref{eq:relative-heller-comparison} is
the automorphism
\[
 \tau_\lambda
 =\theta_\lambda\circ\theta_0^{-1}
 =\theta_0\circ\eta_{\lambda,\delta}\circ\theta_0^{-1}
 \quad\text{of }\mathsf F_0.
\]
For every \(M\in\underline{\mathcal E}\), naturality of \(\delta\) gives
\[
 \delta_{\mathsf F_0(M)}\circ(\theta_0)_M
 =(\theta_0)_M\circ\delta_{\Omega^3(M)}.
\]
Conjugating by \((\theta_0)_M\) therefore yields
\begin{equation}
\label{eq:central-relative-comparison}
 (\tau_\lambda)_M
 =1_{\mathsf F_0(M)}+\lambda\delta_{\mathsf F_0(M)}.
\end{equation}

For a short exact sequence
\[
 e:\quad0\longrightarrow A\longrightarrow B\longrightarrow M
 \longrightarrow0
\]
and \(\delta\in Z^0(\underline{\mathcal E})\) satisfying
\(\delta_A=0=\delta_B\), write
\begin{equation}
\label{eq:central-endpoint-class}
 [\delta_M]_e=\delta_M+D_e
 \quad\in\quad
 \underline{\End}(M)/D_e.
\end{equation}
This is the local obstruction to lifting the corresponding central twist.

\begin{proposition}
\label{prop:central-twist-obstruction}
Let \(\T\) be an idempotent-complete \(\kk\)-linear triangulated category
with suspension \(\Sigma\) and triangulation \(\Delta_0\), let
\(\delta\in Z^0(\underline{\mathcal E})\) be square-zero, and let
\[
 e:\quad0\longrightarrow A\longrightarrow B\longrightarrow M
 \longrightarrow0
\]
be short exact with \(\delta_A=0=\delta_B\).  For \(\lambda\in\kk\), the
relative comparison from \(\Delta_0\) to \(\Delta_\lambda^\delta\) admits
an exact lift along \(\mathsf F_0^{-1}(e)\) if and only if
\begin{equation}
\label{eq:central-endpoint-obstruction}
 \lambda[\delta_M]_e=0
 \quad\text{in}\quad
 \underline{\End}(M)/D_e.
\end{equation}
Consequently, if \([\delta_M]_e\ne0\), then the twists
\(\Delta_\lambda^\delta\) are pairwise distinct and
\[
 \Delta_\lambda^\delta\text{ satisfies }\mathrm{TR4}
 \quad\Longleftrightarrow\quad
 \lambda=0.
\]
\end{proposition}

\begin{proof}
In view of \eqref{eq:central-relative-comparison}, an exact lift along
\(\mathsf F_0^{-1}(e)\) is precisely an endomorphism \(t:e\to e\) with
\[
 \underline{t_X}=1_X+\lambda\delta_X
 \qquad(X=A,B,M).
\]
Since \(\delta_A=0=\delta_B\), subtracting the identity endomorphism of
\(e\) shows that such a \(t\) exists exactly when its defect at \(M\),
namely \(\lambda\delta_M\), belongs to \(D_e\).  By
\eqref{eq:central-endpoint-class}, this is equivalent to
\eqref{eq:central-endpoint-obstruction}.

If \([\delta_M]_e\ne0\), no nonzero \(\lambda\) admits this exact lift.
The necessity in \cref{thm:relative-heller-criterion} therefore excludes
\(\mathrm{TR4}\) for \(\Delta_\lambda^\delta\), while \(\lambda=0\) is the
base triangulation.  Finally, equality of
\(\Delta_\lambda^\delta\) and \(\Delta_\mu^\delta\) would imply
\(\theta_\lambda=\theta_\mu\) by
\cref{lem:heller-parameters}.  Evaluating at
\(\Omega^{-3}(M)\) gives
\((\lambda-\mu)\delta_M=0\).  Since
\([\delta_M]_e\ne0\), this forces \(\lambda=\mu\).
\end{proof}

\subsection{A criterion using Serre traces}
\label{sec:serre-traces}

Throughout this subsection, retain the preceding setup and put
\[
 \C=\underline{\mathcal E},
 \qquad
 S=\Omega^{-1}.
\]
Assume that \(\C\), with suspension \(S\), is a \(\kk\)-linear,
Hom-finite, idempotent-complete triangulated category.  Fix an integer
\(n\ge1\) and an
\(n\)-Calabi--Yau structure on \(\C\).  Thus \(S^n\) is equipped with the
structure of a Serre functor.  In particular, \(\C\) is Krull--Schmidt.
Serre duality gives natural perfect pairings; see
\cite[Section~I.1, especially equations~(I.1.1)--(I.1.4)]
{ReitenVanDenBergh2002}:
\[
 \langle-,-\rangle_{X,Y}:
 \underline{\Hom}(X,Y)\times
 \underline{\Hom}(Y,S^n(X))\longrightarrow\kk.
\]
For \(u:X\to S^n(X)\), the associated \emph{Serre trace} is
\[
 \operatorname{Tr}_X(u)=\langle1_X,u\rangle_{X,X}.
\]
It satisfies the cyclicity identity
\begin{equation}
\label{eq:serre-trace-cyclicity}
 \operatorname{Tr}_X(g\circ f)
 =\operatorname{Tr}_Y(S^n(f)\circ g)
\end{equation}
for \(f:X\to Y\) and \(g:Y\to S^n(X)\).

\begin{definition}
\label{def:trace-functional}
Let
\begin{equation*}
\xymatrix@C=2.2pc{
e:\quad0\ar[r]&A\ar[r]&B\ar[r]^-{p}&M\ar[r]&0
}
\end{equation*}
be short exact.  We say that \(e\) has \emph{compatible
\(n\)-period identifications} if there are isomorphisms in \(\C\),
\[
 \phi_B:S^n(B)\xrightarrow{\sim}B,
 \qquad
 \phi_M:S^n(M)\xrightarrow{\sim}M,
\]
such that the following square commutes:
\begin{equation}
\label{eq:relative-period-square}
\begin{tikzcd}[column sep=large,row sep=large]
S^n(B)\arrow[r,"S^n(\underline p)"]
 \arrow[d,"\phi_B"']&
S^n(M)\arrow[d,"\phi_M"]\\
B\arrow[r,"\underline p"']&M.
\end{tikzcd}
\end{equation}
For such identifications, define
\[
 \ell_M:\underline{\End}(M)\longrightarrow\kk,
 \qquad
 \ell_M(a)=\operatorname{Tr}_M(\phi_M^{-1}\circ a).
\]
\end{definition}

The compatibility with \(\underline p\) in
\eqref{eq:relative-period-square} ensures that \(\ell_M\) vanishes on \(D_e\).

\begin{proposition}
\label{prop:trace-separation}
In the setting of \cref{def:trace-functional}, one has
\[
 D_e\subseteq\ker\ell_M.
\]
\end{proposition}

\begin{proof}
For \(a=\underline p\circ v\circ\partial_e\in D_e\),
the commutative square \eqref{eq:relative-period-square} gives
\[
 \phi_M^{-1}\circ\underline p
 =S^n(\underline p)\circ\phi_B^{-1}.
\]
Using this identity and \eqref{eq:serre-trace-cyclicity}, we obtain
\[
\begin{aligned}
 \ell_M(a)
 &=\operatorname{Tr}_M
   \bigl(S^n(\underline p)\circ\phi_B^{-1}\circ
   v\circ\partial_e\bigr)\\
 &=\operatorname{Tr}_B
   \bigl(S^n(v\circ\partial_e)\circ
   S^n(\underline p)\circ\phi_B^{-1}\bigr)\\
 &=\operatorname{Tr}_B
   \bigl(S^n(v\circ\partial_e\circ\underline p)
   \circ\phi_B^{-1}\bigr)=0,
\end{aligned}
\]
because \(\partial_e\circ\underline p=0\) in
\eqref{eq:stable-triangle-of-short-exact-sequence}.
\end{proof}

\begin{corollary}
\label{cor:trace-detector}
Retain the hypotheses of \cref{prop:central-twist-obstruction} and the
standing \(n\)-Calabi--Yau assumptions of this subsection.  Assume that
\(e\) has compatible \(n\)-period identifications as in
\cref{def:trace-functional} and that
\[
 \ell_M(\delta_M)\ne0.
\]
Then \([\delta_M]_e\ne0\), and hence the twists are pairwise distinct with
\[
 \Delta_\lambda^\delta\text{ satisfying }\mathrm{TR4}
 \quad\Longleftrightarrow\quad
 \lambda=0.
\]
\end{corollary}

\begin{proof}
By \cref{prop:trace-separation}, every element of \(D_e\) is annihilated by
\(\ell_M\).  Thus \(\ell_M(\delta_M)\ne0\) implies
\(\delta_M\notin D_e\), and the result follows from
\cref{prop:central-twist-obstruction}.
\end{proof}

An \emph{\(S\)-orbit} means an orbit of the action of \(S\) on the
isomorphism classes of indecomposable objects of \(\C\).  We choose
representatives and suppress the resulting identifications.

\begin{definition}
\label{def:periodic-socle-orbit}
A \emph{socle orbit} consists of an \(S\)-orbit
\(\mathcal O\) and maps
\[
 0\ne\rho_X\in\operatorname{rad}\bigl(\End_{\C}(X)\bigr)
 \qquad(X\in\mathcal O)
\]
such that
\[
 \rho_{S(X)}=S(\rho_X),\qquad
 \End_{\C}(X)/
 \operatorname{rad}\bigl(\End_{\C}(X)\bigr)\cong\kk,\qquad
 \operatorname{soc}\bigl(\End_{\C}(X)\bigr)=\kk\rho_X
\]
for every \(X\in\mathcal O\).
\end{definition}

\pagebreak[3]
The following proposition is motivated by the twisting construction in
\cite[Section~3.2]{ChenLiuLuZhang2026}.

\begin{proposition}
\label{prop:periodic-orbit-detector}
Fix a socle orbit \((\mathcal O,\rho)\) in \(\C\).  Assume that there is a
short exact sequence in \(\mathcal E\),
\begin{equation*}
	\xymatrix@C=2.2pc{
		\quad0\ar[r]&A\ar[r]&B\ar[r]^-{p}&M\ar[r]&0
	}
\end{equation*}
satisfying the following conditions:
\begin{enumerate}[label=\textup{(\arabic*)}]
\item \(M\in\mathcal O\);
\item neither \(A\) nor \(B\), regarded as an object of \(\C\), has a
 direct summand isomorphic to an object of \(\mathcal O\);
\item the sequence has compatible \(n\)-period identifications
\[
 \phi_B:S^n(B)\xrightarrow{\sim}B,
 \qquad
 \phi_M:S^n(M)\xrightarrow{\sim}M
\]
in the sense of \cref{def:trace-functional}.
\end{enumerate}
Then the maps \(\rho_X\), extended by zero away from \(\mathcal O\),
define a square-zero element
\(\delta_{\mathcal O}\in Z^0(\C)\).  The associated central Heller twists
\(\Delta_\lambda^{\delta_{\mathcal O}}\) of \(\Delta_0\) are pairwise
distinct, and
\[
 \Delta_\lambda^{\delta_{\mathcal O}}
 \text{ satisfies }\mathrm{TR4}
 \quad\Longleftrightarrow\quad
 \lambda=0.
\]
\end{proposition}

\begin{proof}
Since \(\C\) is Hom-finite and idempotent-complete, it is
Krull--Schmidt.  As \(S^n\) is a Serre functor, the
Auslander--Reiten triangle ending at \(M\) has the form
\[
 S^{n-1}(M)\longrightarrow E_M\longrightarrow M
 \xrightarrow{\alpha_M}S^n(M);
\]
see \cite[Proposition~I.2.3 and Theorem~I.2.4]
{ReitenVanDenBergh2002}.  The construction in the proof of
Proposition~I.2.3 applies over \(\kk\) here, since
\(\End_\C(M)/\operatorname{rad}\End_\C(M)\cong\kk\) by
\cref{def:periodic-socle-orbit}.
The morphism \(\alpha_M\) is almost vanishing:
\(\alpha_M\circ g=0\) for every non-retraction \(g:X\to M\), and
\(h\circ\alpha_M=0\) for every non-section
\(h:S^n(M)\to Y\).

It follows that \(\phi_M\circ\alpha_M\) is nonzero and is annihilated on
both sides by \(\operatorname{rad}\End_{\C}(M)\).  Hence the socle-orbit
condition gives
\[
 \phi_M\circ\alpha_M=\kappa\rho_M
 \qquad\text{for some }\kappa\in\kk^\times.
\]
In particular, \(\rho_M\) is almost vanishing.  Since
\(\rho_{S(X)}=S(\rho_X)\), the same is true of every \(\rho_X\) with
\(X\in\mathcal O\).

We next extend these maps by zero.  If
\(f\in\End_{\C}(X)\) for \(X\in\mathcal O\), write
\(f=a\,1_X+r\), where \(a\in\kk\) and
\(r\in\operatorname{rad}\End_{\C}(X)\).  The almost-vanishing property
then gives
\[
 f\circ\rho_X=a\rho_X=\rho_X\circ f.
\]
Now let \(X\ncong Y\) be indecomposable and let \(f:X\to Y\).  If
\(X\in\mathcal O\) and \(Y\notin\mathcal O\), then
\(f\circ\rho_X=0\); if \(X\notin\mathcal O\) and
\(Y\in\mathcal O\), then \(\rho_Y\circ f=0\); and if
\(X,Y\in\mathcal O\), both composites vanish.  Thus the rule
\[
 (\delta_{\mathcal O})_X=
 \begin{cases}
  \rho_X,&X\in\mathcal O,\\
  0,&X\notin\mathcal O,
 \end{cases}
\]
on indecomposable representatives, extended additively to direct sums,
defines a natural endomorphism of the identity functor on \(\C\).
The equality \(\rho_{S(X)}=S(\rho_X)\) shows that
\(\delta_{\mathcal O}\in Z^0(\C)\).  Since each \(\rho_X\) is radical
and almost vanishing, \(\rho_X^2=0\).  Hence
\(\delta_{\mathcal O}\) is square-zero.

Condition~\textup{(2)} gives
\((\delta_{\mathcal O})_A=0=(\delta_{\mathcal O})_B\), whereas
\((\delta_{\mathcal O})_M=\rho_M\).  It remains to verify the trace
condition in \cref{cor:trace-detector}.  The Serre pairing is perfect and
\[
 \operatorname{Tr}_M(\alpha_M\circ u)
 =\langle u,\alpha_M\rangle_{M,M}.
\]
Since \(\alpha_M\ne0\), there is \(u\in\End_{\C}(M)\) for which this
scalar is nonzero.  Write
\(u=c\,1_M+r\), where \(r\in\operatorname{rad}\End_{\C}(M)\).
Almost vanishing gives \(\alpha_M\circ r=0\), and hence
\[
 0\ne\operatorname{Tr}_M(\alpha_M\circ u)
 =c\operatorname{Tr}_M(\alpha_M).
\]
Thus \(\operatorname{Tr}_M(\alpha_M)\ne0\).  Since
\(\phi_M^{-1}\circ\rho_M=\kappa^{-1}\alpha_M\), it follows that
\(\ell_M(\rho_M)\ne0\).  The result follows from
\cref{cor:trace-detector}.
\end{proof}

\section{The generalized type-\texorpdfstring{$\mathbb{L}_2$}{L2}
category}
\label{sec:L2-test}

\subsection{The algebra and the results of D\'iaz Cabrera--Muro}

Throughout this section, \(\kk\) is algebraically closed and
\[
 \Lambda=P(\mathbb L_2).
\]
D\'iaz Cabrera--Muro present \(\Lambda\) as the finite-dimensional
symmetric algebra given by the quiver
\begin{equation*}
 \xymatrix@C=3.6pc@R=2pc{
  1\ar@(ul,dl)[]_{b}\ar@<0.55ex>[r]^{a_1}&
  2\ar@<0.55ex>[l]^{a_2}
 }
\end{equation*}
with relations
\[
 a_1a_2=b^2,
 \qquad
 a_2a_1=0;
\]
see \cite[pp.~2--3]{DiazCabreraMuro2026}.  All modules in this section
are finite-dimensional.  Write
\[
 P_i=e_i\Lambda,
 \qquad
 S_i=P_i/\bigl(P_i\operatorname{rad}(\Lambda)\bigr).
\]
The following basis will be used repeatedly:
\[
\begin{array}{c|c|c|c}
e_1\Lambda e_1&e_1\Lambda e_2&e_2\Lambda e_1&e_2\Lambda e_2\\ \hline
e_1,b,b^2,b^3&a_1,ba_1&a_2,a_2b&e_2,a_2ba_1.
\end{array}
\]
In particular,
\begin{equation}
\label{eq:L2-elementary-relations}
 b^4=0,
 \qquad
 b^2a_1=0,
 \qquad
 a_2b^2=0,
\end{equation}
and \(b^2\) is central.  Since \(\Lambda\) is a finite-dimensional
symmetric algebra, \(\operatorname{mod}\Lambda\) is a Hom-finite
Krull--Schmidt Frobenius category; see
\cite[Chapter~I, Section~2]{Happel1988}.

Put
\[
 \T=\operatorname{proj}\Lambda,
 \qquad
 \mathcal E=\mathsf{mod}\,\T.
\]
Since the regular module \(\Lambda\) is an
additive generator of \(\T\), evaluation at
\(\Lambda\) is an exact equivalence
\begin{equation}
\label{eq:L2-Freyd-evaluation}
 \operatorname{ev}_\Lambda:
 \mathcal E\xrightarrow{\ \sim\ }\operatorname{mod}\Lambda,
 \qquad
 \mathcal M\longmapsto\mathcal M(\Lambda).
\end{equation}
For example,
\(\operatorname{ev}_\Lambda(h_P)=\Hom_\Lambda(\Lambda,P)\cong P\).
Thus \(\mathcal E\) is Frobenius, and we use
\eqref{eq:L2-Freyd-evaluation} to identify
\(\underline{\mathcal E}\) with
\(\underline{\operatorname{mod}}\Lambda\).

For \(\alpha\in\operatorname{Aut}(\Lambda)\), let
\[
 G_\alpha:\T\longrightarrow\T,
 \qquad P\longmapsto P_\alpha,
 \qquad\text{and}\qquad
 F_\alpha:\operatorname{mod}\Lambda\longrightarrow
 \operatorname{mod}\Lambda,
 \qquad M\longmapsto M_\alpha
\]
be restriction of scalars.  The underlying vector space of \(M_\alpha\)
is \(M\), with right action
\begin{equation}
\label{eq:L2-module-twist}
 m\mathbin{\cdot_\alpha}a=m\alpha(a)
 \qquad(m\in M,\ a\in\Lambda).
\end{equation}
On morphisms, \(F_\alpha\) is the identity on the underlying linear map.
It is therefore an exact autoequivalence.  We use the convention
\begin{equation}
\label{eq:L2-twist-composition}
 (M_\alpha)_\beta=M_{\alpha\circ\beta}.
\end{equation}

The autoequivalence \(G_\alpha\) induces
\((G_\alpha)_*(\mathcal M)=\mathcal M\circ G_\alpha^{-1}\) on
\(\mathcal E\).  Under evaluation, this becomes \(F_\alpha\).  Indeed,
\[
 (G_\alpha)_*(h_P)\cong h_{G_\alpha(P)}=h_{P_\alpha},
\]
whose value at \(\Lambda\) is \(P_\alpha\).  Exactness extends the
identification from representable presentations to all of \(\mathcal E\).
Consequently the following diagram commutes up to natural isomorphism:
\begin{equation}
\label{eq:L2-twist-evaluation-square}
\begin{tikzcd}[column sep=large,row sep=large]
\mathcal E\arrow[r,"(G_\alpha)_*"]
 \arrow[d,"\operatorname{ev}_\Lambda"']&
\mathcal E\arrow[d,"\operatorname{ev}_\Lambda"]\\
\operatorname{mod}\Lambda\arrow[r,"F_\alpha"']&
\operatorname{mod}\Lambda.
\end{tikzcd}
\end{equation}

Following \cite[Section~2]{DiazCabreraMuro2026}, we describe the
pre-triangulations by their inverse suspensions.  Thus
\(\Sigma_\alpha^{-1}=G_\alpha\) means that triangles have the form
\[
 P_\alpha\longrightarrow Q\longrightarrow R\longrightarrow P.
\]
Under \eqref{eq:L2-Freyd-evaluation}, the target of the inverse Heller
comparison
\(\Omega^3\to(\Sigma_\alpha^{-1})_*\) is exactly \(F_\alpha\).

\begin{definition}
\label{def:L2-reduced}
A module is \emph{projective-free} if it has no nonzero projective direct
summand.  A short exact sequence is \emph{reduced} if both of its end
terms are projective-free.
\end{definition}

We now record the results of D\'iaz Cabrera--Muro that will be used.  This
also fixes the comparisons.

\begin{enumerate}[label=\textup{(\roman*)},leftmargin=2.1em]
\item The automorphism
\begin{equation}
\label{eq:L2-nu}
 \nu(e_i)=e_i,
 \qquad
 \nu(a_i)=-a_i,
 \qquad
 \nu(b)=-b
\end{equation}
determines the algebraic triangulation \(\Delta_\nu\) on \(\T\)
\cite[Theorem~2.6 and (3.2)]{DiazCabreraMuro2026}.  Its
inverse suspension is \(G_\nu\), and we write its inverse Heller
comparison as
\begin{equation}
\label{eq:L2-algebraic-heller}
 \theta_\nu:\Omega^3\xrightarrow{\sim}F_\nu
 \quad\text{on }\underline{\operatorname{mod}}\Lambda.
\end{equation}

\item The automorphism
\begin{equation}
\label{eq:L2-phi}
 \varphi(e_i)=e_i,
 \qquad
 \varphi(a_i)=a_i,
 \qquad
 \varphi(b)=b+b^3
\end{equation}
represents a nontrivial class in \(\operatorname{Out}(\Lambda)\), but it
induces an exact autoequivalence of the stable module category which is
isomorphic to the identity.  We fix the isomorphism of triangulated functors
\begin{equation}
\label{eq:L2-stable-zeta}
 \zeta:\operatorname{Id}\xrightarrow{\sim}F_\varphi
 \quad\text{on }\underline{\operatorname{mod}}\Lambda
\end{equation}
constructed in \cite[Proposition~3.4]{DiazCabreraMuro2026}.  Its proof
also records that
\begin{equation}
\label{eq:L2-socle}
 \operatorname{soc}(\Lambda)=(b^3,a_2ba_1).
\end{equation}
This ideal annihilates every indecomposable nonprojective module.
Krull--Schmidt decomposition therefore gives
\begin{equation}
\label{eq:L2-socle-annihilation}
 M\operatorname{soc}(\Lambda)=0
\end{equation}
for every projective-free module \(M\).  Since \(\varphi\) differs from
the identity only by the \(b^3\)-term in \(\varphi(b)\), the chosen
natural isomorphism satisfies
\begin{equation}
\label{eq:L2-zeta-on-projective-free}
 F_\varphi(M)=M,
 \qquad
 \zeta_M=\underline{1_M}
\end{equation}
on every projective-free module.

\item Put
\[
 \sigma=\nu\circ\varphi.
\]
D\'iaz Cabrera--Muro construct a pre-triangulation \(\Delta_\sigma\) on
\(\T\), with inverse suspension \(G_\sigma\), whose inverse Heller
comparison is
\begin{equation}
\label{eq:L2-new-heller}
 (\theta_\sigma)_M
 =\zeta_{F_\nu(M)}\circ(\theta_\nu)_M:
 \Omega^3(M)\xrightarrow{\sim}F_\sigma(M).
\end{equation}
This is the construction in
\cite[Theorem~2.6 and Corollary~3.8]{DiazCabreraMuro2026}.  The resulting
pre-triangulation is neither algebraic nor topological by
\cite[Corollary~4.2]{DiazCabreraMuro2026}.
\end{enumerate}

For the rest of this section, set
\[
 \mathsf F=F_\varphi.
\]
By \eqref{eq:L2-twist-composition},
\begin{equation}
\label{eq:L2-composition-of-twists}
 G_\varphi\circ G_\nu=G_\sigma,
 \qquad
 \mathsf F\circ F_\nu=F_\sigma.
\end{equation}
Our task is to prove that \(\Delta_\sigma\) satisfies \(\mathrm{TR4}\).

\subsection{The relative exact-lifting problem}

By \eqref{eq:L2-new-heller}, the relative comparison from
\(\Delta_\nu\) to \(\Delta_\sigma\) has component
\begin{equation}
\label{eq:L2-relative-comparison}
 \tau_M=\zeta_{F_\nu(M)}:
 F_\nu(M)\xrightarrow{\sim}F_\sigma(M).
\end{equation}
Since \(F_\nu\) is an exact autoequivalence and
\(F_\sigma=\mathsf F\circ F_\nu\), it is enough, by
\cref{thm:relative-heller-criterion}, to lift
\(\zeta:\operatorname{Id}\to\mathsf F\) along every short exact sequence.
Thus, for
\[
 e:\quad
 0\longrightarrow A\xrightarrow{j}B\xrightarrow{s}C\longrightarrow0,
\]
we seek a commutative diagram
\begin{equation}
\label{eq:L2-zeta-lift}
\begin{tikzcd}[column sep=large]
0\arrow[r]&A\arrow[r,"j"]\arrow[d,"t_A"']&
B\arrow[r,"s"]\arrow[d,"t_B"']&
C\arrow[r]\arrow[d,"t_C"']&0\\
0\arrow[r]&\mathsf F(A)\arrow[r,"\mathsf F(j)"']&
\mathsf F(B)\arrow[r,"\mathsf F(s)"']&
\mathsf F(C)\arrow[r]&0
\end{tikzcd}
\end{equation}
with \(\underline{t_X}=\zeta_X\) for \(X=A,B,C\).

\subsection{Exact lifts}

Whenever a module is written as \(X=X'\oplus X''\), we denote the
canonical projections by \(\operatorname{pr}_{X'}\) and
\(\operatorname{pr}_{X''}\).

For a module \(M\), put
\[
 T_1(M)=\bigl(M/M\operatorname{rad}(\Lambda)\bigr)e_1
\]
and write
\[
 \pi_M:M\twoheadrightarrow T_1(M)
\]
for the canonical epimorphism.  The assignment \(M\mapsto T_1(M)\) is a
right-exact functor to \(\operatorname{add}(S_1)\), and the maps \(\pi_M\)
are natural.

For a module \(M\), let
\[
 \mu_M:M\longrightarrow M,
 \qquad
 \mu_M(x)=xb^2.
\]
This is a module homomorphism because \(b^2\) is central.  If \(M\) is
projective-free, then \(Mb^3=0\) by
\eqref{eq:L2-socle-annihilation}, and
\[
 \ker \pi_M=M\operatorname{rad}(\Lambda)+Me_2,
 \qquad
 \operatorname{rad}(\Lambda)b^2=\kk b^3,
 \qquad
 e_2b^2=0,
\]
so \(\mu_M\) vanishes on \(\ker\pi_M\).  The universal property of
the quotient therefore gives a unique module homomorphism \(\beta_M\)
for which the following diagram commutes:
\begin{equation}
\label{eq:L2-beta}
\begin{tikzcd}[column sep=large,row sep=large]
 M\arrow[r,two heads,"\pi_M"]
  \arrow[dr,"\mu_M"']&
 T_1(M)\arrow[d,dashed,"\beta_M"]\\
 &M.
\end{tikzcd}
\end{equation}
Thus \(\beta_M(\pi_M(x))=xb^2\) for every \(x\in M\).
For such \(M\), set
\begin{equation}
\label{eq:L2-top-flag}
 K_M=\ker\beta_M,
 \qquad
 G_M=K_M+\pi_M\bigl(\{x\in Me_1\mid xba_1=0\}\bigr).
\end{equation}
Thus \(K_M\subseteq G_M\subseteq T_1(M)\).

The following factorization will be used to construct the right component
of an exact lift.

\begin{lemma}
\label{lem:L2-top-flag-factorization}
Let \(C\) be projective-free.  Suppose that
\(u:T_1(C)\to T_1(C)\) satisfies
\[
 u(K_C)\subseteq K_C,
 \qquad
 u(G_C)\subseteq G_C.
\]
Then
\[
 \beta_C\circ u\circ\pi_C:C\longrightarrow C
\]
factors through an object of \(\operatorname{add}(P_1)\).
\end{lemma}

\begin{proof}
	Put \(T=T_1(C)\), \(K=K_C\), and \(G=G_C\).  For a linear
	functional \(f:T\to\kk\) and \(y\in Ce_1\), set
	\[
	d_{f,y}(c)=f(\pi_C(c))yb^2
	\qquad(c\in C).
	\]
	We claim that \(d_{f,y}\) factors through \(P_1\) whenever either
	\[
	f(K)=0\ \text{ and }\ yba_1=0,
	\qquad\text{or}\qquad
	f(G)=0.
	\]
	
	To prove the claim, consider the linear functional
	\[
	r_f:(Ce_1)b\longrightarrow\kk,
	\qquad
	r_f(xb)=f(\pi_C(x)).
	\]
	It is well defined: if \(xb=0\), then
	$
	\beta_C(\pi_C(x))=xb^2=0,
	$
	so \(\pi_C(x)\in K\), and \(f(K)=0\) in both cases.  Moreover,
	\[
	(Ce_1)b^2\subseteq (Ce_1)b,
	\qquad
	(Ce_2)a_2b\subseteq (Ce_1)b,
	\]
	and
	\[
	r_f(xb^2)=f(\pi_C(xb))=0,
	\qquad
	r_f(za_2b)=f(\pi_C(za_2))=0
	\]
	for \(x\in Ce_1\) and \(z\in Ce_2\).
	
	Choose an extension
	$
	\ell\in\Hom_\kk(Ce_1,\kk)
	$
	of \(r_f\).  In the second case, the extension may be chosen to vanish
	on \((Ce_2)a_2\).  Indeed, put
	\[
	V=(Ce_1)b,
	\qquad
	W=(Ce_2)a_2.
	\]
	If \(xb=za_2\in V\cap W\), then
	$
	xba_1=za_2a_1=0,
	$
	and hence \(\pi_C(x)\in G\).  Since \(f(G)=0\), we have
	\(r_f(xb)=0\).  Thus \(r_f\) on \(V\) and the zero functional on \(W\)
	agree on \(V\cap W\).  Together they define a linear functional on
	\(V+W\), which may be extended to \(Ce_1\).
	
	Define a \(\kk\)-linear map
	\(h_\ell:C=Ce_1\oplus Ce_2\to P_1\) by
	\[
	\begin{aligned}
		h_\ell(x)&=\ell(x)b^3+\ell(xb)b^2
		&&(x\in Ce_1),\\
		h_\ell(z)&=\ell(za_2)ba_1
		&&(z\in Ce_2).
	\end{aligned}
	\]
	This map respects the vertex decompositions.  Since \(\ell\) extends
	\(r_f\), it vanishes on \((Ce_1)b^2\) and \((Ce_2)a_2b\).  Using
	\(a_1a_2=b^2\) and \eqref{eq:L2-elementary-relations}, we obtain
	\[
	\begin{aligned}
		h_\ell(xb)
		&=\ell(xb)b^3+\ell(xb^2)b^2
		=\ell(xb)b^3
		=h_\ell(x)b,\\
		h_\ell(xa_1)
		&=\ell(xa_1a_2)ba_1
		=\ell(xb^2)ba_1
		=0
		=h_\ell(x)a_1,\\
		h_\ell(za_2)
		&=\ell(za_2)b^3+\ell(za_2b)b^2
		=\ell(za_2)b^3
		=h_\ell(z)a_2.
	\end{aligned}
	\]
	Thus \(h_\ell\) is a \(\Lambda\)-module homomorphism.
	
	Let \(v_y:P_1\to C\) be the \(\Lambda\)-module homomorphism determined
	by \(v_y(e_1)=y\).  For \(x\in Ce_1\) and \(z\in Ce_2\), we have
	\[
	(v_y\circ h_\ell)(x+z)
	=\ell(x)yb^3+\ell(xb)yb^2+\ell(za_2)yba_1.
	\]
	The first term vanishes because \(Cb^3=0\).  The last term vanishes by
	\(yba_1=0\) in the first case and by
	\(\ell((Ce_2)a_2)=0\) in the second.  Since
	$
	\ell(xb)=r_f(xb)=f(\pi_C(x))
	$
	and
	$
	\pi_C(z)=0,
	$
	the middle term is \(d_{f,y}(x+z)\).  Consequently,
	$
	d_{f,y}=v_y\circ h_\ell
	$
	is a factorization by \(\Lambda\)-module homomorphisms through \(P_1\),
	proving the claim.
	
	Now put
	\[
	w=\beta_C\circ u:T\longrightarrow\beta_C(T).
	\]
	Then \(w(K)=0\) and \(w(G)\subseteq\beta_C(G)\).  Since \(\pi_C\) is
	surjective and \(\beta_C\circ\pi_C=\mu_C\),
	\[
	\beta_C(T)=\{yb^2\mid y\in Ce_1\},
	\qquad
	\beta_C(G)=
	\{yb^2\mid y\in Ce_1,\ yba_1=0\}.
	\]
	Let \(q_K:T\to T/K\) be the quotient map.  The restriction \(w|_G\)
	induces a linear map
	\[
	\overline w_G:G/K\longrightarrow\beta_C(G),
	\qquad
	g+K\longmapsto w(g).
	\]
	Since \(G/K\) is a subspace of \(T/K\), choose a linear extension
	$
	\overline w_0:T/K\longrightarrow\beta_C(G)
	$
	of \(\overline w_G\), and put
	$
	w_0=\overline w_0\circ q_K,
	w_1=w-w_0.
	$
	Then
	$
	w_0(K)=0,
	w_1(G)=0.
	$
	
	Since \(w_0\) factors through \(T/K\) and has image in
	\(\beta_C(G)\), while \(w_1\) factors through \(T/G\) and has image
	in \(\beta_C(T)\), we may choose functionals
	\(f_i,g_\ell:T\to\kk\) and elements \(y_i,z_\ell\in Ce_1\) such that
	\[
	\begin{aligned}
		w_0(t)&=\sum_{i=1}^r f_i(t)y_i b^2,
		& f_i(K)&=0,
		& y_i ba_1&=0,\\
		w_1(t)&=\sum_{\ell=1}^s g_\ell(t)z_\ell b^2,
		& g_\ell(G)&=0.
	\end{aligned}
	\]
	Consequently,
	\[
	\beta_C\circ u\circ\pi_C
	=\sum_{i=1}^r d_{f_i,y_i}
	+\sum_{\ell=1}^s d_{g_\ell,z_\ell}.
	\]
	Every summand factors through \(P_1\) by the claim above.  Hence
	\(\beta_C\circ u\circ\pi_C\) factors through an object of
	\(\operatorname{add}(P_1)\).
\end{proof}

\begin{lemma}
	\label{lem:L2-adapted-top-splitting}
	Let \(C\) be projective-free and let
	\(\rho:T_1(C)\twoheadrightarrow R\) be an epimorphism with
	\(R\in\operatorname{add}(S_1)\).  There is a section
	\(\eta:R\to T_1(C)\) such that
	\begin{equation}
		\label{eq:L2-adapted-top-section}
		\eta(\rho(K_C))\subseteq K_C,
		\qquad
		\eta(\rho(G_C))\subseteq G_C,
	\end{equation}
	and the endomorphism
	\begin{equation}
		\label{eq:L2-top-defect}
		\beta_C\circ\eta\circ\rho\circ\pi_C:C\longrightarrow C
	\end{equation}
	factors through an object of \(\operatorname{add}(P_1)\).
\end{lemma}

\begin{proof}
	Put \(T=T_1(C)\), \(K=K_C\), and \(G=G_C\).  The action of \(\Lambda\)
	on every object of \(\operatorname{add}(S_1)\) factors through
	\(\Lambda/\operatorname{rad}(\Lambda)\) and its \(e_1\)-component.
	Consequently every \(\kk\)-linear map between \(T\) and \(R\) is
	\(\Lambda\)-linear, and it is enough to work with finite-dimensional
	\(\kk\)-vector spaces.  Choose a section of
	\(\rho|_K:K\twoheadrightarrow\rho(K)\), extend it to a section of
	\(\rho|_G:G\twoheadrightarrow\rho(G)\), and then extend it to a
	section \(\eta:R\to T\) of \(\rho\).  Then
	\[
	\eta(\rho(K))\subseteq K,
	\qquad
	\eta(\rho(G))\subseteq G.
	\]
	Thus \(u=\eta\circ\rho\) preserves \(K\) and \(G\), and
	\cref{lem:L2-top-flag-factorization} shows that
	\(\beta_C\circ\eta\circ\rho\circ\pi_C\) factors through an object of
	\(\operatorname{add}(P_1)\).
\end{proof}

\begin{lemma}
\label{lem:L2-projective-comparison}
Let
\[
 P=P_1^{\oplus r_1}\oplus P_2^{\oplus r_2}.
\]
There is an isomorphism
\[
 \xi_P:P\xrightarrow{\sim}\mathsf F(P)
\]
which is the identity on \(\operatorname{rad}(P)\) and satisfies
\begin{equation}
\label{eq:L2-projective-difference}
 \xi_P(x)-x+xb^2\in Pb^3
 \qquad(x\in P).
\end{equation}
\end{lemma}

\begin{proof}
It is enough to consider \(P_1\) and \(P_2\).  Define
\[
 \xi_{P_1}(x)=(e_1-b^2)\varphi(x),
 \qquad
 \xi_{P_2}=1_{P_2}.
\]
For \(a\in\Lambda\),
\[
 \xi_{P_1}(xa)
 =(e_1-b^2)\varphi(x)\varphi(a)
 =\xi_{P_1}(x)\mathbin{\cdot_\varphi}a.
\]
Since \(e_1-b^2\) is invertible in \(e_1\Lambda e_1\), with inverse
\(e_1+b^2\), the first map is an isomorphism.  Moreover \(P_2b^3=0\), so
\(\mathsf F(P_2)=P_2\) and \(1_{P_2}:P_2\to\mathsf F(P_2)\) is a module
isomorphism.  The map \(\xi_{P_1}\) fixes
\(\operatorname{rad}(P_1)\).  Thus, if \(x=\lambda e_1+r\) with
\(r\in\operatorname{rad}(P_1)\), then
\[
 \xi_{P_1}(x)-x+xb^2=rb^2\in P_1b^3.
\]
On \(P_2\), right multiplication by \(b^2\) is zero.  Taking direct
sums gives \(\xi_P\).
\end{proof}

\begin{proposition}
\label{prop:L2-exact-lifting}
Let
\begin{equation}
\label{eq:L2-reduced-sequence}
 e:\quad
 0\longrightarrow A\xrightarrow{j}B\xrightarrow{s}C\longrightarrow0
\end{equation}
be reduced.  Then the stable natural isomorphism
\(\zeta:\operatorname{Id}\xrightarrow{\sim}\mathsf F\) admits an exact
lift along \(e\).
\end{proposition}

\begin{proof}
\smallskip\noindent\emph{Step 1: The top of the projective summand.}
Write \(B=N\oplus P\), where \(N\) is projective-free and
\[
 P=P_1^{\oplus r_1}\oplus P_2^{\oplus r_2}.
\]
Write \(j_N=\operatorname{pr}_N\circ j\) and
\(j_P=\operatorname{pr}_P\circ j\).  Then
\begin{equation}
\label{eq:L2-radical-component}
 j_P(A)\subseteq\operatorname{rad}(P).
\end{equation}
Indeed, otherwise a component \(A\to P_i\) would be surjective by
Nakayama's lemma.  Since \(P_i\) is projective, it would split and give a
projective direct summand of \(A\).

Put \(R=T_1(P)\).  By \eqref{eq:L2-radical-component},
\(\pi_P\circ\operatorname{pr}_P\) annihilates \(j(A)\).  Hence there is a
unique epimorphism \(r:C\twoheadrightarrow R\) such that
\begin{equation}
\label{eq:L2-module-quotient}
 r\circ s=\pi_P\circ\operatorname{pr}_P.
\end{equation}
Since \(R\operatorname{rad}(\Lambda)=0\) and \(Re_2=0\), the map \(r\)
annihilates
\[
 C\operatorname{rad}(\Lambda)+Ce_2=\ker\pi_C.
\]
It therefore induces a unique epimorphism
\[
 \rho:T_1(C)\twoheadrightarrow R
\]
such that \(r=\rho\circ\pi_C\).  Thus
\begin{equation}
\label{eq:L2-top-quotient}
\rho\circ\pi_C\circ s
 =\pi_P\circ\operatorname{pr}_P.
\end{equation}
The quotient maps constructed above are summarized by the commutative
diagram
\begin{equation*}
\begin{tikzcd}[column sep=large,row sep=large]
A\arrow[r,"j"]&
N\oplus P\arrow[r,two heads,"s"]
 \arrow[d,two heads,"\pi_P\circ\operatorname{pr}_P"']&
C\arrow[dl,two heads,"r"]
 \arrow[d,two heads,"\pi_C"]\\
&R=T_1(P)&T_1(C)\arrow[l,two heads,"\rho"]
\end{tikzcd}
\end{equation*}
By naturality of \(\pi\),
\[
 \pi_C\circ s
 =T_1(s)\circ(\pi_N\oplus\pi_P).
\]
Since \(\pi_N\oplus\pi_P\) is an epimorphism,
\eqref{eq:L2-top-quotient} is equivalent to
\begin{equation}
\label{eq:L2-top-square}
 \rho\circ T_1(s)=\operatorname{pr}_R.
\end{equation}

\smallskip\noindent\emph{Step 2: The left and middle components.}
Choose the section \(\eta:R\to T_1(C)\) supplied by
\cref{lem:L2-adapted-top-splitting}.  The epimorphism \(T_1(s)\) splits
in the semisimple category \(\operatorname{add}(S_1)\), so \(\eta\)
lifts to a map
\[
 \widetilde\eta:R\longrightarrow T_1(N)\oplus R.
\]
Equation \eqref{eq:L2-top-square} gives
\[
 \operatorname{pr}_R\circ\widetilde\eta
 =\rho\circ T_1(s)\circ\widetilde\eta
 =\rho\circ\eta
 =1_R.
\]
Hence there is a unique morphism
\[
 \gamma:R\longrightarrow T_1(N)
\]
such that
\[
 \widetilde\eta=
 \begin{bmatrix}\gamma\\1_R\end{bmatrix}.
\]
Consequently,
\begin{equation}
\label{eq:L2-top-lift}
 T_1(s)\circ
 \begin{bmatrix}\gamma\\1_R\end{bmatrix}=\eta.
\end{equation}

Set
\[
 \chi=\beta_N\circ\gamma\circ\pi_P:P\longrightarrow N.
\]
Let \(\xi_P\) be the isomorphism from
\cref{lem:L2-projective-comparison}.  Using \(\mathsf F(N)=N\), define
\begin{equation}
\label{eq:L2-middle-comparison}
 t_B=
 \begin{bmatrix}
  1_N&-\chi\\
  0&\xi_P
 \end{bmatrix}:
 N\oplus P\longrightarrow N\oplus\mathsf F(P)=\mathsf F(B).
\end{equation}
By \eqref{eq:L2-radical-component},
\(\pi_P\circ j_P=0\), and hence \(\chi\circ j_P=0\).  Moreover,
\(\xi_P\circ j_P=\mathsf F(j_P)\), because \(\xi_P\) is the identity
on \(\operatorname{rad}(P)\) and \(\mathsf F\) leaves the underlying
\(\kk\)-linear maps unchanged.  Since \(\mathsf F(A)=A\) and
\(\mathsf F(N)=N\), we obtain
\begin{equation}
\label{eq:L2-left-comparison}
 t_B\circ j
 =
 \begin{bmatrix}
  j_N-\chi\circ j_P\\
  \xi_P\circ j_P
 \end{bmatrix}
 =
 \begin{bmatrix}
  \mathsf F(j_N)\\
  \mathsf F(j_P)
 \end{bmatrix}
 =\mathsf F(j).
\end{equation}
The decompositions \(B=N\oplus P\) and
\(\mathsf F(B)=N\oplus\mathsf F(P)\) identify both objects with \(N\) in
the stable category.  Under these identifications,
\(\underline{t_B}=1_{\underline N}\).  Naturality of \(\zeta\) and
\eqref{eq:L2-zeta-on-projective-free} show that \(\zeta_B\) is represented
by the same identity.  Hence
\begin{equation}
\label{eq:L2-stable-middle-comparison}
 \underline{t_B}=\zeta_B.
\end{equation}

\smallskip\noindent\emph{Step 3: The right component.}
Since \(C\) is projective-free, \(\mathsf F(C)=C\) by
\eqref{eq:L2-zeta-on-projective-free}.  Thus \(\mathsf F(s)\circ t_B\)
and \(s\) have the same codomain.
Consider the square
\begin{equation}
\label{eq:L2-beta-compatibility}
\begin{tikzcd}[column sep=huge,row sep=large]
 T_1(N)\oplus P
  \arrow[r,"T_1(s)\circ(1\oplus\pi_P)"]
  \arrow[d,"\beta_N\oplus\mu_P"']&
 T_1(C)\arrow[d,"\beta_C"]\\
 N\oplus P\arrow[r,"s"']&C.
\end{tikzcd}
\end{equation}
It commutes: after precomposition with the epimorphism
\(\pi_N\oplus1_P:N\oplus P\to T_1(N)\oplus P\), both paths are
\(\mu_C\circ s\).  Precomposing this square with
\[
 \begin{bmatrix}\gamma\circ\pi_P\\1_P\end{bmatrix}:
 P\longrightarrow T_1(N)\oplus P
\]
and using \eqref{eq:L2-top-lift} gives
\begin{equation}
\label{eq:L2-beta-top-lift}
 \beta_C\circ\eta\circ\pi_P
 =
 s\circ
 \begin{bmatrix}
  \chi\\ \mu_P
 \end{bmatrix}.
\end{equation}

Define
\begin{equation}
\label{eq:L2-right-defect}
 d_C=\beta_C\circ\eta\circ\rho\circ\pi_C:C\longrightarrow C.
\end{equation}
By \cref{lem:L2-adapted-top-splitting}, \(d_C\) factors through an object
of \(\operatorname{add}(P_1)\).  By
\eqref{eq:L2-projective-difference}, for \(x\in P\) write
\[
 \xi_P(x)-x=-xb^2+z_x
 \qquad\text{for some }z_x\in Pb^3.
\]
The functor \(\mathsf F\) leaves the underlying linear maps unchanged.
Thus, for \(n\in N\) and \(x\in P\),
\[
\begin{aligned}
 \bigl(\mathsf F(s)\circ t_B-s\bigr)(n,x)
 &=s\bigl(-\chi(x),\xi_P(x)-x\bigr)\\
 &=-s\bigl(\chi(x),xb^2\bigr),
\end{aligned}
\]
where the last equality follows from
\(s(Pb^3)\subseteq Cb^3=0\).  Therefore
\[
 \mathsf F(s)\circ t_B-s
 =
 -s\circ
 \begin{bmatrix}
  \chi\\ \mu_P
 \end{bmatrix}
 \circ\operatorname{pr}_P.
\]
Equations \eqref{eq:L2-beta-top-lift} and
\eqref{eq:L2-top-quotient} therefore give
\[
\begin{aligned}
 \mathsf F(s)\circ t_B-s
 &=-\beta_C\circ\eta\circ\pi_P\circ\operatorname{pr}_P\\
 &=-\beta_C\circ\eta\circ\rho\circ\pi_C\circ s\\
 &=-d_C\circ s.
\end{aligned}
\]

Together with \eqref{eq:L2-left-comparison}, this gives a morphism of
short exact sequences
\begin{equation}
\label{eq:L2-final-exact-lift}
\begin{tikzcd}[column sep=large]
0\arrow[r]&A\arrow[r,"j"]\arrow[d,"1_A"']&
B\arrow[r,"s"]\arrow[d,"t_B"']&
C\arrow[r]\arrow[d,"1_C-d_C"']&0\\
0\arrow[r]&\mathsf F(A)=A\arrow[r,"\mathsf F(j)"']&
\mathsf F(B)\arrow[r,"\mathsf F(s)"']&
\mathsf F(C)=C\arrow[r]&0.
\end{tikzcd}
\end{equation}
Since \(d_C\) factors through a projective module,
\(\underline{1_C-d_C}=\underline{1_C}=\zeta_C\).  Together with
\eqref{eq:L2-stable-middle-comparison} and
\(\zeta_A=\underline{1_A}\), this proves that
\eqref{eq:L2-final-exact-lift} is an exact lift of \(\zeta\).
\end{proof}

\begin{theorem}
\label{thm:L2-octahedral}
For every algebraically closed field \(\kk\), the
D\'iaz Cabrera--Muro pre-triangulation \(\Delta_\sigma\) on
\(\operatorname{proj}P(\mathbb L_2)\) satisfies \(\mathrm{TR4}\).
Consequently,
\((\operatorname{proj}P(\mathbb L_2),\Delta_\sigma)\) is an exotic
triangulated category.
\end{theorem}

\begin{proof}
Since \(\Lambda\) is symmetric, projective modules are injective.  It
follows by splitting off projective summands from the two end terms that
every short exact sequence is isomorphic to a finite direct sum of a
reduced sequence and sequences of the two forms
\[
 0\longrightarrow P\xrightarrow{1_P}P\longrightarrow0\longrightarrow0,
 \qquad
 0\longrightarrow0\longrightarrow P\xrightarrow{1_P}P\longrightarrow0,
\]
with \(P\) projective.  By \cref{prop:L2-exact-lifting}, \(\zeta\) admits
an exact lift along the reduced summand.  On either displayed summand, the
zero vertical maps give
an exact lift because projective modules vanish in the stable category.
Exact lifts are preserved by finite direct sums and isomorphisms of short
exact sequences.  Hence \(\zeta\) admits an exact lift along every short
exact sequence.

Now let \(e\) be any short exact sequence and apply this conclusion to
\(F_\nu(e)\).  By
\eqref{eq:L2-composition-of-twists} and
\eqref{eq:L2-relative-comparison}, this gives an exact lift
\[
 F_\nu(e)\longrightarrow F_\sigma(e)
\]
whose stable component at each term \(M\) of \(e\) is
\(\zeta_{F_\nu(M)}=\tau_M\).  The relative criterion
\cref{thm:relative-heller-criterion} therefore proves
\(\mathrm{TR4}\) for \(\Delta_\sigma\).  Its exoticity follows from
\cite[Corollary~4.2]{DiazCabreraMuro2026}.
\end{proof}

\section{Separable descent and applications}
\label{sec:transfer}

An exact functor induces a homomorphism between the local obstruction
groups.  We show that this homomorphism is split injective when the functor
is separable, and use this to prove the descent theorem and its applications.

\begin{definition}[{\normalfont cf.~\cite[Definition~3.7]{Balmer2011}}]
\label{def:separable}
Let \(\D\) and \(\C\) be additive categories.  An additive functor
\(F:\D\to\C\) is \emph{separable} if there are group homomorphisms
\[
 H_{X,Y}:\C(F(X),F(Y))\longrightarrow\D(X,Y)
\]
such that
\[
 H_{X,Y}(F(a))=a
\]
for every \(a:X\to Y\), and
\begin{equation}
\label{eq:separable-bimodule-identity}
 H_{X',Y'}\bigl(F(v)\circ b\circ F(u)\bigr)
 =v\circ H_{X,Y}(b)\circ u
\end{equation}
whenever \(u:X'\to X\), \(b:F(X)\to F(Y)\), and \(v:Y\to Y'\).
\end{definition}

For an initial square \(q\) in a pre-triangulated category, we use the
notation of \cref{not:cone-data}, with subscripts added when necessary:
\[
\begin{gathered}
 A_q=Y\oplus X',\qquad B_q=Z\oplus Y',\qquad
 C_q=\Sigma(X)\oplus Z',\\
 \alpha_q=
 \begin{bmatrix}-v&0\\ g&u'\end{bmatrix},\qquad
 \beta_q^0=
 \begin{bmatrix}-w&0\\0&v'\end{bmatrix},\qquad
 \gamma_q=
 \begin{bmatrix}-\Sigma(u)&0\\ \Sigma(f)&w'\end{bmatrix}.
\end{gathered}
\]
We choose a distinguished triangle
\[
 A_q\xrightarrow{\alpha_q}B_q\xrightarrow{s_q}E_q
 \xrightarrow{t_q}\Sigma(A_q)
\]
when forming the obstruction group of \(q\).

\begin{lemma}
\label{lem:functorial-local-obstruction}
Let \(\D\) and \(\C\) be pre-triangulated categories, and let
\(F:\D\to\C\) be exact.  For every initial square \(q\) in \(\D\), there
is a homomorphism
\[
 \Phi_q:\operatorname{Obs}_{\D}(q)
 \longrightarrow\operatorname{Obs}_{\C}(F(q))
\]
which sends \(o_{\D}(q)\) to \(o_{\C}(F(q))\).  If \(F\) is separable,
then \(\Phi_q\) is split injective.
\end{lemma}

\begin{proof}
Write the suspension comparison of \(F\) as
\[
 \varepsilon_X:F(\Sigma_\D X)\xrightarrow{\sim}\Sigma_\C F(X).
\]
The initial square \(F(q)\) is obtained by applying \(F\) to the square
and its two chosen distinguished triangles, using \(\varepsilon\) on their last
arrows.  In particular,
\[
 C_{F(q)}=\Sigma_\C F(X)\oplus F(Z').
\]
After the canonical biproduct identification of \(F(C_q)\), put
\[
 \kappa_q=
 \begin{bmatrix}\varepsilon_X&0\\0&1_{F(Z')}\end{bmatrix}:
 F(C_q)\xrightarrow{\sim}C_{F(q)}.
\]
We use the image under \(F\) of the chosen triangle for \(q\), with its
last arrow composed with \(\varepsilon_{A_q}\), as the chosen triangle
for \(F(q)\).  The resulting cone data satisfy
\[
 \beta^0_{F(q)}=\kappa_q\circ F(\beta_q^0),\qquad
 \gamma_{F(q)}
 =\varepsilon_{A_q}\circ F(\gamma_q)\circ\kappa_q^{-1},\qquad
 t_{F(q)}=\varepsilon_{A_q}\circ F(t_q).
\]
Define
\[
 \Phi_q^0(\phi)=\kappa_q\circ F(\phi),
 \qquad
 \Phi_q^1(\bar b,r)=
 \bigl(\overline{\kappa_q\circ F(b)},
       \varepsilon_{A_q}\circ F(r)\bigr).
\]
The second formula is well defined because
\(j\mapsto\kappa_q\circ F(j)\) carries \(J_q^\D\) into
\(J_{F(q)}^\C\).  The
displayed identities give
\[
 \Phi_q^1\circ\delta_q^\D
 =\delta_{F(q)}^\C\circ\Phi_q^0.
\]
Thus \(\Phi_q^0\) and \(\Phi_q^1\) induce \(\Phi_q\) on cokernels, and
\(\Phi_q(o_\D(q))=o_\C(F(q))\).

Suppose now that \(F\) is separable, with maps \(H_{X,Y}\) as in
\cref{def:separable}.  On the two Hom-group terms above define
\[
 R_q^0(\psi)=H_{E_q,C_q}(\kappa_q^{-1}\circ\psi),
 \qquad
 R_q^1(\bar b,r)=
 \bigl(\overline{H_{B_q,C_q}(\kappa_q^{-1}\circ b)},
 H_{E_q,\Sigma_\D A_q}(\varepsilon_{A_q}^{-1}\circ r)\bigr).
\]
Write \(\iota\) and \(\operatorname{pr}\) for the relevant canonical
inclusions and projections.  For \(a:F(Z)\to F(Z')\), the bimodule identity gives
\[
 H_{B_q,C_q}\bigl(\kappa_q^{-1}\circ\iota_{F(Z')}\circ a
 \circ\operatorname{pr}_{F(Z)}\bigr)
 =\iota_{Z'}\circ H_{Z,Z'}(a)\circ\operatorname{pr}_Z.
\]
Thus \(R_q^1\) is well defined on the quotient by \(J_{F(q)}^\C\), and the same
identity shows that
\[
 R_q^1\circ\delta_{F(q)}^\C=\delta_q^\D\circ R_q^0,
 \qquad
 R_q^i\circ\Phi_q^i=1\quad(i=0,1).
\]
Hence \(R_q^0\) and \(R_q^1\) induce a retraction of \(\Phi_q\), which is
therefore split injective.
\end{proof}

\begin{theorem}
\label{thm:separable-descent}
Let \(\D\) be a pre-triangulated category, let \(\C\) be a triangulated
category, and let
$F:\D\to\C$ be exact and separable.  Then the given pre-triangulation on
$\D$ satisfies $\mathrm{TR4}$; in particular, $\D$ is a triangulated
category.
\end{theorem}

\begin{proof}
Let \(q\) be an initial square in \(\D\).  Since \(\C\) is triangulated,
axiom \(\mathrm{TR4}'\) supplies a good completion of \(F(q)\), and
\cref{thm:boundary-lifting} gives \(o_\C(F(q))=0\).  By
\cref{lem:functorial-local-obstruction}, the
homomorphism
\[
 \Phi_q:\operatorname{Obs}_\D(q)\longrightarrow
 \operatorname{Obs}_\C(F(q))
\]
is split injective and sends \(o_\D(q)\) to \(o_\C(F(q))\).  Hence
\(o_\D(q)=0\).  By \cref{thm:boundary-lifting}, every initial square in
\(\D\) has a good completion.  In particular, every composition square
does, so \cref{thm:composition-obstruction} shows that the given
pre-triangulation on \(\D\) satisfies \(\mathrm{TR4}\).
\end{proof}

In particular, a pre-triangulated category which fails \(\mathrm{TR4}\)
admits no exact separable functor to a triangulated category, and hence no
fully faithful exact functor to one.

We next use separable descent when a canonical pre-triangulation is already
available and only $\mathrm{TR4}$ remains.

\subsection{Separable monads and comonads}

\begin{definition}[{\normalfont cf.~\cite[Definitions~2.1, 2.4, and~3.5]
{Balmer2011}}]
A \emph{monad} \(\mathbb M=(M,\eta,\mu)\) on an additive category
\(\C\) consists of an additive endofunctor \(M\) and natural
transformations
\[
 \eta:\operatorname{Id}_\C\Rightarrow M,
 \qquad
 \mu:M^2\Rightarrow M
\]
such that, for every \(X\in\C\),
\[
 \mu_X\circ M(\mu_X)=\mu_X\circ\mu_{M(X)},\qquad
 \mu_X\circ M(\eta_X)=1_{M(X)}
 =\mu_X\circ\eta_{M(X)}.
\]
An \(M\)-module is a pair \((X,\rho)\) with \(\rho:M(X)\to X\) such that
\[
 \rho\circ\eta_X=1_X,
 \qquad
 \rho\circ M(\rho)=\rho\circ\mu_X.
\]
A morphism \((X,\rho_X)\to(Y,\rho_Y)\) is a map \(a:X\to Y\) satisfying
\(a\circ\rho_X=\rho_Y\circ M(a)\).  These objects and morphisms form the
Eilenberg--Moore category \(\mathbb M\text{-}\mathsf{Mod}_\C\).  Its
forgetful functor sends \((X,\rho)\) to \(X\), and its free functor sends
\(X\) to \((M(X),\mu_X)\).

The monad is \emph{separable} if there is a natural transformation
\(\sigma:M\Rightarrow M^2\) such that, for every \(X\in\C\),
\begin{equation}
\label{eq:separable-monad-identities}
 \mu_X\circ\sigma_X=1_{M(X)},\qquad
 M(\mu_X)\circ\sigma_{M(X)}
 =\sigma_X\circ\mu_X
 =\mu_{M(X)}\circ M(\sigma_X).
\end{equation}
Thus \(\sigma\) is a section of \(\mu\) and a morphism of
\(M\)-bimodules.  If \(\C\) is triangulated, \(M\) is exact, and
\(\eta,\mu,\sigma\) are morphisms of triangulated functors, the monad is
called \emph{stably separable}.
\end{definition}

\begin{corollary}
\label{cor:separable-monad}
Let $\C$ be an idempotent-complete triangulated category and let
$\mathbb M$ be a stably separable monad on $\C$.  Then the
Eilenberg--Moore category $\mathbb M\text{-}\mathsf{Mod}_\C$ has a unique
triangulation such that a triangle is distinguished if and only if its
image under the forgetful functor is distinguished in $\C$.  The free and
forgetful functors are triangulated functors.
\end{corollary}

\begin{proof}
By \cite[Theorem~4.1 and Corollary~4.3]{Balmer2011}, the
Eilenberg--Moore category has the unique pre-triangulation detected by its
forgetful functor, and the free and forgetful functors are exact.  By
\cite[Proposition~3.11]{Balmer2011}, the forgetful functor is separable.
\Cref{thm:separable-descent} supplies \(\mathrm{TR4}\).
\end{proof}

\begin{remark}
Balmer's construction gives the canonical pre-triangulation above from an
ordinary triangulation
\cite[Theorem~4.1 and Corollary~4.3]{Balmer2011}.  His higher-order
transfer theorem yields the octahedral axiom when \(\C\) carries a
triangulation of order at least three and \(\mathbb M\) is exact to that
order \cite[Theorem~5.17]{Balmer2011}.  The preceding corollary requires
only the ordinary triangulation on \(\C\).  Here \(\mathrm{TR4}\) follows
from \cref{thm:separable-descent}.
\end{remark}

\begin{definition}
An \emph{additive comonad} \(\mathbb L=(L,\epsilon,\delta)\) on an
additive category \(\C\) consists of an additive endofunctor \(L\) and
natural transformations
\[
 \epsilon:L\Rightarrow\operatorname{Id}_\C,
 \qquad
 \delta:L\Rightarrow L^2
\]
such that, for every \(X\in\C\),
\[
 L(\delta_X)\circ\delta_X=\delta_{L(X)}\circ\delta_X,\qquad
 L(\epsilon_X)\circ\delta_X=1_{L(X)}
 =\epsilon_{L(X)}\circ\delta_X.
\]
An \(L\)-comodule is a pair \((X,\rho)\) with \(\rho:X\to L(X)\) such
that
\[
 \epsilon_X\circ\rho=1_X,
 \qquad
 \delta_X\circ\rho=L(\rho)\circ\rho.
\]
A comodule morphism \((X,\rho_X)\to(Y,\rho_Y)\) is a map \(a:X\to Y\)
with \(\rho_Y\circ a=L(a)\circ\rho_X\).  The resulting Eilenberg--Moore
category is denoted by \(\C^{\mathbb L}\).  The comonad is
\emph{separable} if there is a natural transformation
\(\widehat\delta:L^2\Rightarrow L\) such that, for every \(X\in\C\),
\begin{equation}
\label{eq:separable-comonad-identities}
 \widehat\delta_X\circ\delta_X=1_{L(X)},\qquad
 \widehat\delta_{L(X)}\circ L(\delta_X)
 =\delta_X\circ\widehat\delta_X
 =L(\widehat\delta_X)\circ\delta_{L(X)}.
\end{equation}
\end{definition}

\begin{lemma}[{\normalfont see~\cite[Remark~5.19]{Balmer2011}}]
\label{lem:separable-comonad-forgetful}
Let \(\mathbb L\) be an additive separable comonad on an additive category
\(\C\).
Then the forgetful functor \(U:\C^{\mathbb L}\to\C\) is separable.
\end{lemma}

\subsection{Equivariantization}

\begin{definition}[{\normalfont cf.~\cite[Definitions~2.5 and~3.1]
{Sun2019}}]
We use the right-action convention.  An admissible action of a
finite group \(G\) on a triangulated category \(\C\) consists of exact
autoequivalences \((F_g,\theta_g)\), where
\(\theta_g:F_g\Sigma\xrightarrow{\sim}\Sigma F_g\), and natural isomorphisms
\[
 \delta_{g,h}:F_gF_h\xrightarrow{\sim}F_{hg}
\]
such that
\[
 \delta_{hg,\ell}\circ(\delta_{g,h}F_\ell)
 =\delta_{g,\ell h}\circ(F_g\delta_{h,\ell})
\]
and
\[
 \theta_{hg}\circ(\delta_{g,h}\Sigma)
 =(\Sigma\delta_{g,h})\circ(\theta_gF_h)\circ(F_g\theta_h).
\]
The cocycle condition determines a unique unit isomorphism
\(F_e\xrightarrow{\sim}\operatorname{Id}_{\C}\); see
\cite[Remark~2.7]{Sun2019}.  Thus every \(\delta_{g,h}\) is an isomorphism of triangulated functors.  An equivariant object is a
pair
\[
 (X,\{\varphi_g^X\}_{g\in G}),
\]
where \(\varphi_g^X:F_g(X)\xrightarrow{\sim}X\) and
\[
 \varphi_g^X\circ F_g(\varphi_h^X)
 =\varphi_{hg}^X\circ\delta_{g,h,X}.
\]
An equivariant morphism \(a:X\to Y\) satisfies
\[
 a\circ\varphi_g^X=\varphi_g^Y\circ F_g(a)
 \qquad(g\in G).
\]
The resulting category is denoted by \(\C^G\), and
\(\omega:\C^G\to\C\) denotes the forgetful functor.
\end{definition}

An integer \(n\) is \emph{invertible in an additive category \(\C\)} if
multiplication by \(n\) is an automorphism of \(\C(X,Y)\) for every
\(X,Y\in\C\).  For a \(\kk\)-linear category, this holds whenever
\(\operatorname{char}\kk\nmid n\).

\begin{proposition}
\label{prop:equivariant-triangulations}
Let $G$ be a finite group acting admissibly on a triangulated category
$\C$, and assume that \(|G|\) is invertible in \(\C\).  Then
\(\C^G\), endowed with its canonical pre-triangulation, is a triangulated
category.  A triangle in \(\C^G\) is distinguished if and only if its
image under the forgetful functor \(\omega:\C^G\to\C\) is distinguished.
\end{proposition}

\begin{proof}
By \cite[Proposition~3.3]{Sun2019}, $\C^G$ has its canonical
pre-triangulation and $\omega$ is exact.  For $X,Y\in\C^G$, write
\(\varphi_g^X:F_g(\omega(X))\xrightarrow{\sim}\omega(X)\) for the
equivariant structure.  For $a:\omega(X)\to\omega(Y)$ in \(\C\), the
Reynolds operator
\[
 H_{X,Y}(a)=\frac1{|G|}\sum_{g\in G}
 \varphi_g^Y\circ F_g(a)\circ(\varphi_g^X)^{-1}
\]
is equivariant and fixes every equivariant \(a\).  Here \(|G|^{-1}\) is the
inverse of multiplication by \(|G|\) on the relevant Hom group.  Moreover,
for equivariant morphisms \(u:X'\to X\) and \(v:Y\to Y'\), one has
\[
 H_{X',Y'}(v\circ a\circ u)=v\circ H_{X,Y}(a)\circ u.
\]
Thus the maps \(H_{X,Y}\) satisfy \cref{def:separable}.  Hence \(\omega\)
is separable and
\cref{thm:separable-descent} applies.
\end{proof}

In the \(\kk\)-linear setting, Elagin proved the result when \(\C\)
admits a dg enhancement and \(|G|\) is invertible in \(\kk\)
\cite[Corollary~6.10]{Elagin2014}.
Sun proved uniqueness once the canonical triangulation exists and left its
existence open in general \cite[Corollary~3.6 and the paragraph following
Example~3.8]{Sun2019}.  Thus the proposition answers the existence question
under the non-modular hypothesis of \cite[Proposition~3.3]{Sun2019}.  The
consequences in \cite[Sections~3.3--3.5]{Sun2019} then apply under their
remaining hypotheses.

\subsection{Deformations of K3 categories}

We finally apply separable descent to the deformation categories of
Markman--Mehrotra.  Let \(\pi:\mathcal M\to U\) be the family of
irreducible holomorphic symplectic manifolds in
\cite[Theorem~1.8]{MarkmanMehrotra2015}.
For a contractible Stein open subset \(V\subset U\), write
\(\mathcal M_V=\pi^{-1}(V)\), and let \(\theta\) be the Brauer class on
\(\mathcal M_V\) supplied by part~\textup{(2)} of that theorem.
For \(u\in V\), let \(\mathcal M_u\) be the fiber and \(\theta_u\) the
restriction of \(\theta\) to it.  The bounded derived categories of
\(\theta\)-twisted and \(\theta_u\)-twisted coherent sheaves are denoted by
\[
 \C_V=D^b(\mathcal M_V,\theta),
 \qquad
 \C_u=D^b(\mathcal M_u,\theta_u).
\]
Their construction gives exact comonads
\(\mathbb L=(L,\epsilon,\delta)\) on \(\C_V\) and
\(\mathbb L_u=(L_u,\epsilon_u,\delta_u)\) on \(\C_u\).  Let
\[
 \D_V=\C_V^{\mathbb L},
 \qquad
 \D_u=\C_u^{\mathbb L_u}
\]
be the corresponding Eilenberg--Moore categories.  By
\cite[Theorem~1.8(3) and Remark~6.12]{MarkmanMehrotra2015}, they carry
canonical pre-triangulations and the forgetful functors
\[
 U_V:\D_V\longrightarrow\C_V,
 \qquad
 U_u:\D_u\longrightarrow\C_u
\]
are exact.  The geometric hypotheses of
\cite[Theorem~1.8]{MarkmanMehrotra2015} are unconditional by
\cite[Theorem~1.11]{MarkmanMehrotraVerbitsky2019}.

\begin{theorem}
\label{thm:k3-categories}
Retain the notation and geometric setup of
\cite[Theorem~1.8]{MarkmanMehrotra2015}.  Then \(\D_V\) and \(\D_u\),
endowed with their canonical pre-triangulations, are triangulated
categories.
\end{theorem}

\begin{proof}
In the proof of \cite[Theorem~1.8(3)]{MarkmanMehrotra2015}, Markman and
Mehrotra construct a natural retraction
\(\widehat\delta:L^2\Rightarrow L\) of the comultiplication
\(\delta:L\Rightarrow L^2\) and verify the two bicomodule identities; see
\cite[diagram~(6.12) and the subsequent argument]{MarkmanMehrotra2015}.
Hence \(\mathbb L\) is a separable comonad.  Restricting this retraction
to the fiber over \(u\) shows that \(\mathbb L_u\) is separable as well.

By \cref{lem:separable-comonad-forgetful}, the forgetful functors \(U_V\)
and \(U_u\) are separable.  They are exact by
\cite[Theorem~1.8(3)]{MarkmanMehrotra2015}.  Since \(\C_V\) and
\(\C_u\) are triangulated, \cref{thm:separable-descent} applies.
\end{proof}

\section{Scalar families in type-\texorpdfstring{$A_{3a-1}$}{A(3a-1)}
over arbitrary fields}
\label{sec:deformations}

Over an algebraically closed field, Amiot constructed triangulated
structures on categories of finite-dimensional projective modules over
deformed preprojective algebras of generalized Dynkin type
\cite[Theorem~8.1 and Corollary~9.3]{Amiot2007}.  We use the stable
Auslander realization to obtain the base triangulation for type
\(A_{3a-1}\) over an arbitrary field.

By \cref{prop:periodic-orbit-detector}, we need a socle orbit and a short
exact sequence satisfying its hypotheses.  For \(\Pi_\kk(A_{3a-1})\),
an acyclic two-periodic complex of indecomposable projectives gives the
orbit.  Two successive mapping cones give the sequence.  We use the same
construction with period one in the next section.

Throughout this section and \cref{sec:ade-lines}, an underline denotes
passage to the relevant stable module category.

For later applications, if \(\Lambda\) is a finite-dimensional algebra, write
\(\mathcal P(\Lambda)=\operatorname{proj}\Lambda\).

\subsection{Two successive mapping cones}
\label{sec:two-successive-cones}

\begin{lemma}[{\normalfont see~\cite[Lemma~1.3.2 and
Section~1.5]{Weibel1994}}]
\label{lem:cycles-split-cones}
Let \(f:X^\bullet\to Y^\bullet\) be a chain map between acyclic cochain
complexes in an abelian category.  With
\(\operatorname{Cone}(f)^n=X^{n+1}\oplus Y^n\), the degreewise split
sequence
\[
0\longrightarrow Y^\bullet\longrightarrow\operatorname{Cone}(f)
 \longrightarrow X^\bullet[1]\longrightarrow0
\]
induces, for every \(n\), a short exact sequence
\begin{equation}
\label{eq:cycles-of-cone}
0\longrightarrow Z^n(Y^\bullet)\longrightarrow
 Z^n(\operatorname{Cone}(f))
 \longrightarrow Z^{n+1}(X^\bullet)\longrightarrow0.
\end{equation}
\end{lemma}

For a finite-dimensional self-injective algebra \(\Lambda\), write
\(\mathsf K_{\mathrm{ac}}(\operatorname{proj}\Lambda)\) for the homotopy
category of acyclic complexes in \(\operatorname{proj}\Lambda\).
Taking degree-zero cycles gives the standard triangulated equivalence
\[
 Z^0:\mathsf K_{\mathrm{ac}}(\operatorname{proj}\Lambda)
 \xrightarrow{\sim}\underline{\operatorname{mod}}\Lambda,
 \qquad P^\bullet\longmapsto Z^0(P^\bullet);
\]
see \cite[Sections~I.2--I.3]{Happel1988}.  We use the cochain convention
\(P^\bullet[1]^n=P^{n+1}\) and
\(d_{P^\bullet[1]}^n=-d_P^{n+1}\).  A \emph{strict period} \(p\) on a
complex \(X^\bullet\) means a fixed equality
\(X^\bullet[p]=X^\bullet\).  A chain map is compatible with strict
periods if its shifted components agree with its original components.

Let \(P^\bullet\) be an acyclic complex of finitely generated projective
\(\Lambda\)-modules, and let \(h:P^\bullet\to P^\bullet\) be a chain map
with \(h^2=0\).  Put
\(C^\bullet=\operatorname{Cone}(-h)\), ordered as
\(C^n=P^{n+1}\oplus P^n\), and let
\(\pi_C:C^\bullet\to P^\bullet[1]\) be the standard projection.  In components,
the differential, the projection, and a map \(f:P^\bullet[1]\to C^\bullet\) are given by
\begin{equation}
\label{eq:cone-sign-certificate}
 d_C^n=\begin{bmatrix}-d_P^{n+1}&0\\-h^{n+1}&d_P^n\end{bmatrix},
 \qquad
 \pi_C^n=\begin{bmatrix}1&0\end{bmatrix},
 \qquad
 f^n=\begin{bmatrix}h^{n+1}\\0\end{bmatrix}.
\end{equation}
The equality \(h^2=0\) makes \(f\) a chain map, and direct multiplication
gives
\[
 \pi_C\circ f=h[1],
 \qquad
 f\circ h[1]=0.
\]

Suppose that \(Z^0(P^\bullet)=M\) and \(Z^1(P^\bullet)=N\).  Put
\[
 A=Z^0(C^\bullet),
 \qquad
 D^\bullet=\operatorname{Cone}(f),
 \qquad
 B=Z^0(D^\bullet).
\]
Let \(\pi_D:D^\bullet\to P^\bullet[2]\) be the standard projection.
Then
\cref{lem:cycles-split-cones} gives
\begin{equation}
\label{eq:double-cone-cycle-rows}
\begin{tikzcd}[column sep=small,row sep=small]
\text{complexes:}&0\arrow[r]&P^\bullet\arrow[r]&C^\bullet\arrow[r,"\pi_C"]&
 P^\bullet[1]\arrow[r]&0\\
Z^0\text{:}&0\arrow[r]&M\arrow[r]&A\arrow[r]&N\arrow[r]&0\\[2pt]
\text{complexes:}&0\arrow[r]&C^\bullet\arrow[r]&D^\bullet\arrow[r,"\pi_D"]&
 P^\bullet[2]\arrow[r]&0\\
Z^0\text{:}&0\arrow[r]&A\arrow[r]&B\arrow[r]&
 Z^2(P^\bullet)\arrow[r]&0.
\end{tikzcd}
\end{equation}
The rows of complexes are degreewise split and acyclic.  Suppose now that
\[
 P^\bullet[2]=P^\bullet,
 \qquad h[2]=h.
\]
Then the matrices above give
\[
 C^\bullet[2]=C^\bullet,
 \qquad D^\bullet[2]=D^\bullet,
\]
and every morphism in the two rows of complexes in
\eqref{eq:double-cone-cycle-rows} is unchanged by shifting twice.  In particular,
\(Z^2(P^\bullet)=Z^0(P^\bullet)=M\).

If \(z:P^\bullet\to P^\bullet\) is a chain map satisfying
\(h\circ z=0\), then
\begin{equation}
\label{eq:double-cone-lift}
 \sigma_z=\binom{z[2]}{0}:P^\bullet[2]\longrightarrow D^\bullet
\end{equation}
is a chain map and \(\pi_D\circ\sigma_z=z[2]\).  Indeed, the only additional
chain-map equation is \(f[1]\circ z[2]=0\), which is precisely \(h\circ z=0\) by
\eqref{eq:cone-sign-certificate}.  If, in addition, \(z[2]=z\), then the
map induced by \(\sigma_z\) on degree-zero cycles is a lift
\[
 M\longrightarrow B
\]
of \(z|_M:M\to M\) through \(B\to M\).  If also \(z\circ h=0\), then
\(\sigma_z\circ h[2]=0\), so this lift annihilates \(h|_M\).

\begin{lemma}
\label{lem:socle-annihilation}
Let \(M,X\) be objects of an additive category.  If
\(0\ne\rho\in\End(M)\) and
\[
 \Hom(M,X)\circ\rho=0,
\]
then \(X\) has no direct summand isomorphic to \(M\).
\end{lemma}

\begin{proof}
Otherwise there are maps \(s:M\to X\) and \(q:X\to M\) with
\(q\circ s=1_M\).  Then
\(\rho=q\circ(s\circ\rho)=0\), a contradiction.
\end{proof}

\subsection{The stable Auslander model}
\label{sec:uniform-a3a}

Let \(\kk\) be a field.  Fix \(a\ge2\), put \(m=3a\) and
\(n=m-1\), and set
\[
 R_m=\kk[t]/(t^m),\qquad U_i=R_m/(t^i),\qquad
 \Lambda=\Pi_\kk(A_n).
\]
For \(1\le i<m\), let \(e_i\) be the primitive idempotent of \(\Lambda\)
corresponding to vertex \(i\), and put
\[
 P_i=e_i\Lambda.
\]
Thus \(P_i\) is the corresponding indecomposable projective right
\(\Lambda\)-module.  By
\cite[Theorem~4.8]{BrennerButlerKing2002}, \(\Lambda\) is Frobenius,
hence self-injective, so the construction of the preceding subsection
applies.  Write \(S=\Omega_\Lambda^{-1}\) for the suspension of
\(\underline{\operatorname{mod}}\Lambda\).

The algebra \(R_m\) is symmetric, so \(\operatorname{mod}R_m\) is a
Frobenius category and
\(\underline{\operatorname{mod}}R_m\) has suspension
\(\Omega_{R_m}^{-1}\).  For \(R_m\)-modules \(X,Y\), write
\[
 \underline{\Hom}_{R_m}(X,Y)
 =\Hom_{R_m}(X,Y)/
 \{\text{maps factoring through a projective \(R_m\)-module}\}.
\]
Since \(R_m\) is local, its finitely generated projective modules are
free.  Let \(G_m=U_1\oplus\cdots\oplus U_{m-1}\).  The standard stable
Auslander realization gives
\begin{equation}
\label{eq:a3a-stable-auslander}
 \Lambda\xrightarrow{\sim}
 \underline{\End}_{R_m}(G_m)^{\mathrm{op}},
 \qquad
 e_i\Lambda e_j\xrightarrow{\sim}
 \underline{\Hom}_{R_m}(U_i,U_j);
\end{equation}
see \cite[Section~7, pp.~219--220]{DlabRingel1992}.  The opposite algebra
accounts for the reversal of indices.

\begin{remark}
\label{rem:a3a-opposite-convention}
The realization \eqref{eq:a3a-stable-auslander} reverses arrows and
compositions.  More precisely, if
\[
 a:P_j\longrightarrow P_i
 \quad\longleftrightarrow\quad
 \underline{\widetilde a}:U_i\longrightarrow U_j
\]
and
\[
 b:P_\ell\longrightarrow P_j
 \quad\longleftrightarrow\quad
 \underline{\widetilde b}:U_j\longrightarrow U_\ell,
\]
then \(a\circ b\) corresponds to the stable class
\(\underline{\widetilde b\circ\widetilde a}\).  This convention will
also be used in \cref{sec:ade-lines}.
\end{remark}

The modules \(U_1,\ldots,U_{m-1}\) represent all indecomposable
nonprojective \(R_m\)-modules.  Thus \(G_m\) is an additive generator of
\(\underline{\operatorname{mod}}R_m\), and the isomorphisms in
\eqref{eq:a3a-stable-auslander} define a contravariant additive
equivalence
\[
 \mathsf A:\mathcal P(\Lambda)\xrightarrow{\sim}
 \underline{\operatorname{mod}}R_m,
 \qquad
 \mathsf A(P_i)=U_i.
\]

We now record the elementary morphism calculations used below.
\begin{fact}
\label{fact:a3a-cyclic-homs}
\begin{enumerate}[label=\textup{(\arabic*)},leftmargin=2.1em]
\item Let \(1\le i,j<m\).  For \(0\le q<j\), the assignment
\[
 \mu_{ij}^q:U_i\longrightarrow U_j,
 \qquad
 \mu_{ij}^q(1)=t^q,
\]
defines a homomorphism precisely when \(j\le i+q\).  Indeed, a
homomorphism \(U_i\to U_j\) is determined by the image of \(1\), and
\(\mu_{ij}^q\) is well defined precisely when
\(t^{i+q}=0\) in \(U_j\).  Consequently,
\[
 \bigl\{\mu_{ij}^q\bigm|0\le q<j,\ j\le i+q\bigr\}
\]
is a \(\kk\)-basis of \(\Hom_{R_m}(U_i,U_j)\).

\item A homomorphism \(u:U_i\to U_j\) factors through a projective
\(R_m\)-module if and only if
\[
 u(1)\in t^{m-i}U_j.
\]
Indeed, if \(u\) factors through \(R_m^{\oplus d}\), every coordinate of
the image of \(1\) in \(R_m^{\oplus d}\) is annihilated by \(t^i\), and hence belongs to
\(t^{m-i}R_m\).  Conversely, if \(u(1)=t^{m-i}v\), then \(u\) factors as
\[
 U_i\xrightarrow{\alpha}R_m\xrightarrow{\beta}U_j,
 \qquad
 \alpha(1)=t^{m-i},\quad\beta(1)=v.
\]
It follows that \(\mu_{ij}^q\) factors through \(R_m\) precisely when
\(i+q\ge m\).  Hence
\[
 \bigl\{\underline\mu_{ij}^q\bigm|
 0\le q<j,\ j\le i+q<m\bigr\}
\]
is a \(\kk\)-basis of
\(\underline{\Hom}_{R_m}(U_i,U_j)\).

\item For composable monomial maps,
\[
 \underline\mu_{j\ell}^{\,r}\circ\underline\mu_{ij}^{\,q}
 =
 \begin{cases}
  \underline\mu_{i\ell}^{\,q+r},
   &q+r<\ell\text{ and }i+q+r<m,\\
  0,&\text{otherwise}.
 \end{cases}
\]
Indeed, the composite sends \(1\) to \(t^{q+r}\).  This element either
vanishes in \(U_\ell\), represents the displayed stable basis map, or
gives a map factoring through \(R_m\).
\end{enumerate}
\end{fact}

\begin{notation}
\label{not:a3a-generators}
Put \(R_a=\kk[t]/(t^a)\).  Define \(t_i:P_i\to P_i\) for \(i=a,2a\),
\(x:P_a\to P_{2a}\), and \(y:P_{2a}\to P_a\) by
\[
 \mathsf A(t_i)=\underline\mu_{ii}^{1},\qquad
 \mathsf A(x)=\underline\mu_{2a,a}^{0},\qquad
 \mathsf A(y)=\underline\mu_{a,2a}^{a}.
\]
Parts~\textup{(2)}--\textup{(3)} of \cref{fact:a3a-cyclic-homs}
give, for \(i=a,2a\),
\[
 \End_\Lambda(P_i)
 =\bigoplus_{q=0}^{a-1}\kk\,t_i^q
 \cong R_a,
\]
where the isomorphism \(R_a\to\End_\Lambda(P_i)\) sends \(t\) to
\(t_i\).  Applying part~\textup{(2)}
to \((i,j)=(2a,a)\) and \((a,2a)\) gives
\[
\underline{\Hom}_{R_m}(U_{2a},U_a)
 =\bigoplus_{q=0}^{a-1}\kk\,\underline\mu_{2a,a}^q,\qquad
\underline{\Hom}_{R_m}(U_a,U_{2a})
 =\bigoplus_{q=a}^{2a-1}\kk\,\underline\mu_{a,2a}^q.
\]
Via \eqref{eq:a3a-stable-auslander}, these bases yield
\begin{equation}
\label{eq:a3a-four-corners}
\begin{aligned}
 \Hom_\Lambda(P_a,P_{2a})
 &=\bigoplus_{q=0}^{a-1}\kk\,(t_{2a}^q\circ x)
  =\bigoplus_{q=0}^{a-1}\kk\,(x\circ t_a^q),\\
 \Hom_\Lambda(P_{2a},P_a)
 &=\bigoplus_{q=0}^{a-1}\kk\,(t_a^q\circ y)
  =\bigoplus_{q=0}^{a-1}\kk\,(y\circ t_{2a}^q).
\end{aligned}
\end{equation}
Moreover, the generators satisfy
\begin{equation}
\label{eq:a3a-corner-relations}
 t_{2a}\circ x=x\circ t_a,
 \qquad
 t_a\circ y=y\circ t_{2a},
 \qquad
 x\circ y=0=y\circ x.
\end{equation}
The first two identities follow from part~\textup{(3)}.  Under
\eqref{eq:a3a-stable-auslander}, the composite
\(y\circ x\) corresponds to
\(\underline\mu_{2a,a}^0\circ\underline\mu_{a,2a}^a\), represented by
the zero map \(U_a\to U_a\), since \(t^a=0\) in \(U_a\).
The composite \(x\circ y\) corresponds to
\(\underline\mu_{a,2a}^a\circ\underline\mu_{2a,a}^0\).
This stable class is represented by the map \(U_{2a}\to U_{2a}\) sending
\(1\) to \(t^a\), which factors through \(R_m\) because \(2a+a=m\).
\end{notation}

\subsection{Periodicity and exact sequences}
\label{sec:a3a-periodic-detector}

We now construct the socle orbit and the short exact sequence
required by \cref{prop:periodic-orbit-detector}.  Consider the strict
two-periodic complex
\begin{equation}
\label{eq:a3a-periodic-complex}
P^\bullet:\quad
\xymatrix@C=1.8pc{
\cdots\ar[r]&P_a\ar[r]^{x}&P_{2a}\ar[r]^{y}&P_a\ar[r]^{x}&
P_{2a}\ar[r]^{y}&P_a\ar[r]&\cdots
}
\end{equation}
where the first displayed copy of \(P_{2a}\) is placed in degree zero.  Put
\[
 M=\im x\subseteq P_{2a},\qquad N=\im y\subseteq P_a,
 \qquad r=t^{a-1}\in R_a,
\]
and define \(h:P^\bullet\to P^\bullet\) by the following two-periodic
diagram:
\[
\begin{tikzcd}[column sep=2.5em,row sep=large]
\cdots\arrow[r]&
P_a\arrow[r,"x"]\arrow[d,"t_a^{a-1}"']&
P_{2a}\arrow[r,"y"]\arrow[d,"t_{2a}^{a-1}"']&
P_a\arrow[r]\arrow[d,"t_a^{a-1}"']&
\cdots\\
\cdots\arrow[r]&
P_a\arrow[r,"x"']&
P_{2a}\arrow[r,"y"']&
P_a\arrow[r]&
\cdots .
\end{tikzcd}
\]
The relations in \eqref{eq:a3a-corner-relations} make the squares
commute, so \(h\) is a strict two-periodic chain map.  Moreover,
\(h^2=0\), since \(t_i^a=0\) and \(2a-2\ge a\).  Set
\[
 \rho_M=\underline{h|_M},\qquad \rho_N=\underline{h|_N}.
\]

\begin{lemma}
\label{lem:a3a-periodic-exact}
The complex \(P^\bullet\) is exact.
\end{lemma}

\begin{proof}
Let \(\iota:U_a\to U_{2a}\) be given by \(\iota(1)=t^a\), and let
\(\pi:U_{2a}\to U_a\) be the canonical quotient.  Consider the short exact
sequence
\begin{equation*}
\xymatrix@C=2.2pc{
0\ar[r]&U_a\ar[r]^{\iota}&U_{2a}\ar[r]^{\pi}&U_a\ar[r]&0.
}
\end{equation*}
Comparing it with the chosen cosyzygy sequence gives
\[
\begin{tikzcd}[column sep=2.4em]
0\arrow[r]&U_a\arrow[r,"\iota"]\arrow[d,equal]&
U_{2a}\arrow[r,"\pi"]\arrow[d,"\gamma"]&
U_a\arrow[r]\arrow[d,"\partial"]&0\\
0\arrow[r]&U_a\arrow[r,"\jmath"]&
R_m\arrow[r]&U_{2a}\arrow[r]&0,
\end{tikzcd}
\qquad
\jmath(1)=t^{2a},\quad \gamma(1)=t^a.
\]
The bottom row identifies \(\Omega^{-1}_{R_m}U_a\) with \(U_{2a}\), and
the induced map on cokernels satisfies \(\partial(1)=t^a\).  Thus
\(\partial=\iota\).  The top short exact sequence therefore induces the
distinguished triangle
\[
 U_a\xrightarrow{\underline\iota}U_{2a}\xrightarrow{\underline\pi}U_a
 \xrightarrow{\underline\iota}\Omega^{-1}_{R_m}U_a=U_{2a}
\]
in \(\underline{\operatorname{mod}}R_m\).

For \(1\le j<m\), applying
\(\underline{\Hom}_{R_m}(-,U_j)\) to this triangle gives the exact
sequence
\[
 \cdots\xrightarrow{\iota^*}\underline{\Hom}_{R_m}(U_a,U_j)
 \xrightarrow{\pi^*}\underline{\Hom}_{R_m}(U_{2a},U_j)
 \xrightarrow{\iota^*}\underline{\Hom}_{R_m}(U_a,U_j)
 \xrightarrow{\pi^*}\cdots,
\]
where \(\iota^*\) and \(\pi^*\) denote precomposition with
\(\underline\iota\) and \(\underline\pi\), respectively.  Under
\(P_i e_j=e_i\Lambda e_j\cong
\underline{\Hom}_{R_m}(U_i,U_j)\), this is precisely the complex obtained
from \eqref{eq:a3a-periodic-complex} by applying \((-)e_j\): up to signs,
the maps \(x\) and \(y\) correspond to \(\pi^*\) and \(\iota^*\),
respectively.  These signs do not affect exactness.  Since
\(1=e_1+\cdots+e_{m-1}\), exactness of all the complexes
\(P^\bullet e_j\) implies that \(P^\bullet\) is exact.
\end{proof}

The relations in \eqref{eq:a3a-corner-relations} show that \(t_{2a}\)
preserves \(M\) and \(t_a\) preserves \(N\).  Restriction, followed by
passage to the stable category, therefore defines algebra homomorphisms
\begin{equation}
\label{eq:a3a-stable-homs}
\begin{aligned}
 \operatorname{res}_M:R_a&\longrightarrow
 \underline{\End}_\Lambda(M),
 &t&\longmapsto\underline{t_{2a}|_M},\\
 \operatorname{res}_N:R_a&\longrightarrow
 \underline{\End}_\Lambda(N),
 &t&\longmapsto\underline{t_a|_N}.
\end{aligned}
\end{equation}

\begin{lemma}
\label{lem:a3a-rank-one}
\begin{enumerate}[label=\textup{(\arabic*)}]
\item The two maps in \eqref{eq:a3a-stable-homs} are isomorphisms, and
\[
 \underline{\Hom}_\Lambda(M,N)=0
 =\underline{\Hom}_\Lambda(N,M).
\]
\item Under these isomorphisms, \(r=t^{a-1}\) corresponds to
\(\rho_M\) and \(\rho_N\).  In particular, they are nonzero and generate
the respective socles:
\[
 \operatorname{soc}\bigl(\underline{\End}_\Lambda(M)\bigr)=\kk\rho_M,
 \qquad
 \operatorname{soc}\bigl(\underline{\End}_\Lambda(N)\bigr)=\kk\rho_N.
\]
The strict period of \(P^\bullet\) identifies
\[
 S(M)\cong N,\qquad S(N)\cong M,
\]
and, under these identifications,
\[
 S(\rho_M)=\rho_N,\qquad S(\rho_N)=\rho_M.
\]
\end{enumerate}
\end{lemma}

\begin{proof}
For~\textup{(1)}, consider first the restriction map
\[
 \End_\Lambda(P_{2a})\longrightarrow\End_\Lambda(M),
 \qquad v\longmapsto v|_M.
\]
It is surjective because \(P_{2a}\) is injective.  If \(v|_M=0\), then
\(v\circ x=0\).  By \eqref{eq:a3a-four-corners}, precomposition with
\(x\) sends the basis \(1,t_{2a},\ldots,t_{2a}^{a-1}\) to a basis of
\(\Hom_\Lambda(P_a,P_{2a})\).  Hence \(v=0\).  Thus restriction is an
isomorphism.

We next show that every endomorphism of \(M\) factoring through a
projective module is zero.  Let
\(\iota_M:M\hookrightarrow P_{2a}\) and
\(\bar x:P_a\twoheadrightarrow M\) be the maps satisfying
\(x=\iota_M\circ\bar x\).  Suppose that
\(u=b'\circ a':M\to M\) factors through a projective module \(Q\).
Since \(\Lambda\) is self-injective, \(a'\) extends to a map
\(\widetilde a:P_{2a}\to Q\), so that
\(a'=\widetilde a\circ\iota_M\).  Projectivity of \(Q\) lifts \(b'\)
through \(\bar x\) to a map \(\widetilde b:Q\to P_a\) satisfying
\(\bar x\circ\widetilde b=b'\).  Put
\(w=\widetilde b\circ\widetilde a:P_{2a}\to P_a\).  Then
\[
 \iota_M\circ u=x\circ w\circ\iota_M.
\]
By \eqref{eq:a3a-four-corners}, \(w\) is a linear combination of the
maps \(t_a^q\circ y\).  Hence \(x\circ w=0\) by
\eqref{eq:a3a-corner-relations}, and therefore \(u=0\).
This proves that \(\operatorname{res}_M\) is an isomorphism.  A similar
argument proves that \(\operatorname{res}_N\) is an isomorphism.

We next prove that \(\Hom_\Lambda(M,N)=0\).  Let \(u:M\to N\).  The composite
\(M\xrightarrow{u}N\hookrightarrow P_a\) extends across
\(M\hookrightarrow P_{2a}\) to a map \(P_{2a}\to P_a\).  Every such map
is a linear combination of the maps \(t_a^q\circ y\), all of which
vanish on \(M=\im x\).  Thus \(\Hom_\Lambda(M,N)=0\).  A similar argument
gives \(\Hom_\Lambda(N,M)=0\).

For~\textup{(2)}, the definition of \(h\) shows that \(r=t^{a-1}\)
corresponds to \(\rho_M\) and \(\rho_N\).  Since
\(\operatorname{soc}(R_a)=\kk t^{a-1}\), the two displayed socle
identities follow.  Under the triangulated equivalence \(Z^0\) recalled
above, shifting \(P^\bullet\) once identifies \(S(M)\) with \(N\) and
\(S(N)\) with \(M\).  The strict two-periodicity of \(h\) gives
\(S(\rho_M)=\rho_N\) and \(S(\rho_N)=\rho_M\).
\end{proof}

By \cref{lem:a3a-periodic-exact}, we may apply the two successive mapping-cone constructions in
\eqref{eq:double-cone-cycle-rows} to \eqref{eq:a3a-periodic-complex} and
\(h\).  Thus
\(C^\bullet=\operatorname{Cone}(-h)\), \(A=Z^0(C^\bullet)\),
\(D^\bullet=\operatorname{Cone}(f)\), and
\(B=Z^0(D^\bullet)\).  Let
\[
 p=Z^0(\pi_D):B\longrightarrow Z^2(P^\bullet)=M.
\]
The last row of \eqref{eq:double-cone-cycle-rows} is the short exact
sequence
\begin{equation}
\label{eq:a3a-detector-sequence}
\begin{tikzcd}[column sep=2.5em]
0\arrow[r]&A\arrow[r]&B\arrow[r,"p"]&M\arrow[r]&0.
\end{tikzcd}
\end{equation}

\begin{lemma}
\label{lem:a3a-two-cones}
Neither \(A\) nor \(B\) has a direct summand isomorphic, in the stable
category, to \(M\) or \(N\).  There are isomorphisms
\[
 \phi_B:S^2(B)\xrightarrow{\sim}B,
 \qquad
 \phi_M:S^2(M)\xrightarrow{\sim}M
\]
satisfying
\[
 \phi_M\circ S^2(\underline p)=\underline p\circ\phi_B.
\]
\end{lemma}

\begin{proof}
Let \(a_0:M\to A\) and \(b:A\to N\) be the maps in the first short exact
sequence of \eqref{eq:double-cone-cycle-rows}.  Under the identification
\(S(M)\cong N\) induced by \(P^\bullet\), its connecting morphism is
\(\rho_N\) by \eqref{eq:cone-sign-certificate}.  Hence this short exact
sequence induces the triangle
\begin{equation}
\label{eq:a3a-first-cone-triangle}
\begin{tikzcd}[column sep=2.1em]
M\arrow[r,"\underline{a_0}"]&A\arrow[r,"\underline b"]&N\arrow[r,"\rho_N"]&N.
\end{tikzcd}
\end{equation}
The chain map \(f:P^\bullet[1]\to C^\bullet\) induces a module
homomorphism \(f:N\to A\) satisfying
\[
 \underline b\circ\underline f=\rho_N.
\]
By \cref{lem:socle-annihilation}, it is enough to prove
\begin{equation}
\label{eq:a3a-annihilation-identities}
\begin{aligned}
 \underline{\Hom}_\Lambda(M,A)\circ\rho_M&=0,&
 \underline{\Hom}_\Lambda(N,A)\circ\rho_N&=0,\\
 \underline{\Hom}_\Lambda(M,B)\circ\rho_M&=0,&
 \underline{\Hom}_\Lambda(N,B)\circ\rho_N&=0.
\end{aligned}
\end{equation}

We begin with \(A\).  If \(u:M\to A\) is a module homomorphism, then
\(\underline{\Hom}_\Lambda(M,N)=0\) and exactness applied to
\eqref{eq:a3a-first-cone-triangle} give
\(\underline u=\underline{a_0}\circ c\) for some \(c\in R_a\).
The map preceding \(\underline{a_0}\) in the triangle is \(\pm\rho_M\),
so \(\underline{a_0}\circ\rho_M=0\).  Since \(R_a\) is commutative,
\[
 \underline u\circ\rho_M
 =\underline{a_0}\circ c\circ\rho_M
 =\underline{a_0}\circ\rho_M\circ c=0.
\]
If \(u:N\to A\) is a module homomorphism, exactness and
\(\underline{\Hom}_\Lambda(N,M)=0\) show that
\(\underline b\circ-:\underline{\Hom}_\Lambda(N,A)\to
\underline{\End}_\Lambda(N)\) is injective.
Consecutive maps in \eqref{eq:a3a-first-cone-triangle} have zero composite,
so \(\rho_N\circ\underline b=0\).  Commutativity of
\(\underline{\End}_\Lambda(N)\cong R_a\) gives
\[
 \underline b\circ\underline u\circ\rho_N
 =\rho_N\circ\underline b\circ\underline u=0.
\]
Injectivity now yields \(\underline u\circ\rho_N=0\).

Let \(g:A\to B\) be induced by the inclusion \(C^\bullet\to D^\bullet\).
The second cone construction gives the triangle
\begin{equation}
\label{eq:a3a-second-cone-triangle}
\begin{tikzcd}[column sep=2.2em]
N\arrow[r,"\underline f"]&A\arrow[r,"\underline g"]&B\arrow[r,"\underline p"]&M.
\end{tikzcd}
\end{equation}
Since \(\underline{\Hom}_\Lambda(N,M)=0\), the stable class of every
module homomorphism \(N\to B\) has the form
\(\underline g\circ\underline u\) for some \(u:N\to A\).
The result for \(A\) therefore gives
\(\underline{\Hom}_\Lambda(N,B)\circ\rho_N=0\).

It remains to treat morphisms \(M\to B\).
Let \(v:M\to B\) be a module homomorphism and set
\(z=\underline p\circ\underline v\in
\underline{\End}_\Lambda(M)\cong R_a\).
Exactness of \eqref{eq:a3a-second-cone-triangle} gives
\(S(\underline f)\circ z=0\).
Since \(\underline b\circ\underline f=\rho_N\ne0\),
\(\underline f\ne0\).  Consequently \(z\) is not a unit.
Thus \(z\in tR_a=\operatorname{Ann}_{R_a}(r)\), and we may write
\[
 z=\sum_{q=1}^{a-1}c_qt^q.
\]
Define a strict two-periodic endomorphism
\(z^\bullet:P^\bullet\to P^\bullet\) by taking its component on \(P_i\)
to be \(\sum_{q=1}^{a-1}c_qt_i^q\), for \(i=a,2a\).
The relations \eqref{eq:a3a-corner-relations} show that this is a chain map.
Its restriction to \(M\) represents \(z\) under \(\operatorname{res}_M\).
Since \(rz=0=zr\), we have
\(h\circ z^\bullet=0=z^\bullet\circ h\).
Hence \(\sigma_{z^\bullet}\) from \eqref{eq:double-cone-lift} induces a
module homomorphism \(\widetilde z:M\to B\) satisfying
\[
 \underline p\circ\underline{\widetilde z}=z,
 \qquad
 \underline{\widetilde z}\circ\rho_M=0.
\]
Now \(\underline p\circ(\underline v-\underline{\widetilde z})=0\).
Exactness gives
\(\underline v-\underline{\widetilde z}
 =\underline g\circ\underline u\) for some \(u:M\to A\).
The result for \(A\) and the displayed properties of \(\widetilde z\) give
\[
 \underline v\circ\rho_M
 =\underline{\widetilde z}\circ\rho_M
  +\underline g\circ\underline u\circ\rho_M=0.
\]
This proves \eqref{eq:a3a-annihilation-identities}, and hence the first
assertion.

Finally, \(P^\bullet[2]=P^\bullet\) and \(h[2]=h\).  Since the cone
constructions are defined by the matrices in
\eqref{eq:cone-sign-certificate}, these equalities give
\[
 C^\bullet[2]=C^\bullet,\qquad
 f[2]=f,\qquad
 D^\bullet[2]=D^\bullet.
\]
Under these identifications, the projection \(\pi_D\) is
two-periodic, as expressed by the following commutative square:
\[
\begin{tikzcd}[column sep=large,row sep=large]
D^\bullet[2]\arrow[r,"{\pi_D[2]}"]\arrow[d,equal]&
P^\bullet[4]\arrow[d,equal]\\
D^\bullet\arrow[r,"\pi_D"']&P^\bullet[2]
\end{tikzcd}
\]
The triangulated equivalence \(Z^0\) identifies the shift \([2]\) with
\(S^2\).  Hence the strict equalities above induce isomorphisms
\[
 \phi_B:S^2(B)\xrightarrow{\sim}B,
 \qquad
 \phi_M:S^2(M)\xrightarrow{\sim}M.
\]
Applying \(Z^0\), together with its suspension comparison, to the last
diagram gives
\[
\begin{tikzcd}[column sep=large,row sep=large]
S^2(B)\arrow[r,"S^2(\underline p)"]\arrow[d,"\phi_B"']&
S^2(M)\arrow[d,"\phi_M"]\\
B\arrow[r,"\underline p"']&M,
\end{tikzcd}
\]
which proves the second assertion.
\end{proof}

\begin{theorem}
\label{thm:scalar-family}
Let \(\kk\) be a field and \(a\ge2\).  For
\(\Lambda=\Pi_\kk(A_{3a-1})\), the Hom-finite idempotent-complete
\(\kk\)-linear category \(\mathcal P(\Lambda)\) carries a family
\(\{\Delta_{\kk,\lambda}\}_{\lambda\in\kk}\) of pairwise distinct
pre-triangulations such that
\[
 \Delta_{\kk,\lambda}\text{ satisfies }\mathrm{TR4}
 \quad\Longleftrightarrow\quad\lambda=0.
\]
\end{theorem}

\begin{proof}
Since \(R_m\) is commutative and symmetric, the duality
\(D=\Hom_\kk(-,\kk)\) is a contravariant autoequivalence of
\(\underline{\operatorname{mod}}R_m\).  Since \(\mathsf A\) is also
contravariant, their composite is a covariant additive equivalence
\[
 F=D\circ\mathsf A:\mathcal P(\Lambda)\xrightarrow{\sim}
 \underline{\operatorname{mod}}R_m.
\]
Transport the canonical triangulation of
\(\underline{\operatorname{mod}}R_m\) along \(F\), and denote the
resulting triangulation on \(\mathcal P(\Lambda)\) by
\(\Delta_{\kk,0}\).  Moreover,
\(\underline{\operatorname{mod}}\Lambda\) is \(2\)-Calabi--Yau
\cite[Section~8.4, Corollary~1, and Section~8.5, Lemma~2]{Keller2005}.

By \cref{lem:a3a-rank-one},
\(\underline{\End}_\Lambda(X)\cong R_a\) for \(X=M,N\), while
\(\underline{\Hom}_\Lambda(M,N)=0=
\underline{\Hom}_\Lambda(N,M)\).  Since \(R_a\) is local, \(M\) and
\(N\) are indecomposable in the stable category.  The vanishing of the
two morphism spaces shows that they are not isomorphic.  The strict
period of \(P^\bullet\) identifies
\[
 S(M)\cong N,\qquad S(N)\cong M,
\]
and carries \(\rho_M\) to \(\rho_N\) and \(\rho_N\) to \(\rho_M\).
Thus \(\mathcal O=\{M,N\}\) is an \(S\)-orbit.  For \(X=M,N\), the same
lemma gives
\[
 \underline{\End}_\Lambda(X)/
 \operatorname{rad}\bigl(\underline{\End}_\Lambda(X)\bigr)\cong\kk,
 \qquad
 \operatorname{soc}\bigl(\underline{\End}_\Lambda(X)\bigr)=\kk\rho_X.
\]
Since \(a\ge2\), \(t^{a-1}\) is a nonzero element of
\(\operatorname{rad}(R_a)\).  Hence each \(\rho_X\) is nonzero and belongs
to the radical.
Consequently, \((\mathcal O,\rho)\) is a socle orbit.

In the short exact sequence \eqref{eq:a3a-detector-sequence}, the
right-hand term \(M\) belongs to \(\mathcal O\), and
\cref{lem:a3a-two-cones} verifies conditions~\textup{(2)}--\textup{(3)}
of \cref{prop:periodic-orbit-detector}.  Applying that proposition gives
the required twists \(\Delta_{\kk,\lambda}\) and the asserted
equivalence.
\end{proof}

\section{Dynkin preprojective algebras in characteristic two}
\label{sec:ade-lines}

We construct central Heller twists for Dynkin preprojective algebras
in characteristic two.  A criterion for periodic modules reduces the
construction to endomorphisms of a projective module.  We then verify
its hypotheses using cyclic modules in type \(A_n\) and a common
dimension argument in types \(D_n,E_6,E_7,E_8\).

\begin{notation}
\label{not:ade-modules}
Throughout this section, \(\kk\) is a field of characteristic two,
and all modules are finite-dimensional right modules.  For a
self-injective algebra \(\Lambda\), we write
\(S=\Omega_\Lambda^{-1}\) for the suspension of
\(\underline{\operatorname{mod}}\Lambda\).  The stable class of a
homomorphism \(f\) is denoted by \(\underline f\).
For a quotient vector space or algebra \(B/J\), the notation
\([x]\) means the coset \(x+J\).  For vectors \(v_1,\ldots,v_r\),
\(\kk\{v_1,\ldots,v_r\}\) denotes their \(\kk\)-linear span.

Let \(Q\) be a simply laced Dynkin diagram, with vertex set \(Q_0\).
Its doubled quiver \(\overline Q\) has two opposite arrows
\(x:i\to j\) and \(\bar x:j\to i\) for each edge.  Put
\[
 r_i=\sum_{x:i\to j}x\bar x,\qquad
 I_Q=(r_i\mid i\in Q_0),\qquad
 \Lambda_Q=\Pi_\kk(Q)=\kk\overline Q/I_Q,
\]
where the sum is over all arrows starting at \(i\).
Products of paths are read from left to right.  We use the same
symbol for a path and its image in \(\Lambda_Q\).  For a path
\(\gamma=x_1\cdots x_\ell\), put
\(\bar\gamma=\bar x_\ell\cdots\bar x_1\).

The trivial path at \(i\) is denoted by \(e_i\), and
\[
 P_i=e_i\Lambda_Q,\qquad S_i=P_i/\operatorname{rad}(P_i).
\]
Thus \(S_i\) is the simple module supported at \(i\).
The grading on \(\Lambda_Q\) is by path length.  When
\(\Lambda=\Lambda_Q\), its degree-\(\ell\) component is written
\(\Lambda_\ell\).  For \(c\in e_i\Lambda e_j\), let
\[
 L_c:e_j\Lambda\longrightarrow e_i\Lambda,\qquad w\longmapsto cw.
\]
In particular, left multiplication identifies \(e_i\Lambda e_i\)
with \(\End_\Lambda(e_i\Lambda)\), and
\(L_c\circ L_{c'}=L_{cc'}\) whenever the products are defined.
\end{notation}

\begin{fact}
\label{fact:ade-triangulations}
The following properties hold for \(\Lambda_Q\).
\begin{enumerate}[label=\textup{(\roman*)},leftmargin=2.1em]
\item The algebra \(\Lambda_Q\) is finite-dimensional and Frobenius,
and hence self-injective
\cite[Theorem~4.8]{BrennerButlerKing2002}.
\item Keller's orbit construction gives a triangulation \(\Delta_0\)
on \(\mathcal P(\Lambda_Q)\)
\cite[Section~7.3]{Keller2005}; the corrected formulation using the
triangulated hull is in \cite[Section~1]{KellerCorrection}.
\item The stable module category
\(\underline{\operatorname{mod}}\Lambda_Q\) is \(2\)-Calabi--Yau
\cite[Section~8.4, Corollary~1, and Section~8.5, Lemma~2]{Keller2005}.
The suspension and Serre-functor conventions are explained in
\cite[Sections~2.3--2.4]{KellerCorrection}.
\end{enumerate}
\end{fact}

As in \cref{eq:L2-Freyd-evaluation}, evaluation at \(\Lambda_Q\)
identifies the Freyd category of \(\mathcal P(\Lambda_Q)\) with
\(\operatorname{mod}\Lambda_Q\).  Thus
\cref{prop:periodic-orbit-detector} applies to short exact sequences
of \(\Lambda_Q\)-modules.

\subsection{A criterion using periodic modules}

We first express the stable endomorphism algebra of a periodic module
as a quotient of the endomorphism algebra of a projective module.

\begin{lemma}
\label{lem:periodic-block-stable-endomorphisms}
Let \(\Lambda\) be finite-dimensional and self-injective, let \(P\)
be projective, and put
\(E=\End_\Lambda(P)\).  Suppose that \(d\in E\) satisfies
\(\im d=\ker d=:M\).  Then restriction induces an algebra isomorphism
\begin{equation}
\label{eq:periodic-block-stable-endomorphisms}
 \frac{\{x\in E\mid d\circ x\circ d=0\}}{dE+Ed}
 \xrightarrow{\sim}\underline{\End}_\Lambda(M),
 \qquad [x]\longmapsto\underline{x|_M}.
\end{equation}
Here \(dE=\{d\circ y\mid y\in E\}\) and
\(Ed=\{y\circ d\mid y\in E\}\).
\end{lemma}

\begin{proof}
Let \(\iota:M\hookrightarrow P\) be the inclusion and let
\(\pi:P\twoheadrightarrow M\) be the corestriction of \(d\), so
that \(d=\iota\circ\pi\).  Since \(\im d=\ker d=M\),
the condition \(d\circ x\circ d=0\) is equivalent to
\(x(M)\subseteq M\).  Thus the numerator in
\eqref{eq:periodic-block-stable-endomorphisms} is a subalgebra of
\(E\), and restriction defines an algebra homomorphism
\[
 \operatorname{res}:
 \{x\in E\mid x(M)\subseteq M\}\longrightarrow\End_\Lambda(M).
\]
This homomorphism is surjective.  Indeed, for
\(f\in\End_\Lambda(M)\), injectivity of \(P\) extends
\(\iota\circ f:M\to P\) along \(\iota\) to an endomorphism
\(x\) of \(P\), and then \(x|_M=f\).

We first compute the ordinary kernel of \(\operatorname{res}\).
If \(x|_M=0\), then \(\ker\pi=M\) gives a factorization
\(x=s\circ\pi\), with \(s:M\to P\).  Extend \(s\) along
\(\iota\) to \(y\in E\).  Then \(x=y\circ d\).
Conversely, every element of \(Ed\) vanishes on \(M\), so
\(\ker(\operatorname{res})=Ed\).

Now suppose that \(x|_M=\beta\circ\alpha\) factors through a
projective module \(Q\), where \(\alpha:M\to Q\) and
\(\beta:Q\to M\).  Projectivity of \(Q\) gives a lift
\(\widetilde\beta:Q\to P\) with
\(\pi\circ\widetilde\beta=\beta\).  By injectivity of \(P\),
the map \(\widetilde\beta\circ\alpha:M\to P\) extends to
\(y\in E\).  Hence
\[
 x\circ\iota
 =\iota\circ\beta\circ\alpha
 =d\circ y\circ\iota,
\]
and therefore \(x-d\circ y\in Ed\).
Conversely, \((d\circ y)|_M\), regarded as an endomorphism of
\(M\), factors through \(P\), and \((y\circ d)|_M=0\).
The kernel after passage to stable endomorphisms is consequently
\(dE+Ed\), as required.
\end{proof}

\begin{proposition}
\label{prop:strict-period-one-detector}
Let \(\kk\) be a field of characteristic two, and let \(\Lambda\) be a
finite-dimensional self-injective \(\kk\)-algebra.  Assume that
\(\underline{\operatorname{mod}}\Lambda\) is \(2\)-Calabi--Yau
and that \(\mathcal P(\Lambda)\) carries a triangulation
\(\Delta_0\).  Let \(P\) be projective and put
\(E=\End_\Lambda(P)\).  Suppose that \(d,h\in E\) satisfy
\begin{equation}
\label{eq:ade-block-certificate}
\begin{gathered}
 \im d=\ker d=:M,\qquad d\circ h=h\circ d,\qquad h^2=0,\\
 \frac{\{x\in E\mid d\circ x\circ d=0\}}{dE+Ed}
 \cong\kk[\varepsilon]/(\varepsilon^2),
 \qquad [1_P]\longmapsto1,\quad[h]\longmapsto\varepsilon.
\end{gathered}
\end{equation}
Then \((\{M\},\underline{h|_M})\) is a socle orbit.  Its central
Heller twists \(\Delta_\lambda\) of \(\Delta_0\) are pairwise
distinct, and
\[
 \Delta_\lambda\text{ satisfies }\mathrm{TR4}
 \quad\Longleftrightarrow\quad\lambda=0.
\]
\end{proposition}

\begin{proof}
Put \(S=\Omega_\Lambda^{-1}\) and \(\rho=\underline{h|_M}\).  By
\cref{lem:periodic-block-stable-endomorphisms},
\[
 \underline{\End}_\Lambda(M)=\kk1_M\oplus\kk\rho,
 \qquad \rho^2=0\ne\rho.
\]
Thus \(M\) is indecomposable in the stable category, and the radical
and socle of its stable endomorphism algebra are both \(\kk\rho\).
The exact sequence
\[
 0\longrightarrow M\longrightarrow P\xrightarrow{d}M
 \longrightarrow0
\]
gives a connecting isomorphism \(\kappa_M:M\xrightarrow{\sim}S(M)\).
Since \(d\circ h=h\circ d\), naturality gives
\[
 \kappa_M\circ\rho=S(\rho)\circ\kappa_M.
\]
Transporting \(\rho\) to \(S(M)\) along \(\kappa_M\), we therefore have
\[
 \rho_{S(M)}=\kappa_M\circ\rho\circ\kappa_M^{-1}=S(\rho).
\]
Hence \((\{M\},\rho)\) is a socle orbit in the sense of
\cref{def:periodic-socle-orbit}.

Apply the construction of \cref{sec:two-successive-cones} to the acyclic
complex
\[
 P^\bullet=(\cdots\xrightarrow{d}P\xrightarrow{d}P
 \xrightarrow{d}P\xrightarrow{d}\cdots)
\]
and the chain endomorphism with every component equal to \(h\).
As \(\operatorname{char}\kk=2\), both \(P^\bullet\) and \(h\) have
strict period one.  Retain the notation \(C^\bullet,D^\bullet,f\) and
\(\sigma_h\) from \eqref{eq:cone-sign-certificate} and
\eqref{eq:double-cone-lift}, and put
\[
 A=Z^0(C^\bullet),\qquad B=Z^0(D^\bullet).
\]
The cycle sequences \eqref{eq:double-cone-cycle-rows} are
\begin{equation}
\label{eq:ade-periodic-rows}
 0\longrightarrow M\xrightarrow{u}A\xrightarrow{v}M\longrightarrow0,
 \qquad
 0\longrightarrow A\xrightarrow{j}B\xrightarrow{p}M\longrightarrow0.
\end{equation}
Here \(u,v,j,p\) are induced by the inclusions and projections in
\eqref{eq:double-cone-cycle-rows}.  We also write \(f:M\to A\) and
\(\sigma_h:M\to B\) for the induced maps on cycles.

The cone complexes and all these chain maps have strict period one.
As for \(M\), the exact cycle sequences of \(C^\bullet,D^\bullet\)
give connecting isomorphisms
\(\kappa_A:A\xrightarrow{\sim}S(A)\) and
\(\kappa_B:B\xrightarrow{\sim}S(B)\).
These isomorphisms are natural for the maps on cycles above: for each
such map \(t:X\to Y\),
\begin{equation}
\label{eq:ade-period-naturality}
 \kappa_Y\circ\underline t=S(\underline t)\circ\kappa_X.
\end{equation}
Consequently the isomorphisms
\[
 \phi_X=(S(\kappa_X)\circ\kappa_X)^{-1}:
 S^2(X)\xrightarrow{\sim}X\qquad(X=B,M)
\]
satisfy
\(\underline p\circ\phi_B=\phi_M\circ S^2(\underline p)\).
Thus the second sequence in \eqref{eq:ade-periodic-rows} has the
compatible period identifications required by
\cref{prop:periodic-orbit-detector}.

The mapping-cone triangles of \cref{sec:two-successive-cones}, under
the triangulated equivalence \(Z^0\), give
\[
 M\xrightarrow{\makebox[2.6em]{$\scriptstyle\underline u$}}
 A\xrightarrow{\makebox[2.6em]{$\scriptstyle\underline v$}}
 M\xrightarrow{\makebox[2.6em]{$\scriptstyle\kappa_M\circ\rho$}}S(M),
 \qquad
 M\xrightarrow{\makebox[2.6em]{$\scriptstyle\underline f$}}
 A\xrightarrow{\makebox[2.6em]{$\scriptstyle\underline j$}}
 B\xrightarrow{\makebox[2.6em]{$\scriptstyle\kappa_M\circ\underline p$}}S(M).
\]
By \eqref{eq:ade-period-naturality}, the zero composites in these
triangles and their rotations give
\(\underline u\circ\rho=0\) and
\(\underline f\circ\underline p=0\).
The identities in \eqref{eq:cone-sign-certificate} and
\eqref{eq:double-cone-lift}, with \(z=h\), also give
\[
 \underline v\circ\underline f=\rho,\qquad
 \underline f\circ\rho=0,\qquad
 \underline p\circ\underline{\sigma_h}=\rho,\qquad
 \underline{\sigma_h}\circ\rho=0.
\]

It remains to exclude direct summands isomorphic to \(M\) in \(A\)
and \(B\).  Let \(g:M\to A\) be a module homomorphism.
The first triangle gives
\(\rho\circ\underline v\circ\underline g=0\), whence
\(\underline v\circ\underline g=c\rho\) for some \(c\in\kk\).
Exactness of \(\underline{\Hom}_\Lambda(M,-)\) gives
\(\underline g-c\underline f=\underline u\circ w\) for some
\(w\in\underline{\End}_\Lambda(M)\).  Since
\(w\circ\rho\in\kk\rho\), it follows that
\(\underline g\circ\rho=0\).

Now let \(g:M\to B\) be a module homomorphism.
Since \(\underline f\ne0\) and
\(\underline f\circ\underline p=0\), the endomorphism
\(\underline p\circ\underline g\) is not invertible.
Thus \(\underline p\circ\underline g=c\rho\) for some \(c\in\kk\).
Exactness of the second triangle gives
\(\underline g-c\underline{\sigma_h}=\underline j\circ a\) for some
\(a\in\underline{\Hom}_\Lambda(M,A)\).  The preceding paragraph
shows that \(a\circ\rho=0\), and hence
\(\underline g\circ\rho=0\).
By \cref{lem:socle-annihilation}, neither \(A\) nor \(B\) has a direct
summand isomorphic to \(M\) in the stable category.
The second sequence in \eqref{eq:ade-periodic-rows} therefore satisfies
all the hypotheses of \cref{prop:periodic-orbit-detector}.
\end{proof}
\subsection{Type \texorpdfstring{$A_n$}{An}}

Fix \(n\ge6\), put \(m=n+1\), and write
\[
 R_m=\kk[t]/(t^m),\qquad U_i=R_m/(t^i),\qquad
 \Lambda=\Pi_\kk(A_n),\qquad P_i=e_i\Lambda.
\]
Set \(X=U_2\oplus U_{m-2}\) and \(P=P_2\oplus P_{m-2}\).
The stable Auslander realization recalled in \cref{sec:uniform-a3a}
\cite[Section~7, pp.~219--220]{DlabRingel1992} gives an algebra
isomorphism
\begin{equation}
\label{eq:an-corner-anti-isomorphism}
 \Phi:\End_\Lambda(P)^{\mathrm{op}}
 \xrightarrow{\sim}
 C:=\underline{\End}_{R_m}(X).
\end{equation}
The cyclic-module formulas in \cref{fact:a3a-cyclic-homs} hold for
arbitrary \(m\), with the same proof.  Retain the notation
\(\mu_{ij}^q\), and let \(t_i=\mu_{ii}^1\) for \(i=2,m-2\).
Define endomorphisms of \(X\) by
\begin{equation}
\label{eq:an-delta-eta}
 \delta=
 \begin{bmatrix}
 t_2&\mu_{m-2,2}^0\\
 \mu_{2,m-2}^{m-4}&t_{m-2}
 \end{bmatrix},\qquad
 \eta=
 \begin{bmatrix}
 0&0\\
 \mu_{2,m-2}^{m-4}&0
 \end{bmatrix}.
\end{equation}
Let \(d,h\in\End_\Lambda(P)\) correspond under \(\Phi\) to
\(\underline\delta,\underline\eta\), respectively.

\begin{proposition}
\label{prop:an-period-one}
The endomorphisms \(d,h\) satisfy \eqref{eq:ade-block-certificate}.
Consequently \cref{prop:strict-period-one-detector} applies.
\end{proposition}

\begin{proof}
We first prove exactness.  Direct multiplication on \(X\) gives
\[
 \delta^2=
 \begin{bmatrix}0&0\\0&t^2(1+t^{m-6})\end{bmatrix}.
\]
Indeed, the off-diagonal entries cancel in characteristic two, while
\(\mu_{2,m-2}^{m-4}\circ\mu_{m-2,2}^0+t_{m-2}^2\) is
multiplication by \(t^{m-4}+t^2\) on \(U_{m-2}\).
The map
\[
 \ell:U_{m-2}\longrightarrow R_m,\qquad
 \ell(\bar z)=t^2(1+t^{m-6})z
\]
is well defined and injective: \(m\ge7\), so \(1+t^{m-6}\) is a
unit of \(R_m\), and multiplication by \(t^2\) identifies
\(U_{m-2}\) with the ideal \(t^2R_m\).
Let \(\alpha:X\to R_m\) be given by \(\alpha(x,y)=\ell(y)\), and
let \(\epsilon:R_m\to X\) be given by \(\epsilon(c)=(0,\bar c)\).
Then
\begin{equation}
\label{eq:an-short-certificate}
\begin{tikzcd}[column sep=3.6em]
 0 \arrow[r] & X \arrow[r,"{\binom{\delta}{\alpha}}"]
   & X\oplus R_m \arrow[r,"{(\delta\ -\epsilon)}"]
   & X \arrow[r] & 0
\end{tikzcd}
\end{equation}
is exact.  The composite is zero because
\(\delta^2=\epsilon\circ\alpha\).
The first map is injective since \(\ell\) and
\(\mu_{2,m-2}^{m-4}\) are injective.  The second is surjective because
\(\mu_{m-2,2}^0\) and \(R_m\twoheadrightarrow U_{m-2}\) are
surjective.  Finally, \(\dim(X\oplus R_m)=2\dim X\), so the
sequence is exact at its middle term.

The free summand \(R_m\) vanishes in the stable category, and
\eqref{eq:an-short-certificate} induces a distinguished triangle
\begin{equation*}
 \xymatrix@C=2.2pc{
 X\ar[r]^-{\underline\delta}&
 X\ar[r]^-{\underline\delta}&
 X\ar[r]&\Omega_{R_m}^{-1}(X).
 }
\end{equation*}
For every vertex \(s\), the realization identifies \(Pe_s\) with
\(\underline{\Hom}_{R_m}(X,U_s)\), and \(d\) acts by precomposition
with \(\underline\delta\).
Applying \(\underline{\Hom}_{R_m}(-,U_s)\) to this triangle gives
\(\im(d|_{Pe_s})=\ker(d|_{Pe_s})\).
Since \(P=\bigoplus_sPe_s\), this proves \(\im d=\ker d\).

We now verify \(d\circ h=h\circ d\), \(h^2=0\), and the quotient
condition in \eqref{eq:ade-block-certificate}.
Regard maps between the summands of \(X\) as endomorphisms of \(X\)
by extending them by zero on the other summand.
By \cref{fact:a3a-cyclic-homs},
\[
 C=\kk\{\underline{1}_{U_2},\underline{1}_{U_{m-2}},
 \underline t_2,\underline t_{m-2},
 \underline\mu_{m-2,2}^{0},\underline\mu_{m-2,2}^{1},
 \underline\mu_{2,m-2}^{m-4},\underline\mu_{2,m-2}^{m-3}\}.
\]
The only nonzero products of two radical basis elements are
\begin{equation}
\label{eq:an-small-algebra}
\begin{aligned}
 \underline t_2\circ\underline\mu_{m-2,2}^{0}
 &=\underline\mu_{m-2,2}^{0}\circ\underline t_{m-2}
  =\underline\mu_{m-2,2}^{1},\\
 \underline t_{m-2}\circ\underline\mu_{2,m-2}^{m-4}
 &=\underline\mu_{2,m-2}^{m-4}\circ\underline t_2
  =\underline\mu_{2,m-2}^{m-3}.
\end{aligned}
\end{equation}
Thus
\[
 \underline\eta^{\,2}=0,\qquad
 \underline\delta\circ\underline\eta
 =\underline\eta\circ\underline\delta
 =\underline\mu_{2,m-2}^{m-3}.
\]
Transporting these identities along \(\Phi\) gives \(h^2=0\) and
\(d\circ h=h\circ d\).

Put \(\mathfrak r=\operatorname{rad}C\).
For \(\lambda,\mu\in\kk\) and \(z\in\mathfrak r\), the products in
\eqref{eq:an-small-algebra} give
\[
 \underline\delta\circ
 (\lambda\underline{1}_{U_2}+\mu\underline{1}_{U_{m-2}}+z)
 \circ\underline\delta=(\lambda+\mu)
 (\underline\mu_{m-2,2}^{1}+\underline\mu_{2,m-2}^{m-3}).
\]
Hence the numerator of the required quotient is
\(\kk\underline{1}_X\oplus\mathfrak r\).
For its denominator, we have
\[
\begin{aligned}
 \underline\delta\circ\underline{1}_{U_2}
  &=\underline t_2+\underline\mu_{2,m-2}^{m-4},&
 \underline{1}_{U_2}\circ\underline\delta
  &=\underline t_2+\underline\mu_{m-2,2}^{0},\\
 \underline\delta\circ\underline{1}_{U_{m-2}}
  &=\underline t_{m-2}+\underline\mu_{m-2,2}^{0},&
 \underline{1}_{U_{m-2}}\circ\underline\delta
  &=\underline t_{m-2}+\underline\mu_{2,m-2}^{m-4}.
\end{aligned}
\]
These four elements have sum zero, and any three are independent.
Products of \(\underline\delta\) with radical basis elements span
\(\kk\underline\mu_{m-2,2}^{1}\oplus
\kk\underline\mu_{2,m-2}^{m-3}\).
Therefore
\begin{equation}
\label{eq:an-quotient-certificate}
 \underline\delta C+C\underline\delta
 =\kk\{\underline\mu_{m-2,2}^{1},\underline\mu_{2,m-2}^{m-3},
 \underline t_2+\underline\mu_{2,m-2}^{m-4},
 \underline t_{m-2}+\underline\mu_{m-2,2}^{0},
 \underline t_2+\underline\mu_{m-2,2}^{0}\}.
\end{equation}
Modulo this subspace, the classes of \(\underline t_2\),
\(\underline t_{m-2}\), \(\underline\mu_{m-2,2}^{0}\), and
\(\underline\mu_{2,m-2}^{m-4}\) coincide and are nonzero.
The other two radical basis elements have zero class.
It follows that
\[
 \frac{\{z\in C\mid\underline\delta\circ z\circ\underline\delta=0\}}
      {\underline\delta C+C\underline\delta}
 =\kk\{[\underline{1}_X],[\underline\mu_{2,m-2}^{m-4}]\}
 \cong\kk[\varepsilon]/(\varepsilon^2).
\]
Since this quotient is commutative, \(\Phi\) induces the required
algebra isomorphism in \eqref{eq:ade-block-certificate}, with
\([1_P]\mapsto1\) and \([h]\mapsto\varepsilon\).
\end{proof}

\subsection{Types \texorpdfstring{$D_n$ and $E_6,E_7,E_8$}{Dn and E6, E7, E8}}

We shall construct a square-zero endomorphism of a single
indecomposable projective module.  To prove exactness, it suffices
to show that its image has half the dimension of the projective
at every vertex.  The Hom--Ext formula below restricts any possible
failure of this equality to a small number of dimension vectors.
We then exclude them using the preprojective relations.

\subsubsection{Dimension vectors}

\begin{notation}
\label{not:ade-dimension-vectors}
For a \(\Lambda_Q\)-module \(V\), its \emph{dimension vector} is
the column vector
\[
 \mathbf d(V)=(\dim_\kk(Ve_i))_{i\in Q_0}
 \in\mathbb Z_{\geq0}^{Q_0}.
\]
Let \(\mathbf e_i\) be the coordinate vector with entry one at
\(i\) and zero elsewhere.  Thus \(\mathbf d(S_i)=\mathbf e_i\).
The \emph{support} of \(V\) is
\(\operatorname{supp}(V)=\{i\in Q_0\mid Ve_i\ne0\}\),
which we also regard as the full subdiagram on these vertices.
For a vertex \(s\), \(Q\setminus\{s\}\) denotes the full
subdiagram obtained by deleting \(s\).

Write \(C_Q\) for the symmetric Cartan matrix of \(Q\): its
diagonal entries are \(2\), its entries at adjacent vertices are
\(-1\), and all other entries are zero.  The associated bilinear
form is
\[
 (z,w)_Q=z^{\mathsf t}C_Qw
 \qquad(z,w\in\mathbb R^{Q_0}).
\]
Vector inequalities are coordinatewise.  All vectors and Cartan
forms in these inequalities are over \(\mathbb Z\) or \(\mathbb R\),
independently of the characteristic of \(\kk\).
\end{notation}

\begin{fact}
\label{lem:ade-euler}
For \(\Lambda_Q\)-modules \(V,W\), one has
\begin{equation}
\label{eq:ade-hom-ext}
\begin{aligned}
 (\mathbf d(V),\mathbf d(W))_Q
 &=\dim_\kk\Hom_{\Lambda_Q}(V,W)
  +\dim_\kk\Hom_{\Lambda_Q}(W,V)\\
 &\qquad-\dim_\kk\Ext^1_{\Lambda_Q}(V,W).
\end{aligned}
\end{equation}
This is Crawley--Boevey's formula
\cite[Lemma~1]{CrawleyBoevey2000}, which is stated over an arbitrary
field.  The same statement for right modules follows by applying it
to the opposite algebra, whose doubled-quiver presentation is again
that of \(\Lambda_Q\).
\end{fact}

We will use the following consequence of the standard description
of simply laced root systems.  The statement is expressed entirely
in terms of the Cartan matrix.

\begin{lemma}
\label{lem:ade-cartan-uniqueness}
Let \(T\) be a connected simply laced Dynkin diagram.  There is
exactly one nonzero integer vector \(z\) such that
\[
 z\geq0,\qquad C_Tz\geq0,\qquad (z,z)_T\leq2.
\]
This vector has positive entries at every vertex.
If \(T\) is disconnected, every nonzero vector satisfying these
conditions is supported on one connected component, where its
restriction is the vector just described.
\end{lemma}

\begin{proof}
Identify \(\mathbb Z^{T_0}\) with the root lattice of type \(T\)
by taking the simple roots as a basis.  In this basis the Gram
matrix is \(C_T\), and the roots are precisely the lattice vectors
of squared length two
\cite[Chapter~VI, Section~4]{Bourbaki2002}.
Positive definiteness and integrality give
\((z,z)_T\in2\mathbb Z_{>0}\) for \(z\ne0\).
Thus a vector in the statement is a root.  The condition
\(C_Tz\geq0\) says that its inner product with every simple root
is nonnegative, so it is a dominant root.
In a connected simply laced root system, the unique dominant root
is the highest root; its simple-root coefficients are all positive
\cite[Chapter~VI, Section~1, no.~8]{Bourbaki2002}.
This proves the connected case, including existence.

For disconnected \(T\), the Cartan form is the orthogonal sum of
the forms of its connected components.  Each nonzero integral
component contributes at least two to \((z,z)_T\), so at most
one component can be nonzero.  The assertion follows from the
connected case.
\end{proof}

\subsubsection{The projective module and its endomorphisms}

We fix the vertex labels by the following diagrams.  For \(D_n\),
\(n\ge5\), use
\[
\begin{tikzcd}[column sep=1.8em,row sep=1.3em]
 &2\arrow[d,dash]&&&\\
 1\arrow[r,dash]&0\arrow[r,dash]&3\arrow[r,dash]&
 \cdots\arrow[r,dash]&n-1
\end{tikzcd}
\]
For \(E_6,E_7,E_8\), use
\[
\begin{tikzcd}[column sep=1.8em,row sep=1.3em]
 &&1\arrow[d,dash]&&&&\\
 3\arrow[r,dash]&2\arrow[r,dash]&0\arrow[r,dash]&
 4\arrow[r,dash]&5\arrow[r,dash,dashed]&6\arrow[r,dash,dashed]&7
\end{tikzcd}
\]
The solid subdiagram is \(E_6\).  Adjoining vertex \(6\), and then
vertex \(7\), gives \(E_7\) and \(E_8\), respectively.
In type \(D_4\), used in the next subsection, \(0\) is the
central vertex and \(1,2,3\) are the leaves.

The construction uses homogeneous elements of a projective module.
We therefore record the required degree bound and the location of
the socle.

\begin{fact}
\label{fact:ade-graded-structure}
The following graded properties hold for Dynkin preprojective
algebras.
\begin{enumerate}[label=\textup{(\roman*)},leftmargin=2.1em]
\item The largest nonzero path degree \(N_Q\) is given by
\[
\begin{array}{c|ccccc}
 Q&A_m&D_m&E_6&E_7&E_8\\\hline
 N_Q&m-1&2m-4&10&16&28.
\end{array}
\]
In particular, every path of length greater than \(N_Q\) is zero
in \(\Lambda_Q\)
\cite[Corollary~4.3]{BrennerButlerKing2002}.

\item Each \(P_i=e_i\Lambda_Q\) has simple socle concentrated
in degree \(N_Q\).  Write \(\nu\) for the permutation of the
vertices determined by
\[
 \operatorname{soc}(P_i)\cong S_{\nu(i)}.
\]
For the \(D\) and \(E\) diagrams labelled above, \(\nu\) is
the identity in types \(D_{2m},E_7,E_8\).  In type \(D_{2m+1}\),
it interchanges \(1,2\) and fixes the other vertices.  In type
\(E_6\), it interchanges the two arms of length two and fixes
\(0,1\)
\cite[Definition~4.6 and Theorem~4.8]{BrennerButlerKing2002}.
\end{enumerate}
\end{fact}

We shall use an indecomposable projective module with isomorphic
top and socle and a four-dimensional endomorphism algebra.
The following choice has both properties.

\begin{notation}
\label{not:ade-chosen-projective}
For the remainder of this subsection, let \(Q=D_n\), \(n\ge5\),
or \(Q=E_6,E_7,E_8\).  In type \(D_n\), take \(s=n-2\),
the neighbour of the end vertex on the long arm.  In types
\(E_6,E_7,E_8\), take \(s\) to be the end vertex of the arm of
length \(1,2,4\), respectively.  Thus
\[
 s=1\ (E_6),\qquad s=3\ (E_7),\qquad s=7\ (E_8).
\]
Put
\[
 \Lambda=\Lambda_Q,\qquad e=e_s,\qquad P=e\Lambda,
 \qquad\Gamma=e\Lambda e,\qquad N=N_Q.
\]
We identify \(\Gamma\) with \(\End_\Lambda(P)\) by left
multiplication, as in \cref{not:ade-modules}.
\end{notation}

The chosen vertex \(s\) is fixed by the permutation in
\cref{fact:ade-graded-structure}.  Consequently
\[
 \operatorname{top}(P)\cong\operatorname{soc}(P)\cong S_s,
 \qquad\operatorname{soc}(P)=e\Lambda_Ne.
\]
Apply \eqref{eq:ade-hom-ext} to \(P\) and each simple module
\(S_i\).  Since \(P\) is projective and has the displayed top
and socle, this gives
\begin{equation}
\label{eq:ade-cartan-projective}
 C_Q\mathbf d(P)=2\mathbf e_s.
\end{equation}
To solve this equation, let \(\theta\) have entry one at the
three leaves and entry two elsewhere in type \(D_n\).  In the
exceptional types, let
\begin{equation}
\label{eq:ade-projective-dimensions}
\begin{array}{c|c}
 Q&\theta\rule{0pt}{2.4ex}\\\hline
 E_6&(3,2,2,1,2,1)\\
 E_7&(4,2,3,2,3,2,1)\\
 E_8&(6,3,4,2,5,4,3,2),
\end{array}
\end{equation}
where the coordinates are indexed by the vertices in increasing label
order.
In each case, twice the entry at a vertex minus the sum of the
neighbouring entries gives
\(C_Q\theta=\mathbf e_s\), and \(\theta_s=2\).
Since \(C_Q\) is invertible, \eqref{eq:ade-cartan-projective} yields
\begin{equation}
\label{eq:ade-half-projective}
 \mathbf d(P)=2\theta,\qquad\dim_\kk\Gamma=4.
\end{equation}

We now specify a homogeneous element of \(\Gamma\).
In type \(D_n\), let \(t=n-1\) and let \(x:s\to t\) be
the arrow to the end vertex.  Define
\[
 a=x\bar x,\qquad\alpha=2.
\]
In types \(E_6,E_7,E_8\), let \(\gamma:s\to0\) be the
path along the chosen arm, without repeated vertices.  Its length
is \(1,2,4\), respectively.  Let \(y_1,y_2\) be the arrows
starting at \(0\) along the other two arms, and let \(y_3\)
be the arrow from \(0\) towards \(s\).  Define
\[
 p_i=\gamma y_i\bar y_i\bar\gamma\quad(i=1,2),
 \qquad a=p_1,\qquad\alpha=2\operatorname{length}(\gamma)+2.
\]
Thus \(p_i\) travels from \(s\) to \(0\), crosses one edge
of another arm, and returns along the same route.
In all four types, put \(\beta=N-\alpha\).  These are the
degrees in which we will find the two generators of \(\Gamma\):
\begin{equation}
\label{eq:trivalent-degree-data}
\begin{array}{c|cccc}
 Q&D_n&E_6&E_7&E_8\\\hline
 \alpha&2&4&6&10\\
 \beta&2n-6&6&10&18.
\end{array}
\end{equation}
In particular, \(0<\alpha<\beta\) and \(\alpha+\beta=N\).

\begin{lemma}
\label{lem:trivalent-local-algebra}
There exists \(b\in e\Lambda_\beta e\) such that
\[
 \Gamma=\kk\{e,a,b,ab\},\qquad
 a^2=b^2=0,\qquad ab=ba\ne0.
\]
\end{lemma}

\begin{proof}
We first verify that \(a\ne0\).  In type \(D_n\), the two
length-two paths from \(s\) to itself have only their sum as a
relation.  Hence \(a=x\bar x\ne0\).  The relation at the leaf
\(t\) is \(\bar x x=0\), and therefore \(a^2=0\).

For an exceptional type, consider the vector space
\[
 V_\alpha=e(\kk\overline Q)_\alpha e
\]
with basis the length-\(\alpha\) paths from \(s\) to itself,
before imposing the relations.  Define a linear functional
\(\varphi:V_\alpha\to\kk\) to have value one on
\(p_1,p_2\) and zero on every other basis path.
We claim that \(\varphi\) vanishes on \(eI_Qe\cap V_\alpha\).
This subspace is spanned by products \(\xi r_i\zeta\), where
\(\xi,\zeta\) are paths of total length \(\alpha-2\) and the
product starts and ends at \(s\).

The only consecutive opposite arrows in \(p_j\) are
\(y_j\bar y_j\).  Thus \(p_j\) can occur in
\(\xi r_i\zeta\) only when \(i=0\),
\(\xi=\gamma\), and \(\zeta=\bar\gamma\).
For this product, the relation
\(r_0=y_1\bar y_1+y_2\bar y_2+y_3\bar y_3\) gives
\[
 \varphi(\gamma r_0\bar\gamma)=1+1+0=0.
\]
Every other product \(\xi r_i\zeta\) contains neither
\(p_1\) nor \(p_2\), so its image under \(\varphi\) is also
zero.  The functional therefore induces a linear functional on
\(e\Lambda_\alpha e\).  Its value on \(a=p_1\) is one,
proving \(a\ne0\).

The nonzero submodule \(a\Lambda\) of \(P\) contains a simple
submodule.  Since \(\operatorname{soc}(P)\) is simple, it follows
that \(\operatorname{soc}(P)\subseteq a\Lambda\).
Choose \(0\ne c\in e\Lambda_Ne\) and write \(c=aw\).
Let \(w_\beta\) be the degree-\(\beta\) component of \(w\).
Since \(a\) has degree \(\alpha\), the element
\(b=ew_\beta e\) satisfies \(ab=c\ne0\).
The elements \(e,a,b,ab\) have distinct degrees
\(0,\alpha,\beta,N\).  By \eqref{eq:ade-half-projective},
they form a basis of \(\Gamma\).
For an exceptional type, the values in
\eqref{eq:trivalent-degree-data} give
\(2\alpha\notin\{0,\alpha,\beta,N\}\), so \(a^2=0\).
In every type \(2\beta>N\), so \(b^2=0\).

Path reversal \(\gamma\mapsto\bar\gamma\) preserves the
defining relations and induces a graded anti-involution \(\iota\)
of \(\Lambda\).  It preserves \(\Gamma\), each of whose nonzero
homogeneous components is one-dimensional.  On such a component,
\(\iota\) acts by a scalar \(\xi\) satisfying \(\xi^2=1\).
As \(\kk\) has characteristic two, \(\xi=1\).
Thus \(\iota\) fixes every element of \(\Gamma\), and
\(ab=\iota(ab)=\iota(b)\iota(a)=ba\).
\end{proof}

\subsubsection{Exactness}

The following lemma bounds the difference between half the dimension
vector of a projective module and the dimension vector of an image.
It will also be used for \(D_4\).

\begin{lemma}
\label{lem:ade-image-dimensions}
Let \(P=e_s\Lambda_Q\) have simple top and socle \(S_s\).
Suppose that
\[
 \mathbf d(P)=2\theta,\qquad C_Q\theta=\mathbf e_s,
 \qquad\theta_s=2,
\]
and that its endomorphism algebra has the form
\[
 \Gamma=e_s\Lambda_Qe_s=\kk\{e_s,u,v,uv\},
 \qquad u^2=v^2=0,\qquad uv=vu\ne0.
\]
For \(U=uP\) and \(z=\theta-\mathbf d(U)\), one has
\begin{equation}
\label{eq:ade-image-bound}
 z\geq0,\qquad z_s=0,\qquad
 (z,z)_Q\leq2,\qquad (C_Qz)_i\geq0\quad(i\ne s).
\end{equation}
If \(z\ne0\), it is supported on one connected component of
\(Q\setminus\{s\}\), and its restriction there is the vector
of \cref{lem:ade-cartan-uniqueness}.
\end{lemma}

\begin{proof}
Since \(u^2=0\), the image \(U\) is contained in \(\ker L_u\).
Applying rank--nullity at each vertex gives
\(\mathbf d(U)\leq\frac12\mathbf d(P)=\theta\), and hence
\(z\geq0\).  Moreover,
\(Ue_s=u\Gamma=\kk u\oplus\kk uv\), so \(z_s=0\).

We compute the ordinary endomorphism algebra of \(U\).
Commutativity of \(\Gamma\) implies that every endomorphism of
\(P\) preserves \(U=uP\).  Conversely, any endomorphism of
\(U\), viewed as a map \(U\to P\), extends to \(P\) because
\(P\) is injective.  Thus restriction
\(\Gamma\to\End_{\Lambda_Q}(U)\) is surjective.
Its kernel consists of the elements \(c\in\Gamma\) with
\(cu=0\), since \(u\) generates \(U\).  The displayed basis
of \(\Gamma\) gives
\(\operatorname{ann}_\Gamma(u)=\kk u\oplus\kk uv=u\Gamma\).
Consequently
\[
 \End_{\Lambda_Q}(U)\cong\Gamma/u\Gamma
 \cong\kk[\varepsilon]/(\varepsilon^2),
 \qquad\dim_\kk\End_{\Lambda_Q}(U)=2.
\]

As a nonzero quotient of \(P\), the module \(U\) has simple top
\(S_s\).  As a nonzero submodule of \(P\), it has simple socle
\(S_s\).  Since \((\theta,z)_Q=z_s=0\) and
\((\theta,\theta)_Q=\theta_s=2\), the Hom--Ext formula gives
\[
\begin{aligned}
 2+(z,z)_Q
 &=(\mathbf d(U),\mathbf d(U))_Q\\
 &=2\dim_\kk\End_{\Lambda_Q}(U)
   -\dim_\kk\Ext^1_{\Lambda_Q}(U,U)
 \leq4.
\end{aligned}
\]
This proves \((z,z)_Q\leq2\).
For \(i\ne s\), the top and socle of \(U\) give
\(\Hom(U,S_i)=0=\Hom(S_i,U)\).  Applying
\eqref{eq:ade-hom-ext} to \(U,S_i\), we obtain
\[
 \bigl(C_Q\mathbf d(U)\bigr)_i
 =-\dim_\kk\Ext^1_{\Lambda_Q}(U,S_i)\leq0.
\]
Since \(C_Q\theta=\mathbf e_s\), this is equivalent to
\((C_Qz)_i\geq0\).
Finally, deleting the zero coordinate \(z_s\) identifies these
conditions with those for the Cartan matrix of
\(Q\setminus\{s\}\).  The last assertion follows from
\cref{lem:ade-cartan-uniqueness}.
\end{proof}

Fix \(b\) as in \cref{lem:trivalent-local-algebra}.

\begin{proposition}
\label{prop:trivalent-period-one}
For \(Q=D_n\), \(n\ge5\), or \(Q=E_6,E_7,E_8\), the maps
\[
 d=L_{a+b},\qquad h=L_a
\]
satisfy \eqref{eq:ade-block-certificate}.
\end{proposition}

\begin{proof}
Put \(u=a+b\), \(U=uP\), and \(z=\theta-\mathbf d(U)\).
The relations in \(\Gamma\) give \(u^2=a^2=0\) and
\(ua=au=ab\ne0\), and \(\{e,u,a,ua\}\) is a basis.
Thus \cref{lem:ade-image-dimensions} applies with \(v=a\).
We prove \(z=0\) by excluding the possible supports of a nonzero
\(z\).

We use one consequence of support throughout the argument.
If a module is supported on a full subdiagram \(T\), its action
factors through the graded quotient
\(\Lambda_Q\twoheadrightarrow\Pi_\kk(T)\) obtained by setting
the omitted vertex idempotents equal to zero.  Indeed, all paths
through those vertices act as zero, and the surviving defining
relations are precisely those of \(\Pi_\kk(T)\).
Hence any element of degree greater than \(N_T\) annihilates
such a module.

Suppose first that \(Q\) is exceptional and \(z\ne0\).
The diagram \(Q\setminus\{s\}\) is connected.
The unique possible vector \(z\) and the resulting support of
\(U\), computed from \(\mathbf d(U)=\theta-z\), are
\begin{equation}
\label{eq:ade-possible-supports}
\begin{array}{c|c|c|c}
 Q&Q\setminus\{s\}&z&\operatorname{supp}(U)\rule{0pt}{2.4ex}\\\hline
 E_6&A_5&(1,0,1,1,1,1)&\{0,1,2,4\}\ (D_4)\\
 E_7&D_6&(2,1,1,0,2,2,1)&\{0,1,2,3,4\}\ (D_5)\\
 E_8&E_7&(4,2,3,2,3,2,1,0)&\{0,1,2,4,5,6,7\}\ (D_7).
\end{array}
\end{equation}
Here the coordinates of \(z\) retain the original vertex order,
with zero at \(s\).  On the deleted diagram, each displayed vector
has squared length two and its product with the Cartan matrix is
nonnegative.  Thus \cref{lem:ade-cartan-uniqueness} identifies it
as the only possibility.  Subtraction from
\eqref{eq:ade-projective-dimensions} gives the last column.

For \(T=D_4,D_5,D_7\), respectively,
\cref{fact:ade-graded-structure} gives \(N_T=4,6,10\).
These are the respective values of \(\alpha\), whereas \(b\)
has degree \(\beta>\alpha\).
Thus the image of \(b\) in \(\Pi_\kk(T)\) is zero, so \(b\) acts
trivially on \(U\).  In particular, \(ub=0\).  On the other hand,
\[
 ub=(a+b)b=ab\ne0,
\]
a contradiction.  Hence \(z=0\) in the exceptional types.

Now let \(Q=D_n\) and suppose \(z\ne0\).
Deleting \(s=n-2\) leaves two components: \(D_{n-2}\) and the
isolated vertex \(t=n-1\).  When \(n=5\), the first component
is \(A_3\).
If \(z\) is supported at \(t\), then
\(z=\mathbf e_t\), so \(Ue_t=0\).
Since \(x\) ends at \(t\), the element \(ux\) belongs to \(Ue_t\).
The leaf relation gives \(ax=0\), so \(ux=bx\).
This element is nonzero, because
\[
 (ux)\bar x=bx\bar x=ba=ab\ne0.
\]
This contradicts \(Ue_t=0\).

If \(z\) is supported on the other component, its entries are
one at the leaves and two elsewhere; for \(A_3\), all three
entries are one.  These vectors satisfy the conditions of
\cref{lem:ade-cartan-uniqueness} and so exhaust this case.
Subtracting from \(\theta\) leaves support on the three-vertex
path formed by \(s\) and its two neighbours, with dimension
vector \((1,2,1)\).  Thus \(U\) is supported on \(A_3\),
where the largest path degree is two.
Since \(\deg b=\beta=2n-6>2\), the image of \(b\) in
\(\Pi_\kk(A_3)\) is zero.  Hence \(b\) acts trivially on \(U\),
so \(ub=0\).  But \(ub=(a+b)b=ab\ne0\), a contradiction.
Thus \(z=0\) also in type \(D_n\).

We have proved \(z=0\), so \(\mathbf d(U)=\theta\).
As \(\mathbf d(P)=2\theta\), the inclusion
\(\im d\subseteq\ker d\) is an equality.
The relations already give \(d\circ h=h\circ d\) and \(h^2=0\).
Finally, commutativity and \(u^2=0\) give
\[
 \frac{\{c\in\Gamma\mid ucu=0\}}{u\Gamma+\Gamma u}
 =\Gamma/u\Gamma
 =\kk\{[e],[a]\}
 \cong\kk[\varepsilon]/(\varepsilon^2),
\]
where \([e]\) maps to \(1\) and \([a]\) maps to
\(\varepsilon\).  This verifies the remaining condition in
\eqref{eq:ade-block-certificate}.
\end{proof}
\subsection{Type \texorpdfstring{$D_4$}{D4}}

Assume \(|\kk|>2\), and choose \(\omega\in\kk\setminus\{0,1\}\).
Number the central vertex by \(0\)
and the leaves by \(1,2,3\).  Let \(x_i:0\to i\) and \(\bar x_i\)
be the opposite arrows, and put
\[
 \Lambda=\Pi_\kk(D_4),\qquad P=e_0\Lambda,\qquad
 \Gamma=e_0\Lambda e_0,\qquad
 a=x_1\bar x_1,\qquad b=x_2\bar x_2.
\]
The relations give
\[
 x_3\bar x_3=a+b,\qquad \bar x_i x_i=0\quad(1\le i\le3),
\]
so \(a^2=b^2=(a+b)^2=0\), and \(ab=ba\).
By \cref{fact:ade-graded-structure}, the top and socle of \(P\)
are both \(S_0\), and the socle is concentrated in degree four.
The vector \(\theta=(2,1,1,1)\) satisfies
\(C_{D_4}\theta=\mathbf e_0\).
Applying \eqref{eq:ade-hom-ext} to \(P,S_i\), as in
\eqref{eq:ade-cartan-projective}, therefore gives
\(\mathbf d(P)=2\theta\) and \(\dim_\kk\Gamma=4\).
The three length-two paths from \(0\) to itself have exactly their sum as a
relation, so \(a,b\) are independent.  The nonzero submodule
\(a\Lambda\) contains the socle.  Its degree-four component is
\(a(\kk a+\kk b)=\kk ab\), so \(ab\ne0\).  Therefore
\[
 e_0\Lambda e_0=\kk\{e_0,a,b,ab\},
 \qquad a^2=b^2=0,\qquad ab=ba\ne0.
\]

\begin{proposition}
\label{prop:d4-period-one}
The maps \(d=L_{a+\omega b}\) and \(h=L_a\) satisfy
\eqref{eq:ade-block-certificate}.
\end{proposition}

\begin{proof}
Put \(u=a+\omega b\) and \(U=uP\).  Since \(\omega\ne0\),
\[
 \Gamma=\kk\{e_0,u,a,ua\},\qquad
 u^2=a^2=0,\qquad ua=au=\omega ab\ne0.
\]
Thus \cref{lem:ade-image-dimensions} applies with \(s=0\) and \(v=a\).
Deleting \(0\) leaves three isolated vertices.
If \(\theta-\mathbf d(U)\ne0\),
\cref{lem:ade-cartan-uniqueness} gives
\(\theta-\mathbf d(U)=\mathbf e_i\) for some leaf \(i\).
This would give \(Ue_i=0\).
However, all three elements \(ux_i\in Ue_i\) are nonzero, since
\[
 (ux_1)\bar x_1=\omega ab,\qquad
 (ux_2)\bar x_2=ab,\qquad
 (ux_3)\bar x_3=(1+\omega)ab.
\]
Thus \(\mathbf d(U)=\theta\), proving \(\im d=\ker d\).
The displayed relations give \(h^2=0\) and
\(d\circ h=L_{ua}=L_{au}=h\circ d\).
Finally, commutativity of \(\Gamma\) and \(u^2=0\) give
\[
 \frac{\{c\in\Gamma\mid ucu=0\}}{u\Gamma+\Gamma u}
 =\Gamma/u\Gamma
 =\kk\{[e_0],[a]\}
 \cong\kk[\varepsilon]/(\varepsilon^2),
\]
where \([e_0]\) maps to \(1\) and \([a]\) maps to \(\varepsilon\).
This verifies the remaining condition in \eqref{eq:ade-block-certificate}.
\end{proof}

\begin{theorem}
\label{thm:ade-explicit-lines}
Let \(\kk\) have characteristic two and let
\[
 Q\in\{A_n\mid n\ge5\}\cup\{D_n\mid n\ge5\}
       \cup\{E_6,E_7,E_8\}.
\]
If \(|\kk|>2\), one may also take \(Q=D_4\).  Put
\(\mathcal F_Q=\mathcal P(\Pi_\kk(Q))\).  Then \(\mathcal F_Q\) is a
Hom-finite idempotent-complete \(\kk\)-linear category carrying pairwise
distinct pre-triangulations
\(\{\Delta^Q_{\kk,\lambda}\}_{\lambda\in\kk}\) such that
\[
 \Delta^Q_{\kk,\lambda}\text{ satisfies }\mathrm{TR4}
 \quad\Longleftrightarrow\quad
 \lambda=0.
\]
\end{theorem}

\begin{proof}
The case \(A_5\) is \cref{thm:scalar-family} with \(a=2\).
For the remaining types, apply \cref{prop:strict-period-one-detector}
to \cref{prop:an-period-one,prop:trivalent-period-one,prop:d4-period-one}.
\end{proof}

\medskip
\noindent\textbf{Acknowledgements.}
{\itshape\emergencystretch=2em
The authors developed the main ideas, central concepts, and overall proof
strategies, and established the results with assistance from GPT-5.6 Sol.
They checked the mathematical arguments and revised the manuscript, and
take full responsibility for its content and mathematical correctness.
The authors would like to thank Amnon Neeman for helpful discussions and
suggestions.

This work is supported by the National Natural Science Foundation of China
(Nos.~1250011863 and~12371034).
\par}


\begin{thebibliography}{CLLZ26}
\setlength{\itemsep}{0pt}

\bibitem[Ami07]{Amiot2007}
C.~Amiot,
\emph{On the structure of triangulated categories with finitely many
indecomposables},
Bull. Soc. Math. France \textbf{135} (2007), no.~3, 435--474.

\bibitem[Ano26]{AnonymousD4}
Anonymous,
\emph{A computer-assisted mapping-cone obstruction for the preprojective
algebra of type \(D_4\)},
unpublished manuscript, Zenodo (2026),
\url{https://doi.org/10.5281/zenodo.21889660}.

\bibitem[Bal11]{Balmer2011}
P.~Balmer,
\emph{Separability and triangulated categories},
Adv. Math. \textbf{226} (2011), 4352--4372.

\bibitem[Bel00]{Beligiannis2000}
A.~Beligiannis,
\emph{Relative homological algebra and purity in triangulated categories},
J. Algebra \textbf{227} (2000), no.~1, 268--361.

\bibitem[Bou02]{Bourbaki2002}
N.~Bourbaki,
\emph{Lie groups and Lie algebras. Chapters 4--6},
Elements of Mathematics, Springer-Verlag, Berlin, 2002.

\bibitem[BBK02]{BrennerButlerKing2002}
S.~Brenner, M.~C. R.~Butler, and A.~D. King,
\emph{Periodic algebras which are almost Koszul},
Algebr. Represent. Theory \textbf{5} (2002), no.~4, 331--368,
\url{https://doi.org/10.1023/A:1020146502185}.

\bibitem[CLLZ26]{ChenLiuLuZhang2026}
X.-W.~Chen, J.~Liu, X.-S.~Lu, and C.~Zhang,
\emph{A pre-triangulated category which is not triangulated},
Preprint, arXiv:2608.09777 (2026).

\bibitem[CB00]{CrawleyBoevey2000}
W.~Crawley-Boevey,
\emph{On the exceptional fibres of Kleinian singularities},
Amer. J. Math. \textbf{122} (2000), no.~5, 1027--1037.

\bibitem[DCM26]{DiazCabreraMuro2026}
J.~D\'iaz Cabrera and F.~Muro,
\emph{An exotic finite pretriangulated category over any algebraically closed field},
Preprint, arXiv:2608.15203v3 (2026).

\bibitem[DR92]{DlabRingel1992}
V.~Dlab and C.~M. Ringel,
\emph{The module theoretical approach to quasi-hereditary algebras},
in: H.~Tachikawa and S.~Brenner (eds.), Representations of Algebras and
Related Topics, London Math. Soc. Lecture Note Ser., vol.~168,
Cambridge Univ. Press, Cambridge, 1992, 200--224.

\bibitem[Ela15]{Elagin2014}
A.~Elagin,
\emph{On equivariant triangulated categories},
Preprint, arXiv:1403.7027v2 (2015).

\bibitem[Fre66]{Freyd1966}
P.~Freyd,
\emph{Representations in abelian categories},
in: Proceedings of the Conference on Categorical Algebra, La Jolla 1965,
Springer, Berlin--Heidelberg, 1966, 95--120.

\bibitem[GKO13]{GKO2013}
C.~Geiss, B.~Keller, and S.~Oppermann,
\emph{$n$-angulated categories},
J. Reine Angew. Math. \textbf{675} (2013), 101--120.

\bibitem[Hap88]{Happel1988}
D.~Happel,
\emph{Triangulated categories in the representation theory of
finite-dimensional algebras},
London Math. Soc. Lecture Note Ser., vol.~119, Cambridge Univ. Press, 1988.

\bibitem[Hel68]{Heller1968}
A.~Heller,
\emph{Stable homotopy categories},
Bull. Amer. Math. Soc. \textbf{74} (1968), no.~1, 28--63.

\bibitem[Kel05]{Keller2005}
B.~Keller,
\emph{On triangulated orbit categories},
Doc. Math. \textbf{10} (2005), 551--581.

\bibitem[Kel09]{KellerCorrection}
B.~Keller,
\emph{Corrections to ``On triangulated orbit categories''},
author's note, last revised 2009,
\url{https://webusers.imj-prg.fr/~bernhard.keller/publ/corrTriaOrbit.pdf}.

\bibitem[KY11]{KrauseYe2011}
H.~Krause and Y.~Ye,
\emph{On the centre of a triangulated category},
Proc. Edinb. Math. Soc. (2) \textbf{54} (2011), no.~2, 443--466,
\url{https://doi.org/10.1017/S0013091509001199}.

\bibitem[MM15]{MarkmanMehrotra2015}
E.~Markman and S.~Mehrotra,
\emph{Integral transforms and deformations of K3 surfaces},
Preprint, arXiv:1507.03108v1 (2015).

\bibitem[MMV19]{MarkmanMehrotraVerbitsky2019}
E.~Markman, S.~Mehrotra, and M.~Verbitsky,
\emph{Rigid hyperholomorphic sheaves remain rigid along twistor deformations
	of the underlying hypark\"ahler manifold},
Eur. J. Math. \textbf{5} (2019), 964--1012.

\bibitem[MSS07]{MuroSchwedeStrickland2007}
F.~Muro, S.~Schwede, and N.~Strickland,
\emph{Triangulated categories without models},
Invent. Math. \textbf{170} (2007), no.~2, 231--241.

\bibitem[Nee91]{Neeman1991}
A.~Neeman,
\emph{Some new axioms for triangulated categories},
J. Algebra \textbf{139} (1991), no.~1, 221--255.

\bibitem[Nee01]{Neeman2001}
A.~Neeman,
\emph{Triangulated categories},
Ann. of Math. Stud., vol.~148, Princeton Univ. Press, 2001.

\bibitem[Pos21]{Posur2021}
S.~Posur,
\emph{A constructive approach to Freyd categories},
Appl. Categ. Structures \textbf{29} (2021), 171--211.

\bibitem[Pup67]{Puppe1967}
D.~Puppe,
\emph{Stabile Homotopietheorie I},
Math. Ann. \textbf{169} (1967), 243--274.

\bibitem[RVdB02]{ReitenVanDenBergh2002}
I.~Reiten and M.~Van den Bergh,
\emph{Noetherian hereditary abelian categories satisfying Serre duality},
J. Amer. Math. Soc. \textbf{15} (2002), no.~2, 295--366.

\bibitem[RVdB20]{RizzardoVandenBergh2020}
A.~Rizzardo and M.~Van den Bergh,
\emph{A \(k\)-linear triangulated category without a model},
Ann. of Math. (2) \textbf{191} (2020), no.~2, 393--437.

\bibitem[Sun19]{Sun2019}
C.~Sun,
\emph{A note on equivariantization of additive categories and triangulated
categories},
J. Algebra \textbf{534} (2019), 483--530.

\bibitem[Wei94]{Weibel1994}
C.~A. Weibel,
\emph{An introduction to homological algebra},
Cambridge Stud. Adv. Math., vol.~38,
Cambridge Univ. Press, Cambridge, 1994.
\end{thebibliography}
\end{document}